\documentclass[final,onefignum,onetabnum]{siamonline190516}

\usepackage{amsopn}
\usepackage{amsfonts}
\usepackage{amsmath,amsbsy,amsgen,amscd,amssymb,bm}
\usepackage{mathtools}
\usepackage{dsfont}
\usepackage{stmaryrd}
\usepackage{pifont}

\usepackage[scaled=1.2]{urwchancal}

\usepackage{graphicx}
\graphicspath{{figures/}}

\usepackage{epstopdf}

\ifpdf
  \DeclareGraphicsExtensions{.eps,.pdf,.png,.jpg}
\else
  \DeclareGraphicsExtensions{.eps}
\fi

\definecolor{vdkgray}{gray}{0.3}
\definecolor{dkgray}{rgb}{.4,.4,.4}
\definecolor{dkblue}{rgb}{0,0,.5}
\definecolor{ltblue}{rgb}{0.9,0.9,1}
\definecolor{medblue}{rgb}{0,0,.75}
\definecolor{rust}{rgb}{0.5,0.1,0.1}
\definecolor{paleyellow}{rgb}{1,1,0.9}

\usepackage[caption=false]{subfig}

\usepackage{enumitem}
\setlist[enumerate]{leftmargin=.5in}
\setlist[itemize]{leftmargin=.5in}

\usepackage[noend]{myalgpseudocode}
\usepackage{algorithmicx}

\newsiamthm{infthm}{Informal Theorem}
\crefname{infthm}{Informal Theorem}{Informal Theorems}

\numberwithin{equation}{section}
\numberwithin{theorem}{section}
\numberwithin{figure}{section}

\newsiamthm{claim}{Claim}
\crefname{claim}{Claim}{Claims}

\newsiamthm{fact}{Fact}
\crefname{fact}{Fact}{Facts}

\newsiamremark{remark}{Remark}
\crefname{remark}{Remark}{Remarks}

\newsiamremark{warning}{Warning}
\crefname{warning}{Warning}{Warnings}

\newsiamremark{example}{Example}
\crefname{example}{Example}{Examples}

\usepackage{csvsimple}
\usepackage{booktabs, makecell, longtable}
\usepackage{numprint}
\usepackage{siunitx}
\usepackage[normalem]{ulem}

\newcommand{\R}{\mathbb{R}}
\newcommand{\C}{\mathbb{C}}
\newcommand{\F}{\mathbb{F}}

\newcommand{\Sym}{\mathbb{S}}
\newcommand{\real}{\operatorname{Re}}

\newcommand{\eps}{\varepsilon}
\newcommand{\econst}{\mathrm{e}}
\newcommand{\vct}[1]{\bm{#1}}
\newcommand{\mtx}[1]{\bm{#1}}

\newcommand{\Id}{\mathbf{I}}
\newcommand{\trace}{\operatorname{tr}}
\newcommand{\rank}{\operatorname{rank}}
\newcommand{\range}{\operatorname{range}}
\newcommand{\diag}{\operatorname{diag}}
\newcommand{\dist}{\operatorname{dist}}
\newcommand{\abs}[1]{\vert #1 \vert}
\newcommand{\norm}[1]{\Vert #1 \Vert}
\newcommand{\normsq}[1]{\norm{#1}^2}
\newcommand{\fnorm}[1]{\norm{#1}_{\mathrm{F}}}
\newcommand{\fnormsq}[1]{\fnorm{#1}^2}
\newcommand{\triplenorm}[1]{\vert\!\vert\!\vert #1 \vert\!\vert\!\vert}
\newcommand{\ip}[2]{\langle #1, \, #2 \rangle}
\newcommand{\lowrank}[2]{\llbracket {#1} \rrbracket_{#2}}
\newcommand{\psdle}{\preccurlyeq}
\newcommand{\psdge}{\succcurlyeq}
\newcommand{\Expect}{\operatorname{\mathbb{E}}}

\newcommand{\minimize}{\text{minimize}}
\newcommand{\maximize}{\text{maximize}}
\newcommand{\subjto}{\text{subject to}}
\newcommand{\argmin}{\operatorname{arg\,min}}

\newcommand{\grad}{\nabla}
\newcommand{\sgn}{\operatorname{sgn}}

\usepackage{xspace}
\newcommand{\CGAL}{\textsf{CGAL}\xspace}
\newcommand{\sCGAL}{\textsf{SketchyCGAL}\xspace}
\newcommand{\primone}{\text{\textcolor{dkgray}{\ding{202}}}}
\newcommand{\primtwo}{\text{\textcolor{dkgray}{\ding{203}}}}
\newcommand{\primthree}{\text{\textcolor{dkgray}{\ding{204}}}}

\headers{Scalable Semidefinite Programming}{Yurtsever, Tropp, Fercoq, Udell, and Cevher}

\title{Scalable Semidefinite Programming\thanks{Submitted to the editors 6 December 2019. Revised on 18 January 2020, 24 July 2020 and 12 November 2020. 
\funding{VC and AY have received funding from the European Research Council (ERC) under the European Union’s Horizon 2020 research and innovation program under the grant agreement number 725594 (time-data) and the Swiss National Science Foundation (SNSF) under the grant number 200021\_178865/1. 
JAT gratefully acknowledges ONR Awards N00014-11-1-0025, N00014-17-1-2146, and N00014-18-1-2363.
MU gratefully acknowledges DARPA Award FA8750-17-2-0101. 
{Part of this research is conducted while AY is at Massachusetts Institute of Technology, Cambridge, MA, USA. 
AY acknowledges the Early Postdoc.Mobility Fellowship P2ELP2\_187955 from the Swiss National Science Foundation and partial postdoctoral support from the NSF-CAREER grant IIS-1846088.}}}}

\author{
Alp Yurtsever\thanks{Laboratory for Information and Inference Systems, Department of Electrical Engineering, {\'E}cole Polytechnique F{\'e}d{\'e}rale de Lausanne, Lausanne, Switzerland
  (\email{alp.yurtsever@epfl.ch}, \email{volkan.cevher@epfl.ch}, \url{https://lions.epfl.ch/}).}
\and Joel A.~Tropp\thanks{Department of Computing and Mathematical Sciences, California Institute of Technology, Pasadena, CA, USA
  (\email{jtropp@cms.caltech.edu}, \url{http://users.cms.caltech.edu/\string~jtropp}).}
\and Olivier Fercoq\thanks{Laboratoire Traitement et Communication d'Information, T{\'e}l{\'e}com Paris, Institut Polytechnique de Paris, Palaiseau, France 
  (\email{olivier.fercoq@telecom-paris.fr}, \url{https://perso.telecom-paristech.fr/ofercoq/}.)}
\and Madeleine Udell\thanks{Department of Operations Research and Information Engineering, Cornell University, Ithaca, NY, USA
  (\email{udell@cornell.edu}, \url{https://people.orie.cornell.edu/mru8/}).}
\and Volkan Cevher\footnotemark[2]}

\ifpdf
\hypersetup{
  pdftitle={Scalable semidefinite programming},
  pdfauthor={A.~Yurtsever, J.~A.~Tropp, O.~Fercoq, M.~Udell, and V.~Cevher}
}
\fi

\begin{document}

\maketitle

\begin{abstract}
Semidefinite programming (SDP) is a powerful framework from convex optimization
that has striking potential for data science applications.
This paper develops a provably correct randomized algorithm
for solving large, weakly constrained SDP problems by economizing on the storage
and arithmetic costs.  Numerical evidence shows that the method
is effective for a range of applications,
including relaxations of \textsf{MaxCut}, abstract phase retrieval,
and quadratic assignment.
Running on a laptop equivalent, the algorithm
can handle SDP instances where the
matrix variable has over $10^{14}$ entries.
\end{abstract}



\section{Motivation}

For a spectrum of challenges in data science, methodologies
based on semidefinite programming offer remarkable
performance both in theory and for small problem instances.
Even so, practitioners often critique this approach
by asserting that it is impossible to solve
semidefinite programs (SDPs) at the scale
demanded by real-world applications.
We would like to argue against this article of conventional wisdom.

This paper proposes a new algorithm, called \sCGAL,
that can solve very large SDPs to moderate accuracy.
The algorithm marries a primal--dual
optimization technique~\cite{YFC19:Conditional-Gradient-Based} to a randomized
sketch 
for low-rank matrix approximation~\cite{TYUC17:Fixed-Rank-Approximation}.
In each iteration, the primary expense is one low-precision
randomized eigenvector calculation~\cite{KW92:Estimating-Largest}.

For every standard-form SDP that satisfies strong duality,
\sCGAL\ provably converges to a near-optimal low-rank approximation
of a solution.  The algorithm uses limited arithmetic and minimal storage.
It is most effective for weakly constrained problems
whose solutions are nearly low-rank.
In contrast, given the same computational resources, 
other methods for this class of problems 
may fail. 
In particular, \sCGAL\ needs far less storage than the
Burer--Monteiro factorization heuristic~\cite{Burer2003,JMLR:v15:boumal14a}
for 
certain problem instances~\cite{WW18:RankOptimality}.

In addition to the theoretical guarantees,
we offer evidence
that \sCGAL\ is a practical optimization algorithm.
For example, on a laptop equivalent, we can 
solve the \textsf{MaxCut} SDP for
a sparse graph with over 20 million vertices,
where the matrix variable %
has over $10^{14}$ entries.
We also tackle large phase retrieval
problems arising from Fourier ptychography~\cite{Horstmeyer_2015},
as well as relaxations~\cite{Zhao1998,BravoFerreira2018} of the quadratic
assignment problem.

\subsection{Example: The maximum cut in a graph}

To begin, we derive a fundamental
SDP~\cite{DP93:Laplacian-Eigenvalues,DP93:Performance-Eigenvalue,GW95:Improved-Approximation}
that arises in combinatorial optimization.
This example highlights why large SDPs are hard to solve,
and it 
illustrates the potential of our approach.

\subsubsection{MaxCut}

Consider an undirected graph $\mathsf{G} = (\mathsf{V}, \mathsf{E})$
comprising a vertex set $\mathsf{V} = \{1, \dots, n\}$
and a set $\mathsf{E}$ of $m$ edges.
The combinatorial Laplacian of the graph is the real positive-semidefinite (psd) matrix
\begin{equation} \label{eqn:comb-laplacian}
\mtx{L} := \sum\nolimits_{\{i, j\} \in \mathsf{E}} (\mathbf{e}_i - \mathbf{e}_j)(\mathbf{e}_i - \mathbf{e}_j)^*
	\in \R^{n \times n},
\end{equation}
where $\mathbf{e}_i \in \R^n$ denotes the $i$th standard basis vector and ${}^*$
refers to the (conjugate) transpose of a matrix or vector.
We can search for a maximum-weight cut in the graph by solving 
\begin{equation} \label{eqn:maxcut}
\maximize\quad \vct{\chi}^* \mtx{L} \vct{\chi}
\quad\subjto\quad	\vct{\chi} \in \{ \pm 1 \}^n.
\end{equation}
Unfortunately, the formulation~\cref{eqn:maxcut} is \textsf{NP}-hard~\cite{Kar72:Reducibility-Combinatorial}.
One remedy is to relax it to an SDP.

Consider the matrix $\mtx{X} = \vct{\chi\chi}^*$ where $\vct{\chi} \in \{ \pm 1 \}^n$.
The matrix $\mtx{X}$ is psd; its diagonal entries equal one; and it has rank one.
We can express the \textsf{MaxCut} problem~\cref{eqn:maxcut}
in terms of the matrix $\mtx{X}$
by rewriting the objective as a trace.
Bringing forward the implicit constraints on $\mtx{X}$
and dropping the rank constraint, we arrive at the \textsf{MaxCut} SDP:
\begin{equation} \label{eqn:maxcut-sdp}
\maximize\quad \trace(\mtx{LX})
\quad\subjto\quad \diag(\mtx{X}) = \vct{1},
\quad \text{$\mtx{X}$ is psd.}
\end{equation}
As usual, $\diag$ 
extracts the diagonal of a matrix as a vector,
and $\vct{1} \in \R^n$ is the vector of ones.

The matrix solution $\mtx{X}_{\star}$ 
of~\cref{eqn:maxcut-sdp} does not
immediately yield a cut.  Let $\vct{x}_{\star}\vct{x}_{\star}^*$ be a best
rank-one approximation of $\mtx{X}_{\star}$ with respect to the Frobenius norm.
Then the vector $\vct{\chi}_{\star} = \sgn(\vct{x}_{\star})$ is a valid cut. 
In many cases, the cut $\vct{\chi}_{\star}$ %
yields an excellent solution to the discrete \textsf{MaxCut}
problem~\cref{eqn:maxcut}.  We can also use $\mtx{X}_{\star}$ to compute a cut that
is provably near-optimal via a more involved randomized rounding
procedure~\cite{GW95:Improved-Approximation}.

\subsubsection{What's the issue?\nopunct}
\label{sec:maxcut-issue}

We specify an instance of the \textsf{MaxCut} SDP~\cref{eqn:maxcut-sdp}
by means of the Laplacian $\mtx{L}$ of the graph, which has $\mathcal{O}(m+n)$ nonzero entries.
Our goal is to compute a best rank-one approximation of a solution to the SDP, which has
$\mathcal{O}(n)$ degrees of freedom.  In other words, the total cost of representing
the input and output of the problem is $\mathcal{O}(m+n)$.  Sadly, the
matrix variable in~\cref{eqn:maxcut-sdp} seems to require storage $\mathcal{O}(n^2)$.
For example, a graph $\mathsf{G}$ with one million vertices leads to an SDP~\cref{eqn:maxcut-sdp}
with a trillion real variables. 

Storage is one of the main reasons that it has been challenging
to solve large instances of the \textsf{MaxCut} SDP reliably.
Undeterred, we raise a question:

\vspace{0.5pc}

\begin{quotation}
\noindent \emph{
Can we provably find a best rank-one approximation
of a solution to the \textsf{MaxCut} SDP~\cref{eqn:maxcut-sdp}
with storage $\mathcal{O}(m+n)$?
Can we achieve working storage $\mathcal{O}(n)$?}
\end{quotation}

\vspace{0.5pc}

\noindent
We are not aware of any correct algorithm that can solve
an arbitrary instance of \cref{eqn:maxcut-sdp}
with a working storage guarantee better than $\Theta(\min\{m,n^{3/2}\})$;
see \cref{sec:related-work}.

In addition to the limit on storage,
a good algorithm should interact with
the Laplacian $\mtx{L}$ only through noninvasive, low-cost operations,
such as matrix--vector multiplication.

\subsubsection{A storage-optimal algorithm for the \textsf{MaxCut} SDP}

Surprisingly, it is possible to achieve all the goals announced
in the last subsection. 

\begin{theorem}[\textsf{MaxCut} via \sCGAL] \label{thm:scgal-maxcut}
For any $\eps, \zeta > 0$ and any Laplacian $\mtx{L} \in \R^{n \times n}$,
the \sCGAL\ algorithm computes a $(1+\zeta)$-optimal 
rank-one approximation of an $\eps$-optimal point of~\cref{eqn:maxcut-sdp};
see~\cref{sec:approx-soln}.
The working storage is $\mathcal{O}(n/\zeta)$.
The algorithm performs at most $\tilde{\mathcal{O}}(\eps^{-2.5})$
matrix--vector multiplies with the Laplacian $\mtx{L}$, plus
lower-order arithmetic.  The algorithm is randomized;
it succeeds with high probability over its random choices. 
\end{theorem}

\Cref{thm:scgal-maxcut} follows from \cref{thm:scgal}.
As usual, $\tilde{\mathcal{O}}$ suppresses constants and logarithmic factors.
In contrast to Burer--Monteiro methods~\cite{WW18:RankOptimality}
and to the approximate complementarity paradigm~\cite{DYC+19:ApproximateComplementarity},
\sCGAL\ provably succeeds for every instance of \textsf{MaxCut}.

In our experience, \sCGAL works \emph{better} than the theorem says. 
\Cref{fig:MaxCut-1} compares its scalability with four standard SDP solvers
on a laptop equivalent (details in \cref{sec:maxcut-numerics}).

\begin{figure}[t!]
    \centering
    \includegraphics[height=4.25cm]{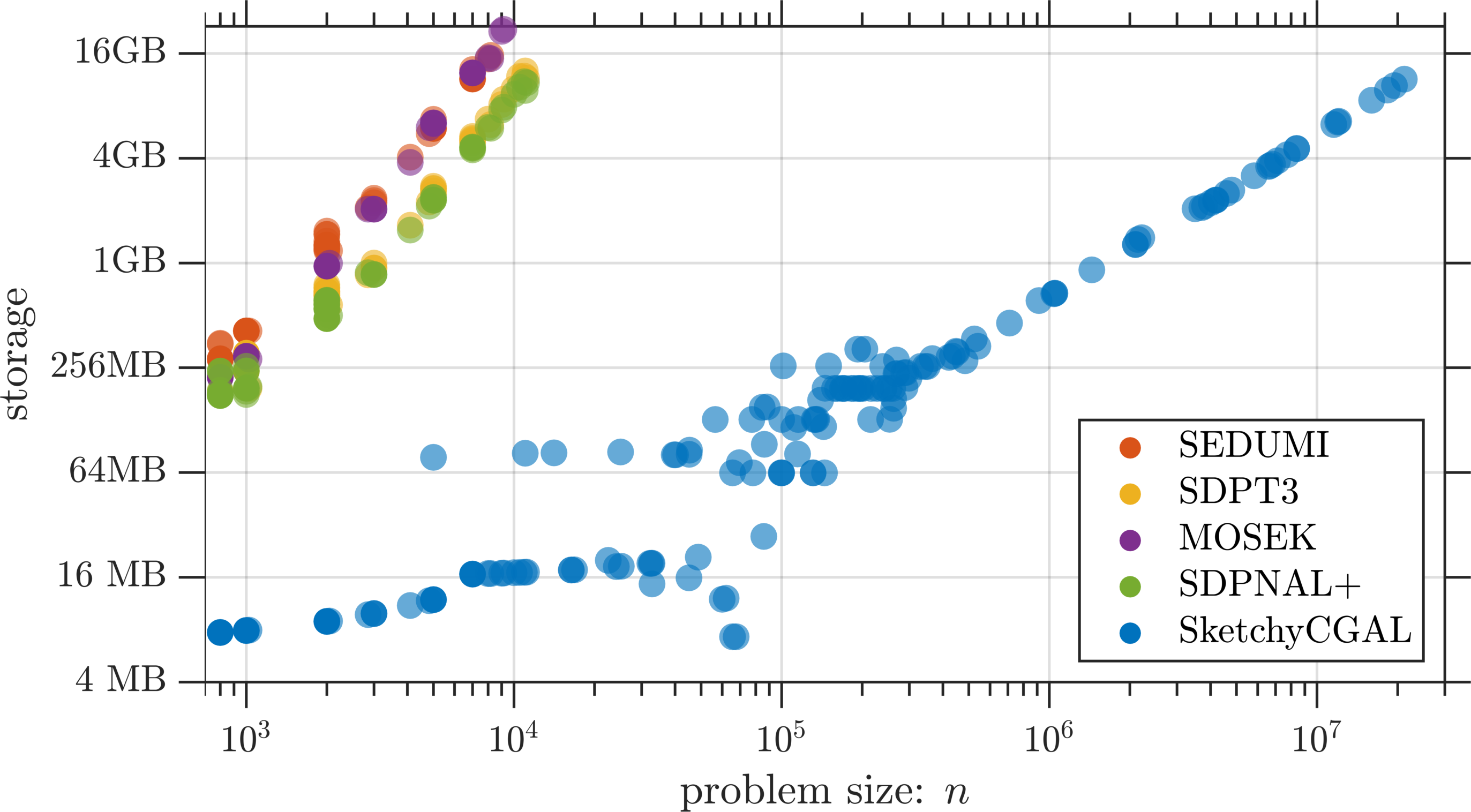}
    \hfill
    \includegraphics[height=4.25cm]{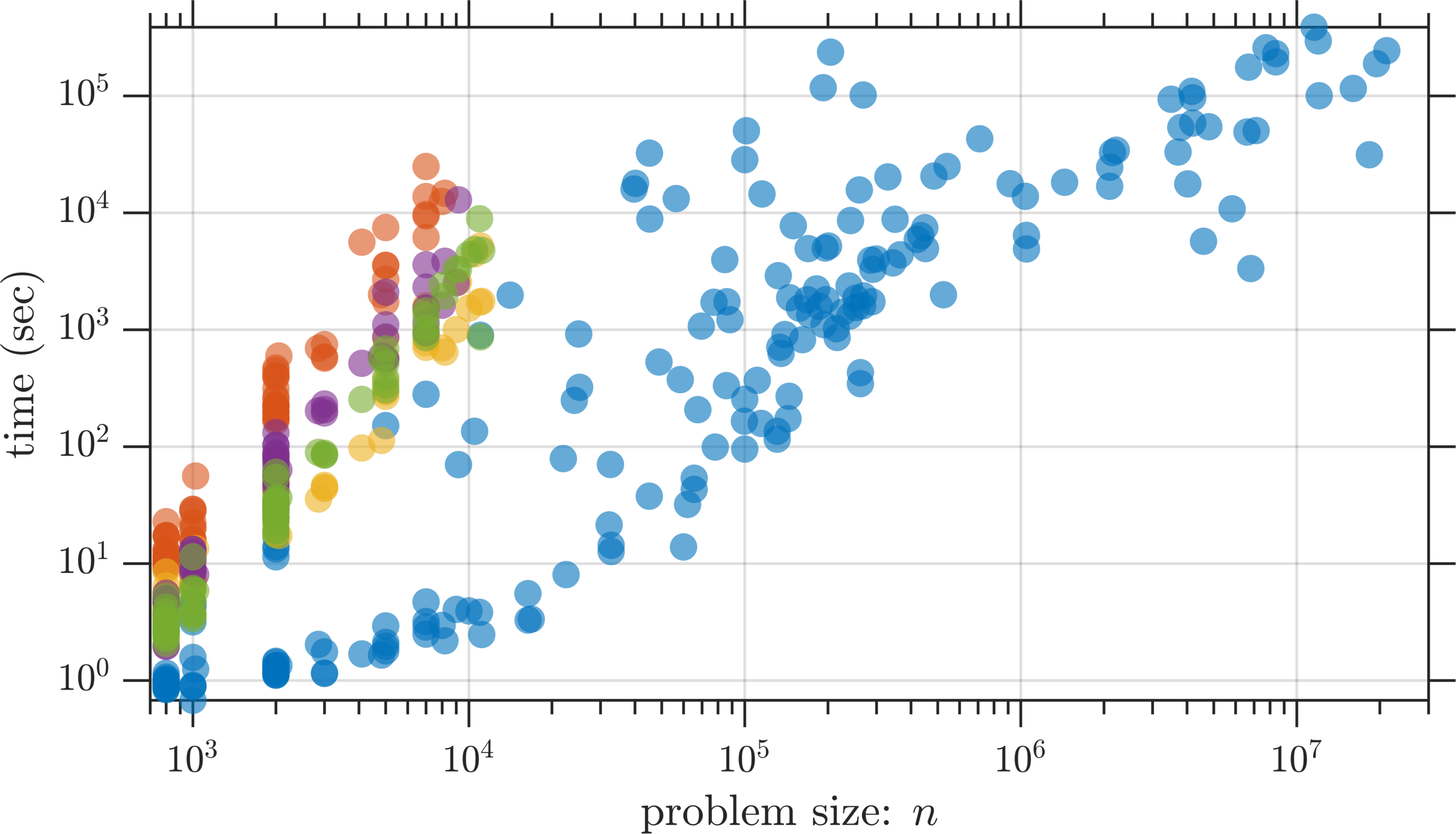}
\caption{\textsf{\textbf{\textsf{MaxCut} SDP: Scalability.}}
Storage cost \textrm{[left]} and runtime \textrm{[right]} of \sCGAL\ with sketch size $R=10$ as compared with four standard SDP solvers.  The relative error tolerance for each solver is $10^{-1}$. 
Each marker represents one dataset.  See \cref{sec:maxcut-fig1} for details.}
\label{fig:MaxCut-1}
\end{figure}

\subsection{A model problem}

\sCGAL\ can solve all standard-form SDPs that satisfy strong duality. 
To simplify parts of the presentation,
we focus on a model problem that includes
an extra trace constraint.
\Cref{sec:extensions} 
extends \sCGAL\ to 
a more expressive problem template
that includes standard-form SDPs
with additional (conic) inequality constraints.

\subsubsection{The trace-constrained SDP}

We work over the field $\F = \R$ or $\F = \C$.
For each $n \in \mathbb{N}$, define the set $\Sym_n := \Sym_n(\F)$
of (conjugate) symmetric $n \times n$ matrices with
entries in $\F$.

Introduce the set of $n \times n$ psd matrices with trace one:
\begin{equation} \label{eqn:Delta}
\mtx{\Delta}_n :=
	\{ \mtx{X} \in \Sym_n : \text{$\trace \mtx{X} = 1$ and $\mtx{X}$ is psd} \}.
\end{equation}
Our model problem is the following trace-constrained SDP:
\begin{equation} \label{eqn:sdp-explicit}
\begin{aligned}
\minimize & \quad \trace(\mtx{CX}) \\
\subjto & \quad \trace(\mtx{A}_i \mtx{X}) = b_i \quad\text{for $i = 1, \dots, d$} 
\quad\text{and}\quad \mtx{X} \in \alpha \mtx{\Delta}_n.
\end{aligned}
\end{equation}
The trace parameter $\alpha > 0$,
each matrix $\mtx{C}, \mtx{A}_1, \dots, \mtx{A}_d \in \Sym_n$,
and $b_1, \dots, b_d \in \R$.
We always assume that~\cref{eqn:sdp-explicit} satisfies strong duality
with its standard-form dual problem.

To solve a general standard-form SDP, we replace the inclusion $\mtx{X} \in \alpha \mtx{\Delta}_n$
with the constraints that $\mtx{X}$ is psd and $\trace( \mtx{X} ) \leq \alpha$
for a large enough parameter $\alpha$. %

\subsubsection{Applications}
The model problem~\cref{eqn:sdp-explicit} has diverse applications
in statistics, signal processing, quantum information theory, combinatorics,
and beyond. 
Evidently, the \textsf{MaxCut} SDP~\cref{eqn:maxcut-sdp} is a special case. 
The template~\cref{eqn:sdp-explicit} includes problems
in computer vision~\cite{HCG14:Scalable-Semidefinite},
in microscopy~\cite{Horstmeyer_2015}, and
in robotics~\cite{RCBL19:SE-Sync-Certifiably}.
It also supports contemporary machine learning tasks,
such as certifying robustness of neural networks~\cite{RSL18:Semidefinite-Relaxations}.
There is a galaxy of other examples.

\subsection{Complexity of SDP formulations and solutions}

This section describes some special features
that commonly appear in large, real-world SDPs.
Our algorithm will take advantage of these features,
even as it provides
guarantees for every instance of~\cref{eqn:sdp-explicit}.

\subsubsection{Structure of the problem data}

The matrices $\mtx{C}$ and $\mtx{A}_i$ that appear
in~\cref{eqn:sdp-explicit}
are often highly structured, or sparse, or have low-rank.
As such, we can specify the SDP using a small amount
of information.
In our work, we exploit this property by treating
the problem data for the SDP~\cref{eqn:sdp-explicit}
as a collection of black boxes that support specific
linear algebraic operations.  The algorithm for solving
the SDP only needs to access the data via these black
boxes, and we can insist that these subroutines
are implemented efficiently.

\subsubsection{Low-rank solutions of SDPs}
\label{sec:low-rank-solutions}

We will also capitalize on the fact that SDPs frequently
have low-rank solutions, or the solutions are approximated
well by low-rank matrices.  There are several reasons why
we can make this surmise.

\paragraph{Weakly-constrained SDPs}

First, many SDPs have low-rank solutions just because they are weakly constrained.
That is, the number $d$ of linear equalities 
in~\cref{eqn:sdp-explicit}
is much smaller than the 
number $n^2$ of components in the matrix variable.
This situation often occurs in signal processing and statistics problems, where $d$
reflects the amount of measured data. 
\sCGAL\ is designed for weakly constrained SDPs,
but it does not require this property.

A weakly constrained SDP has at least one low-rank solution
because of the geometry of the set of
psd matrices; see~\cite[Prop.~II(13.4) and Prob. II.14.5]{Bar02:Course-Convexity}
and~\cite{Pat98:Rank-Extreme}.

\begin{fact}[Barvinok--Pataki]
\label{fact:barvinok-pataki}
When $\mathbb{F} = \R$,
the SDP~\cref{eqn:sdp-explicit} has a solution with
rank $r \leq \sqrt{2(d+1)}$. 
When $\mathbb{F} = \C$,
there is a solution with rank $r \leq \sqrt{d+1}$.
\end{fact}

\noindent
For example, the \textsf{MaxCut} SDP~\cref{eqn:maxcut-sdp}
admits a solution with rank $\sqrt{2(n+1)}$.

Although a weakly-constrained SDP can have solutions with high rank,
a \emph{generic} weakly-constrained SDP has a unique solution, which
must be low-rank~\cite{AHO97:Complementarity-Nondegeneracy}.

\begin{fact}[Alizadeh et al.]
\label{fact:alizadeh}
Let $\mathbb{F} = \R$.
Except for a set of matrices $\{ \mtx{C}, \mtx{A}_1, \dots, \mtx{A}_d \}$
with measure zero, the solution set of the SDP~\cref{eqn:sdp-explicit}
is a unique matrix with rank $r \leq \sqrt{2(d+1)}$.
\end{fact}

\paragraph{Matrix rank minimization}

Second, some SDPs are \emph{designed}
to produce a low-rank matrix that satisfies a system of linear matrix equations.
This idea can be traced to the control theory
literature~\cite{MP97:Rank-Minimization,Par00:Structured-Semidefinite},
and it was explored thoroughly in Fazel's thesis~\cite{Faz02:Matrix-Rank}.
Early applications include Euclidean distance
matrix completion~\cite{AKW99:Solving-Euclidean}
and collaborative filtering~\cite{Sre04:Learning-Matrix}.
Extensive empirical work indicates
that these SDPs often produce low-rank solutions.

\paragraph{Structural properties}

There are other reasons that an SDP must have a low-rank solution.
For instance, consider the optimal power flow SDPs developed
by Lavaei and Low~\cite{LL12:Zero-Duality},
where the rank of the solution
is controlled by the geometry of the power grid.

\subsubsection{Algorithms?\nopunct}

To summarize, many realistic SDPs have structured data,
and they admit solutions that are (close to) low rank.
Are there algorithms that can exploit these features?
Although there are a number of methods that attempt to do so,
none can provably solve every SDP with limited arithmetic
and minimal storage. 
See \cref{sec:related-work} for related work.

Why has it been so difficult to develop provably
correct algorithms for finding low-rank solutions
to structured SDPs?  Most approaches
that try to control the rank run
headlong into a computational complexity barrier:
For any fixed rank parameter $r$,
it is \textsf{NP}-hard to solve %
the model problem~\cref{eqn:sdp-explicit}
if the variable $\mtx{X}$ is also constrained to be a rank-$r$ matrix~\cite[p.~7]{Faz02:Matrix-Rank}.

To escape this sticky fact, we revise
the computational goal, following~\cite{YUTC17:Sketchy-Decisions}.
The key insight is to seek a rank-$r$ matrix
that \emph{approximates a solution}
to~\cref{eqn:sdp-explicit}.
See~\cref{sec:approx-soln} for a detailed explanation.
This shift in perspective opens up new algorithmic prospects.

\subsection{Contributions}

Inspired by~\cite{YUTC17:Sketchy-Decisions},
we derive an algorithm that harnesses
the favorable properties %
common in large SDPs.
The \sCGAL\ algorithm solves
the SDP~\cref{eqn:sdp-explicit} using a primal--dual optimization
method~\cite{YFC19:Conditional-Gradient-Based} developed by a subset of the authors.
Each iteration requires one coarse eigenvector computation~\cite{KW92:Estimating-Largest}
and leads to a rank-one update of the psd matrix variable.
Instead of storing this matrix, we maintain a compressed representation
by means of a matrix sketching technique~\cite{TYUC17:Fixed-Rank-Approximation}.
After the optimization algorithm terminates, we extract from the sketch
a low-rank approximation of a solution of the SDP.
This idea leads to a practical, provably correct
SDP solver that economizes on storage and arithmetic.

\begin{theorem}[The model problem via \sCGAL] \label{thm:scgal-intro}
Assume that the model problem~\cref{eqn:sdp-explicit} satisfies strong duality.
For any $\eps,\zeta > 0$ and any rank parameter $r$,
the \sCGAL\ algorithm computes a
$(1+\zeta)$-optimal rank-$r$ approximation of an
$\eps$-optimal
point of~\cref{eqn:sdp-explicit}; 
see~\cref{sec:approx-soln}.
The storage cost is $\mathcal{O}(d + rn/\zeta)$.
Most of the arithmetic consists
in $\tilde{\mathcal{O}}(\eps^{-2.5})$
matrix--vector products with each matrix
$\mtx{C}, \mtx{A}_1, \dots, \mtx{A}_d$
from the problem data.  The algorithm succeeds
with high probability. 
\end{theorem}

\Cref{thm:scgal} contains full theoretical details.
Note that the storage $\Theta(d + rn)$ is the minimum possible
for any primal--dual algorithm that returns a rank-$r$ solution
to~\cref{eqn:sdp-explicit}.
\Cref{sec:numerics} contains numerical evidence
that the algorithm is effective for a range of
examples.

\subsection{Roadmap}

\Cref{sec:scalable} presents an abstract framework for
studying SDPs that exposes the challenges associated
with large problems.  \Cref{sec:cgal} outlines a primal--dual
algorithm, called \CGAL, for the model problem~\cref{eqn:sdp-explicit}.
\Cref{sec:all-eigs,sec:nystrom} introduce methods
from randomized linear algebra that we use to control
storage and arithmetic costs.
\Cref{sec:sketchy-cgal} develops the \sCGAL\
algorithm, its convergence theory, and its resource
usage guarantees.
\Cref{sec:numerics} contains a numerical study of \sCGAL, and
\cref{sec:related-work} covers related work.

\subsection{Notation}

The symbol $\fnorm{\cdot}$ denotes the Frobenius norm,
while $\norm{\cdot}_*$ is the nuclear norm (i.e., Schatten-1).
The unadorned norm $\norm{\cdot}$
refers to the $\ell_2$ norm of a vector,
the spectral norm of a matrix,
or the operator norm of a linear map from $(\Sym_n, \fnorm{\cdot})$ to $(\R^d, \norm{\cdot})$.
We write $\ip{\cdot}{\cdot}$ for both the
$\ell_2$ inner product on vectors
and the trace inner product on matrices.

The map $\lowrank{\mtx{M}}{r}$ returns an $r$-truncated singular-value decomposition
of the matrix $\mtx{M}$, which is a best rank-$r$ approximation with respect to
every unitarily invariant norm~\cite{Mir60:Symmetric-Gauge}.

We use the standard computer science interpretation of the asymptotic notation
$\mathcal{O}, \tilde{\mathcal{O}}, \Theta$.

\section{Scalable semidefinite programming}
\label{sec:scalable}

To solve the model problem~\cref{eqn:sdp-explicit} efficiently,
we need to exploit structure inherent in the problem data.
This section outlines an abstract approach
that directs our attention to the core
computational difficulties.

\subsection{Abstract form of the model problem}

Let us instate compact notation for the
linear constraints in the model problem~\cref{eqn:sdp-explicit}.
Define a linear map $\mathcal{A}$ and its adjoint $\mathcal{A}^*$
via
\begin{equation} \label{eqn:linear-maps}
\begin{aligned}
\mathcal{A} &: \Sym_n \to \R^d
\quad\text{where}
&&\mathcal{A}\mtx{X} = \begin{bmatrix} \ip{\mtx{A}_1}{\mtx{X}} & \dots & \ip{\mtx{A}_d}{ \mtx{X} } \end{bmatrix}; \\
\mathcal{A}^* &: \R^d \to \Sym_n
\quad\text{where}
&&\mathcal{A}^* \vct{z} = \sum\nolimits_{i=1}^d z_i \mtx{A}_i.
\end{aligned}
\end{equation}
We bound the linear map with the operator norm $\norm{\mathcal{A}} := \norm{\mathcal{A}}_{\mathrm{F} \to \ell_2}$.
Form the vector $\vct{b} := (b_1, \dots, b_d) \in \R^d$ of constraint values.
In this notation, \cref{eqn:sdp-explicit} becomes %
\begin{equation} \label{eqn:model-problem}
\minimize\quad \ip{\mtx{C}}{\mtx{X}}
\quad\subjto\quad \mathcal{A}\mtx{X} = \vct{b},
\quad \mtx{X} \in \alpha \mtx{\Delta}_n.
\end{equation}
Problem instances are parameterized by the tuple $(\mtx{C}, \mathcal{A}, \vct{b}, \alpha)$.

\subsection{Approximate solutions}
\label{sec:approx-soln}

Let $\mtx{X}_{\star}$ be a solution to the model problem~\cref{eqn:model-problem}.
For $\eps \geq 0$,
we say that a matrix $\mtx{X}_{\eps}$ is \emph{$\eps$-optimal} for~\cref{eqn:model-problem} when
\begin{equation*} \label{eqn:eps-optimal}
\norm{ \mathcal{A}\mtx{X}_{\eps} - \vct{b} } \leq \eps
\quad\text{and}\quad
\ip{ \mtx{C} }{ \mtx{X}_{\eps} } - \ip{ \mtx{C} }{ \mtx{X}_{\star} } \leq \eps.
\end{equation*}
Many optimization algorithms aim to produce such $\eps$-optimal points;
cf.~\cref{sec:related-work}.

As we saw in \cref{sec:low-rank-solutions}, there are many
situations where the solutions to~\cref{eqn:model-problem}
have low rank or they admit accurate low-rank approximations.
The $\eps$-optimal points inherit these 
properties for sufficiently small $\eps$. 
This insight suggests a new computational goal.

For a rank parameter $r$, we will seek a rank-$r$ matrix $\widehat{\mtx{X}}$ that
approximates an $\eps$-optimal point $\mtx{X}_{\eps}$.  More precisely, for a fixed suboptimality parameter $\zeta > 0$,
we want
\begin{equation} \label{eqn:goal}
\norm{ \mtx{X}_{\eps} - \widehat{\mtx{X}} }_*
	\leq (1+\zeta) \cdot \norm{ \mtx{X}_{\eps} - \lowrank{\mtx{X}_{\eps}}{r} }_*
	\quad\text{where $\rank \widehat{\mtx{X}} \leq r$ and $\mtx{X}_\eps$ is $\eps$-optimal.}
\end{equation}
Given~\cref{eqn:goal},
if the $\eps$-optimal point $\mtx{X}_{\eps}$ is close to \emph{any} rank-$r$ matrix, then
the rank-$r$ approximate solution $\widehat{\mtx{X}}$ is also close to the $\eps$-optimal point $\mtx{X}_{\eps}$.
This formulation is advantageous because it is easier to 
compute and to store the low-rank matrix $\widehat{\mtx{X}}$.

\subsection{Black-box presentation of problem data}

To develop scalable algorithms for~\cref{eqn:model-problem},
it is productive to hide the internal complexity of the problem instance
from the algorithm~\cite{YUTC17:Sketchy-Decisions}.
To do so, we treat $\mtx{C}$ and $\mathcal{A}$ as black
boxes that support three primitive computations:
\begin{equation} \label{eqn:primitives}
\begin{aligned}
\fcolorbox{black}{ltblue}{$
\begin{aligned}
\primone &&
\vct{u} &\mapsto \mtx{C} \vct{u} \\
&& \R^n &\to \R^n
\end{aligned}$}
\qquad\quad %
\fcolorbox{black}{ltblue}{$
\begin{aligned}
\primtwo &&
(\vct{u}, \vct{z}) &\mapsto (\mathcal{A}^*\vct{z}) \vct{u} \\
&& \R^n \times \R^d &\to \R^n
\end{aligned}$}
\qquad\quad
\fcolorbox{black}{ltblue}{$
\begin{aligned}
\primthree &&
\vct{u} &\mapsto \mathcal{A}(\vct{uu}^*) \\
&& \R^n &\to \R^d
\end{aligned}$}
\end{aligned}
\end{equation}
The vectors $\vct{u} \in \R^n$ and $\vct{z} \in \R^d$ are arbitrary.
Although these functions may seem abstract, they are often quite
natural and easy to implement well.  For example, see~\cref{sec:maxcut-storage-opt}.

We will formulate algorithms for~\cref{eqn:model-problem} that
interact with the problem data only through the operations~\cref{eqn:primitives}.
We tacitly assume that the primitives require minimal storage and arithmetic;
otherwise, it may be impossible to develop a truly efficient algorithm.

\subsection{Example: The \textsf{MaxCut} SDP}
\label{sec:maxcut-storage-opt}

To provide a concrete example, let us summarize the meaning of the primitives
and the desired storage costs for solving the \textsf{MaxCut} SDP~\cref{eqn:maxcut-sdp}.

\begin{itemize}
\item	The primitive $\primone : \vct{u} \mapsto - \mtx{L} \vct{u}$.
In the typical case that the Laplacian $\mtx{L}$ is sparse,
this amounts to a sparse matrix--vector multiply.

\item	The primitive $\primtwo : (\vct{u}, \vct{z}) \mapsto (\diag^* \vct{z})\vct{u}
= (z_1 u_1, \dots, z_n u_n)$.

\item	The primitive $\primthree : \vct{u} \mapsto \diag( \vct{uu}^* )
	=(\abs{u_1}^2, \dots, \abs{u_n}^2)$.

\item	The \textsf{MaxCut} SDP has $n$ linear constraints, and we seek a rank-one
approximation of the solution.  Thus, we desire an algorithm that operates with $\Theta(n)$
working storage.
\end{itemize}

\noindent
Note that we still need $\Theta(m)$ numbers to store the Laplacian $\mtx{L}$
of a generic graph with $m$ edges.  But we do not charge the optimization algorithm
for this storage 
because the algorithm only interacts with $\mtx{L}$ through the primitives~\cref{eqn:primitives}.

\section{An algorithm for the model problem}
\label{sec:cgal}

We will develop a scalable method for the model problem~\cref{eqn:model-problem}
by enhancing an existing algorithm, called \CGAL~\cite{YFC19:Conditional-Gradient-Based},
developed by a subset of the authors.
This method works well, 
but it is not as scalable as we would
like because it lacks storage and
arithmetic guarantees.
This section summarizes the \CGAL\ method and its convergence properties.
Subsequent sections introduce additional ideas that we need to control
resource usage, culminating with the \sCGAL\ algorithm in \cref{sec:sketchy-cgal}.

\subsection{The augmented problem}

For a parameter $\beta > 0$, we revise the problem~\cref{eqn:model-problem}
by introducing an extra term in the objective:
\begin{equation} \label{eqn:augmented-problem}
\minimize\quad \ip{\mtx{C}}{\mtx{X}} + \frac{\beta}{2} \normsq{ \mathcal{A}\mtx{X} - \vct{b} }
\quad\subjto\quad \mathcal{A}\mtx{X} = \vct{b}, \quad
\mtx{X} \in \alpha \mtx{\Delta}_n.
\end{equation}
The original problem~\cref{eqn:model-problem} and the augmented problem~\cref{eqn:augmented-problem}
share the same optimal value and optimal set.  But the new formulation has several benefits:
the augmented objective cooperates with the affine constraint to penalize infeasible points,
and the dual of the augmented problem is smooth, whereas the dual of the original problem is not.
See~\cite[Ch.~2]{Ber82:Constrained-Optimization} for background.

\subsection{The augmented Lagrangian}

We construct the Lagrangian $L_{\beta}$ of the augmented problem~\cref{eqn:augmented-problem}
by introducing a dual variable $\vct{y} \in \R^d$
and promoting the affine constraint:
\begin{equation} \label{eqn:augmented-lagrangian}
L_{\beta}(\mtx{X}; \vct{y}) := \ip{ \mtx{C} }{\mtx{X} }
	+ \ip{ \vct{y} }{ \mathcal{A}\mtx{X} - \vct{b} }
	+ \frac{\beta}{2} \normsq{\mathcal{A}\mtx{X} - \vct{b}}
	\quad\text{for $\mtx{X} \in \alpha \mtx{\Delta}_n$ and $\vct{y} \in \R^d$.}
\end{equation}
For reference, note that the partial derivatives of the augmented Lagrangian $L_{\beta}$ satisfy
\begin{equation} \label{eqn:lagrangian-partials}
\begin{aligned}
\partial_{\mtx{X}} L_{\beta}(\mtx{X}; \vct{y})
	= \mtx{C} + \mathcal{A}^* \vct{y} + \beta \mathcal{A}^* (\mathcal{A}\mtx{X} - \vct{b}) \quad\text{and}\quad
\partial_{\vct{y}} L_{\beta}(\mtx{X}; \vct{y})
	= \mathcal{A}\mtx{X} - \vct{b}.
\end{aligned}
\end{equation}
We attempt to minimize the augmented Lagrangian $L_{\beta}$ with respect to the primal variable $\mtx{X}$,
while we attempt to maximize with respect to the dual variable $\vct{y}$.

\subsection{The augmented Lagrangian strategy}

The form of the augmented Lagrangian suggests an algorithm.
We generate a sequence $\{ (\mtx{X}_t; \vct{y}_t) \}$ of primal--dual
pairs by alternately minimizing over the primal variable $\mtx{X}$
and taking a gradient step in the dual variable $\vct{y}$.

\begin{enumerate} 
\item	\textbf{Initialize:}  Choose $\mtx{X}_1 \in \alpha\mtx{\Delta}_n$ and $\vct{y}_1 \in \R^d$.

\item	\textbf{Primal step:}  $\mtx{X}_{t+1} \in \argmin\{ L_{\beta}(\mtx{X}; \vct{y}_t) : \mtx{X} \in \alpha \mtx{\Delta}_n \}$.

\item	\textbf{Dual step:} $\vct{y}_{t+1} = \vct{y}_t + \beta ( \mathcal{A} \mtx{X} - \vct{b} )$.

\end{enumerate}

\noindent
As we proceed, we can also increase the smoothing parameter $\beta$
to make violations of the affine constraint in~\cref{eqn:augmented-problem}
more and more intolerable.

The augmented Lagrangian strategy is powerful, but it is hard to apply
in this setting because of the cost of implementing the primal step,
even approximately.

\subsection{The CGAL iteration}
\label{sec:cgal-strategy}

The \CGAL\ iteration~\cite{YFC19:Conditional-Gradient-Based}
is related to the augmented Lagrangian paradigm,
but the primal steps are inspired by the conditional gradient method~\cite{FW56:FrankWolfe}.
\CGAL\ identifies an update direction for the primal variable by minimizing
a linear proxy for the augmented Lagrangian.  We take a small primal
step in the update direction, and we improve the dual variable by taking
a small gradient step.  At each iteration, the smoothing
parameter increases according to a fixed schedule.

This subsection outlines the steps in the \CGAL\ iteration.
Afterward, we give 
\emph{a priori} guarantees on the convergence rate.
Pseudocode for \CGAL\ appears as~\cref{alg:cgal}.

\begin{algorithm}[t]
	\footnotesize
  \caption{\CGAL\ for the model problem~\cref{eqn:model-problem} 
  \label{alg:cgal}}
  \begin{algorithmic}[1]
  \vspace{0.5pc}

	\Require{Problem data for~\cref{eqn:model-problem} implemented via the primitives~\cref{eqn:primitives};
	number $T$ of iterations}
    \Ensure{Approximate solution $\mtx{X}_T$ to~\cref{eqn:model-problem}}

\vspace{0.5pc}

	\Function{\CGAL}{$T$}

	\State	Scale problem data (\cref{sec:numerics-scaling})
		\Comment	\textcolor{dkblue}{\textbf{[opt]}} Recommended!
	\State	$\beta_0 \gets 1$ and $K \gets + \infty$
		\Comment	Default parameters

	\State	$\mtx{X} \gets \mtx{0}_{n \times n}$ and $\vct{y} \gets \vct{0}_d$
	\For{$t \gets 1, 2, 3, \dots, T$}
  		\State	$\beta \gets \beta_0 \sqrt{t+1}$ and $\eta \gets 2 /(t+1)$
		\State	$(\xi, \vct{v}) \gets \textsf{ApproxMinEvec}( \mtx{C} + \mathcal{A}^* (\vct{y} + \beta(\mathcal{A} \mtx{X} - \vct{b})); q_t )$
		\Comment \cref{alg:rand-lanczos} with $q_t = t^{1/4} \log n$
		\Statex	\Comment Implement with primitives~\cref{eqn:primitives}\primone\primtwo!
		\State	$\vct{X} \gets (1 - \eta) \, \vct{X} + \eta \, (\alpha \, \vct{vv}^*)$
		\State	$\vct{y} \gets \vct{y} + \gamma (\mathcal{A} \mtx{X} - \vct{b})$
			\Comment	Step size $\gamma$ satisfies~\cref{eqn:cgal-dual-step}
	\EndFor
	\EndFunction

  \vspace{0.25pc}

\end{algorithmic}
\end{algorithm}

\subsubsection{Initialization}

Let $\beta_0 > 0$ be an initial smoothing parameter,
and fix a (large) bound $K > 0$ on the maximum allowable size of the dual variable.
We also assume that we have access to the norm $\norm{\mathcal{A}}$
of the constraint matrix, or---failing that---a \emph{lower} bound. 
Begin with 
an arbitrary choice $\mtx{X}_1 \in \Sym_n$ for the primal variable;
set the initial dual variable $\vct{y}_1 = \vct{0}$.

\subsubsection{Primal updates via linear minimization}

At 
iteration $t = 1, 2, 3, \dots$, we increase the smoothing parameter to
$\beta_t := \beta_0 \sqrt{t+1}$, and we form the partial derivative of the augmented Lagrangian 
with respect to the primal variable at the current pair of iterates:
\begin{equation} \label{eqn:cgal-grad}
\mtx{D}_t := \partial_{\mtx{X}} L_{\beta_t}(\mtx{X}_t; \vct{y}_t)
	= \mtx{C} + \mathcal{A}^* (\vct{y}_t + \beta_t (\mathcal{A}\mtx{X}_t - \vct{b}) ).
\end{equation}
Then we compute an update $\mtx{H}_t \in \alpha \mtx{\Delta}_t$ by
finding the feasible point that is most correlated with the negative gradient $- \mtx{D}_t$.
This amounts to a linear minimization problem:
\begin{equation} \label{eqn:cgal-lmo}
\begin{aligned}
\mtx{H}_t &\in \argmin \{ \ip{ \mtx{D}_t }{ \mtx{H} } : \mtx{H} \in \alpha \mtx{\Delta}_n \} \\
	&= \mathrm{convex\,hull}\{ \alpha \vct{vv}^* : \text{$\vct{v}$ is a unit-norm minimum eigenvector of $\mtx{D}_t$} \}.
\end{aligned}
\end{equation}
In particular, we may take $\mtx{H}_t = \alpha \vct{v}_t \vct{v}_t^*$ for a minimum eigenvector $\vct{v}_t$
of the gradient $\mtx{D}_t$.
Next, update the matrix primal variable by taking a small step in the direction $\mtx{H}_t$:
\begin{equation} \label{eqn:cgal-primal-update}
\mtx{X}_{t+1} := (1 - \eta_t) \mtx{X}_t + \eta_t \mtx{H}_t \in \alpha \mtx{\Delta}_n
\quad\text{where}\quad
\eta_t := \frac{2}{t+1}.
\end{equation}
The appeal of this approach is that it only requires
a single eigenvector computation~\cref{eqn:cgal-lmo}. 
Moreover, we need not compute the eigenvector accurately;
see~\cref{sec:approx-evec} for details.

\subsubsection{Dual updates via gradient ascent}

Next, we update the dual variable by taking a small gradient
step on the dual variable in the augmented Lagrangian: 
\begin{equation} \label{eqn:cgal-dual-update}
\begin{aligned}
\vct{y}_{t+1} = \vct{y}_t + \gamma_t \, \partial_{\vct{y}} L_{\beta_t}(\mtx{X}_{t+1}; \vct{y}_t) 
	= \vct{y}_t + \gamma_t (\mathcal{A}\mtx{X}_{t+1} - \vct{b}).
\end{aligned}
\end{equation}
The dual step size $\gamma_t$ is the largest number 
that satisfies the conditions 
\begin{equation} \label{eqn:cgal-dual-step}
\begin{aligned}
\gamma_t \normsq{\mathcal{A}\mtx{X}_{t+1} - \vct{b}}
	\leq \frac{4 \alpha^2 \beta_0}{(t+1)^{3/2}} \normsq{\mathcal{A}}
	\quad\text{and}\quad
	0 \leq \gamma_t \leq \beta_0.
\end{aligned}
\end{equation}
We omit the dual step if it makes the dual variable too large.
More precisely, if~\cref{eqn:cgal-dual-update,eqn:cgal-dual-step}
result in $\norm{ \vct{y}_{t+1} } > K$,
then we set $\vct{y}_{t+1} = \vct{y}_t$ instead.
This is the \CGAL\ iteration.

\subsubsection{Assessing solution quality}
\label{sec:solution-quality}

Given a primal--dual pair $(\mtx{X}_t; \vct{y}_t)$, we can assess
the quality of the primal solution to the model problem~\cref{eqn:model-problem}. 
First, note that we can simply compute the infeasibility: $\mathcal{A} \mtx{X}_t - \vct{b}$.
Second, we can bound the suboptimality of the primal objective value.
To do so, define the \emph{surrogate duality gap}
\begin{equation} \label{eqn:cgal-gap}
g_t(\mtx{X}) := \max_{\mtx{H} \in \alpha \mtx{\Delta}_n} \ip{\mtx{D}_t}{\mtx{X} - \mtx{H}}
	= \ip{ \mtx{C} }{ \mtx{X}_t } + \ip{ \vct{y}_t + \beta_t (\mathcal{A}\mtx{X}_t -\vct{b}) }{ \mathcal{A}\mtx{X}_t }
	- \lambda_{\min}(\mtx{D}_t).
\end{equation}
The latter expression follows from~\cref{eqn:cgal-grad,eqn:cgal-lmo}.
As usual, $\lambda_{\min}$ is the minimum eigenvalue.
The suboptimality of $\mtx{X}_t$ is bounded as
\begin{equation} \label{eqn:cgal-posterior-error-bd}
\ip{ \mtx{C} }{ \mtx{X}_t } - \ip{ \mtx{C} }{ \mtx{X}_\star }
	\leq g_t(\mtx{X}_t) - \ip{ \vct{y}_t }{ \mathcal{A} \mtx{X}_t - \vct{b} } - \frac{1}{2} \beta_t \normsq{\mathcal{A}\mtx{X}_t - \vct{b} },
\end{equation}
where $\mtx{X}_{\star}$ is a primal optimal point.
See \cref{sec:bound-subopt-supp} for the derivation of~\cref{eqn:cgal-posterior-error-bd}.

For \CGAL, the surrogate duality gap~\cref{eqn:cgal-gap} is analogous to %
the Frank--Wolfe duality gap~\cite{pmlr-v28-jaggi13}.
It offers a practical tool for detecting early convergence.
Unfortunately, we have not been able to prove that
the \CGAL\ error bound~\cref{eqn:cgal-posterior-error-bd}
converges to zero as the algorithm converges.

\subsubsection{Approximate eigenvector computations}
\label{sec:approx-evec}

A crucial fact is %
that the \CGAL\ strategy provably works if we replace~\cref{eqn:cgal-lmo}
with an approximate minimum eigenvector computation.
At step $t$, suppose that we find an update direction
\begin{equation} \label{eqn:cgal-lmo-approx}
\mtx{H}_t := \alpha \vct{v}_t \vct{v}_t^*
\quad\text{where}\quad
\vct{v}_t^* \mtx{D}_t \vct{v}_t
	\leq \lambda_{\min}(\mtx{D}_t) + \frac{1}{\sqrt{t+1}} \norm{\mtx{D}_t}
\end{equation}
for a unit vector $\vct{v}_t \in \F^n$.  %
The matrix $\mtx{H}_t$ defined in~\cref{eqn:cgal-lmo-approx}
serves in place of a solution to~\cref{eqn:cgal-lmo}
throughout the algorithm. %
We discuss solvers for the subproblem~\cref{eqn:cgal-lmo-approx}
in \cref{sec:all-eigs}.

\subsection{Convergence guarantees for CGAL}

The following result describes the convergence of the \CGAL\ iteration.
The analysis is adapted from~\cite[Thm.~3.1]{YFC19:Conditional-Gradient-Based};
see~\cref{sec:cgal-theory-supp} for details and a recapitulation of the proof.

\begin{fact}[\CGAL: Convergence] \label{fact:cgal-converge}
Assume problem~\cref{eqn:model-problem} satisfies strong duality.
The \CGAL\ iteration (\cref{sec:cgal-strategy})
with approximate eigenvector computations~\cref{eqn:cgal-lmo-approx}
yields a sequence $\{\mtx{X}_t : t =1,2,3,\dots\} \subset \alpha \mtx{\Delta}_n$
that satisfies
\begin{align}
\norm{ \mathcal{A}\mtx{X}_t - \vct{b} }
	\leq \frac{\mathrm{Const}}{\sqrt{t}} %
	\quad\text{and}\quad
	\abs{\ip{\mtx{C}}{\mtx{X}_t} - \ip{\mtx{C}}{\mtx{X}_{\star}}}
	\leq \frac{\mathrm{Const}}{\sqrt{t}}.
	\label{eqn:cgal-bounds}
\end{align}
The matrix $\mtx{X}_{\star}$ solves~\cref{eqn:model-problem}.
The constant depends on the problem data $(\mtx{C}, \mathcal{A}, \vct{b}, \alpha)$,
the minimum Euclidean norm of a dual solution, %
and the algorithm parameters $\beta_0$ and $K$.
\end{fact}

In view of~\cref{eqn:cgal-bounds}, \CGAL\ produces
a primal iterate $\mtx{X}_T$ that is $\eps$-optimal after $T = \mathcal{O}( \eps^{-2} )$ iterations.
The numerical work in~\cite[Sec.~5]{YFC19:Conditional-Gradient-Based}
shows that \CGAL\ finds an $\eps$-optimal point
after $T = \mathcal{O}(\eps^{-1})$ iterations for realistic problem instances.
See also~\cref{sec:primal-dual-conv-supp}.

The analysis behind \cref{fact:cgal-converge}
indicates that \CGAL\ converges more quickly
if we pre-condition the problem data by rescaling;
see \cref{sec:numerics-scaling}. %
This step is critical in practice. %

\section{Approximate eigenvectors}
\label{sec:all-eigs}

Most of the computation in \CGAL\ occurs in
the approximate eigenvector step~\cref{eqn:cgal-lmo-approx}.
This section describes our approach to this subproblem.

\subsection{Krylov methods}

The approximate minimum eigenvector computation~\cref{eqn:cgal-lmo-approx}
involves a matrix of the form $\mtx{D}_t = \mtx{C} + \mathcal{A}^* \vct{w}_t$.
We can multiply the matrix $\mtx{D}_t$ by a vector using the %
primitives~\cref{eqn:primitives}\primone\primtwo.  Krylov methods compute
eigenvectors of $\mtx{D}_t$ by repeated matrix--vector multiplication with $\mtx{D}_t$,
so they are obvious tools for the subproblem~\cref{eqn:cgal-lmo-approx}.

Unfortunately, \CGAL\ tends to generate matrices $\mtx{D}_t$ that have clustered
eigenvalues.  These matrices challenge standard eigenvector software,
such as \textsf{ARPACK} \cite{LSY98:ARPACK},
the engine behind the \textsc{Matlab} command \texttt{eigs}.
Instead, we retreat to a more classical technique. %

\subsection{A storage-optimal randomized Lanczos method}
\label{sec:rand-lanczos}

We use a nonstandard implementation of the randomized Lanczos method~\cite{KW92:Estimating-Largest}
to find an approximate eigenvector with minimal storage.
This hinges on two ideas.
First, we run the Lanczos iteration with a random
initial vector, which supports convergence guarantees.
Second, we do not store the Lanczos vectors,
but rather regenerate them to construct the approximate eigenvector.
See~\cref{alg:rand-lanczos} for pseudocode.
Kuczy{\'n}ski \& Wo{\'z}niakowski~\cite[Thm.~4.2(a)]{KW92:Estimating-Largest}
have obtained error bounds. %

\begin{fact}[Randomized Lanczos method] \label{fact:rand-lanczos}
Let $\mtx{M} \in \Sym_n$.  For $\eps \in (0, 1]$ and $\delta \in (0, 0.5]$,
the randomized Lanczos method, \cref{alg:rand-lanczos},
computes a unit vector $\vct{u} \in \F^n$ that satisfies
$$
\vct{u}^* \mtx{M} \vct{u} \leq \lambda_{\min}(\mtx{M}) + \frac{\eps}{8} \norm{\mtx{M}}
\quad\text{with probability at least $1 - 2\delta$}
$$
after $q \geq \tfrac{1}{2} + \eps^{-1/2} \log(n/\delta^2)$ iterations.
The arithmetic cost
is at most $q$ matrix--vector multiplies with $\mtx{M}$ and
$\mathcal{O}(q n)$ extra operations.
The working storage is $\mathcal{O}(n + q)$. %
\end{fact}

With constant probability, we solve the eigenvector problem~\cref{eqn:cgal-lmo-approx}
successfully in every iteration $t$ of \CGAL\ if we use $q_t = \mathcal{O}(t^{1/4} \log(tn))$
iterations of \cref{alg:rand-lanczos}.  In practice, we implement
\CGAL\ with $q_t = t^{1/4} \log n$.
Although the Lanczos method has complicated numerical behavior
in finite-precision arithmetic, it serves well as a subroutine within \CGAL.

\begin{remark}[Time--storage tradeoff]
We can retain the vectors $\vct{v}_i$ in \cref{alg:rand-lanczos}
to reduce the flop count by approximately a factor two,
but the storage increases to $\mathcal{O}(qn)$.
\end{remark}

\begin{algorithm}[t]%
	\footnotesize
  \caption{\textsf{ApproxMinEvec} via storage-optimal randomized Lanczos %
  (\cref{sec:rand-lanczos}).}  %
  \label{alg:rand-lanczos}
  \begin{algorithmic}[1]
  \vspace{0.5pc}

	\Require{Input matrix $\mtx{M} \in \Sym_n$ and maximum number $q$ of iterations}
    \Ensure{Approximate minimum eigenpair $(\xi, \vct{v}) \in \R \times \F^n$ of the matrix $\mtx{M}$}

\vspace{0.5pc}

	\Function{ApproxMinEvec}{$\mtx{M}$; $q$}
		\State	$\vct{v}_1 \gets \texttt{randn}(n, 1)$
		\State	$\vct{v}_1 \gets \vct{v}_1 / \norm{\vct{v}_1}$
			\Comment	Store $\vct{v}_1$ for reuse
		\For{$i \gets 1, 2, 3, \dots, \min\{q, n - 1\}$}
			\Comment	During loop, store only $\vct{v}_{i}$ and  $\vct{v}_{i+1}$
		\State	${\omega_i} \gets \real( \vct{v}_i^* (\mtx{M} \vct{v}_{i}) )$
		\State	$\vct{v}_{i+1} \gets \mtx{M} \vct{v}_{i} - {\omega_i} \vct{v}_i - {\rho_{i-1}} \vct{v}_{i-1}$
			\Comment	Three-term Lanczos recurrence; ${\rho_0} \vct{v}_0 = \vct{0}$
		\State	${\rho_i \gets \norm{\vct{v}_i}}$
		\If{$\smash{{\rho_i}} = 0$} \textbf{break}
			\Comment	Found an invariant subspace! \EndIf{}
		\State	$\vct{v}_{i+1} \gets \vct{v}_{i+1} / {\rho_i}$
		\EndFor
		\State	$\mtx{T} \gets \texttt{tridiag}( \vct{\rho}_{1:(i-1)}, \vct{{\omega}}_{1:i}, \vct{{\rho}}_{1:(i-1)} )$
			\Comment	Form tridiagonal matrix
		\State	$[ \xi, \vct{u} ] \gets \texttt{MinEvec}(\mtx{T})$
			\Comment	Exploit tridiagonal form~\cite[Ch.~7]{Par98:Symmetric-Eigenvalue}, or just use \texttt{eig}
		\State	$\vct{v} \gets \sum_{j=1}^i u_j \vct{v}_j$
			\Comment	Modify lines 4--9 to regenerate $\vct{v}_2,\dots,\vct{v}_i$ and form sum
	\EndFunction

  \vspace{0.25pc}

\end{algorithmic}
\end{algorithm}

\section{Sketching and reconstruction of a psd matrix}
\label{sec:nystrom}

The \CGAL\ iteration generates a psd matrix that solves the model problem~\cref{eqn:model-problem}
via a sequence~\cref{eqn:cgal-primal-update} of rank-one linear updates.
To control storage costs, \sCGAL\ retains only a compressed version
of the psd matrix variable $\mtx{X}_t$.  This section outlines the \emph{Nystr{\"o}m sketch},
an established method~\cite{HMT11:Finding-Structure,Git13:Topics-Randomized,LLS+17:Algorithm-971,TYUC17:Fixed-Rank-Approximation}
that can track the evolving psd matrix and then report a provably accurate low-rank approximation.

For more information about the implementation and behavior
of low-rank matrix approximation from streaming data,
see \cref{sec:nystrom-supp} and our papers~\cite{TYUC17:Fixed-Rank-Approximation,TYUC17:Practical-Sketching,
TYUC19:Streaming-Low-Rank}. %

\subsection{Sketching and updates}

Consider a psd input matrix $\mtx{X} \in \Sym_n$.
Let $R$ be a parameter that
modulates the storage cost of the sketch and the quality of the matrix approximation.

To construct the Nystr{\"o}m sketch, we draw and fix a standard normal%
\footnote{Each entry of the matrix is an independent Gaussian random variable
with mean zero and variance one.  In the complex setting, the real and imaginary
parts of each entry are independent standard normal variables.}
test matrix $\mtx{\Omega} \in \F^{n \times R}$.
Our summary, or \emph{sketch}, of the matrix $\mtx{X}$ takes the form
\begin{equation} \label{eqn:sketch}
\mtx{S} = \mtx{X\Omega} \in \F^{n \times R}.
\end{equation}
The sketch $\mtx{S}$ supports linear rank-one updates to $\mtx{X}$.
Indeed, we can track the evolution
\begin{equation} \label{eqn:sketch-update}
\begin{aligned}
\mtx{X} &\gets (1 - \eta)\, \mtx{X} + \eta \,\vct{vv}^*
\qquad\qquad\text{for $\eta \in[0,1]$ and $\vct{v}\in\F^n$} \\
\quad\text{via}\quad
\mtx{S} &\gets (1 - \eta)\,\mtx{S} + \eta \, \vct{v} (\vct{v}^* \mtx{\Omega}). %
\end{aligned}
\end{equation}
The test matrix $\mtx{\Omega}$ and the sketch $\mtx{S}$ require
storage of $2Rn$ numbers in $\F$.  The arithmetic cost of the linear
update~\cref{eqn:sketch-update} to the sketch is $\Theta(Rn)$ numerical operations.

\begin{remark}[Structured random matrices]
We can reduce storage costs by a factor of two by using a structured
random matrix in place of $\mtx{\Omega}$.
For example, see~\cite[Sec.~3]{TYUC19:Streaming-Low-Rank}
or~\cite{SGTU18:Tensor-Random}.
\end{remark}

\subsection{The reconstruction process}

Given the test matrix $\mtx{\Omega}$
and the sketch $\mtx{S} = \mtx{X\Omega}$,
we form a rank-$R$ approximation $\widehat{\mtx{X}}$
of the sketched matrix $\mtx{X}$. %
The approximation is defined by %
\begin{equation} \label{eqn:nystrom}
\widehat{\mtx{X}} := \mtx{S} (\mtx{\Omega}^* \mtx{S})^\dagger \mtx{S}^*
	= (\mtx{X\Omega}) (\mtx{\Omega}^* \mtx{X \Omega})^{\dagger} (\mtx{X \Omega})^*,
\end{equation}
where ${}^\dagger$ is the pseudoinverse.
This reconstruction is called a \emph{Nystr{\"o}m approximation}.
We often truncate $\widehat{\mtx{X}}$ %
by replacing it with
its best rank-$r$ approximation $\lowrank{\widehat{\mtx{X}}}{r}$ for some $r \leq R$.

See~\cref{alg:nystrom-sketch} for a numerically stable
implementation of the Nystr{\"o}m approximation~\cref{eqn:nystrom}. %
The algorithm takes $\Theta(R^2 n)$ numerical operations
and $\Theta(Rn)$ storage.

\begin{algorithm}[t]%
  \caption{\textsf{NystromSketch} implementation (see \cref{sec:nystrom})
  \label{alg:nystrom-sketch}}
  \begin{algorithmic}[1]
  \vspace{0.5pc}
	\footnotesize

	\Require{Dimension $n$ of input matrix, size $R$ of sketch } %
    \Ensure{Rank-$R$ approximation $\widehat{\mtx{X}}$ of sketched matrix in factored form
    $\widehat{\mtx{X}} = \mtx{U\Lambda U}^*$,
    where $\mtx{U} \in \F^{n \times R}$ has orthonormal columns and $\mtx{\Lambda} \in \R^{R \times R}$
    is nonnegative diagonal}

	\vspace{0.5pc}

	\Function{\textsf{NystromSketch.Init}}{$n$, $R$}
	\State	$\mtx{\Omega} \gets \texttt{randn}(n, R)$
		\Comment	Draw and fix random test matrix
	\State	$\mtx{S} \gets \texttt{zeros}(n, R)$ %
		\Comment	Form sketch of zero matrix
  	\EndFunction

	\vspace{0.5pc}

	\Function{\textsf{NystromSketch.RankOneUpdate}}{$\vct{v}$, $\eta$}
		\Comment	Implements~\cref{eqn:sketch-update}
	\State	$\mtx{S} \gets	(1-\eta)\, \mtx{S} + \eta\, \vct{v} (\vct{v}^* \mtx{\Omega})$
		\Comment	Update sketch of matrix
	\EndFunction

	\vspace{0.5pc}

	\Function{\textsf{NystromSketch.Reconstruct}}{{}} %
	\State	$\sigma \gets \sqrt{n} \, \texttt{eps}(\texttt{norm}(\mtx{S}))$ %
		\Comment	Compute a shift parameter
	\State	$\mtx{S}_{\sigma} \gets \mtx{S} + \sigma \, \mtx{\Omega}$
		\Comment	Implicitly form sketch of $\mtx{X} + \sigma \, \Id$
	\State	$\mtx{L} \gets \texttt{chol}(\mtx{\Omega}^* \mtx{S}_{\sigma})$
	\State	$[\mtx{U}, \mtx{\Sigma}, \sim] \gets \texttt{svd}(\mtx{S}_{\sigma}/\mtx{L})$
		\Comment	Dense SVD
	\State	$\mtx{\Lambda} \gets \max\{0, \mtx{\Sigma}^2 - \sigma \, \Id \}$
		\Comment	Remove shift

	\EndFunction

  \vspace{0.25pc}

\end{algorithmic}
\end{algorithm}

\subsection{\textit{A priori} error bounds}
\label{sec:nystrom-apriori}

The Nystr{\"o}m approximation $\widehat{\mtx{X}}$ yields a provably good estimate
for the matrix $\mtx{X}$ contained in the sketch~\cite[Thm.~4.1]{TYUC17:Fixed-Rank-Approximation}.

\begin{fact}[Nystr{\"o}m sketch: Error bound] \label{fact:nystrom}
Fix a psd matrix $\mtx{X} \in \Sym_n$.  Let $\mtx{S} = \mtx{X\Omega}$
where $\mtx{\Omega} \in \F^{n \times R}$ is standard normal.
For each {$r < R-1$}, the Nystr{\"o}m approximation~\cref{eqn:nystrom} satisfies
\begin{equation} \label{eqn:nystrom-error}
\Expect_{\mtx{\Omega}} \norm{ \mtx{X} - \widehat{\mtx{X}} }_* %
	\leq \left(1 + \frac{r}{R-r-1}\right) \cdot \norm{ \mtx{X} - \lowrank{\mtx{X}}{r} }_*,
\end{equation}
where $\Expect_{\mtx{\Omega}}$ is the expectation with respect to $\mtx{\Omega}$.
If we replace $\widehat{\mtx{X}}$ with its rank-$r$ truncation $\lowrank{ \widehat{\mtx{X}} }{r}$,
the error bound~\cref{eqn:nystrom-error} remains valid.
Similar results hold with high probability.
\end{fact}

\section{Scalable semidefinite programming via \sCGAL}
\label{sec:sketchy-cgal}

We may now present an extension of \CGAL\
that solves the model problem~\cref{eqn:model-problem}
with controlled storage and computation.
Our new algorithm, \sCGAL, enhances %
the \CGAL\ iteration from~\cref{sec:cgal-strategy}. %
Instead of storing
the matrix $\mtx{X}_t$ in
the \CGAL\ iteration, %

\begin{enumerate}
\item	We drive the iteration with the $d$-dimensional
primal state variable $\vct{z}_t := \mathcal{A} \mtx{X}_t$. %

\item	We maintain a Nystr{\"o}m sketch of the primal iterate $\mtx{X}_t$
using storage $\Theta(Rn)$.

\item	When the iteration is halted, say at time $T$,
we use the sketch to construct a rank-$R$ approximation
$\widehat{\mtx{X}}_T$ of the implicitly computed solution $\mtx{X}_T$ of the model problem.
\end{enumerate}

\noindent
As we will see, the resulting method exhibits almost the same
convergence behavior as the \CGAL\ algorithm,
but it also enjoys a strong storage guarantee
(when $d \ll n^2$ and $R \ll n$).

\subsection{The \sCGAL\ iteration}
\label{sec:sketchy-cgal-iteration}

To develop the \sCGAL\ iteration,
we start with the \CGAL\ iteration.
Then we make the substitutions $\vct{z}_t = \mathcal{A}\mtx{X}_t$
and $\vct{h}_t = \mathcal{A}\mtx{H}_t$ to eliminate the matrix variables.
Let us summarize what happens; see~\cref{sec:loop-invariants} for additional explanation.
\Cref{alg:sketchy-cgal} contains pseudocode with implementation recommendations.

\subsubsection{Initialization}

First, we choose the size $R$ of the Nystr{\"o}m sketch.
Then draw and fix the random test matrix $\mtx{\Omega} \in \F^{n \times R}$.
Select an initial smoothing parameter $\beta_0$ and a (finite) bound $K$ on the dual variable.
Define the smoothing parameters $\beta_t := \beta_0 \sqrt{t+1}$
and the step size parameters $\eta_t := 2 / (t+1)$.
Initialize the primal state variable $\vct{z}_1 := \vct{0} \in \R^d$,
the sketch $\mtx{S}_1 := \mtx{0} \in \F^{n\times R}$,
and the dual variable $\vct{y}_1 := \vct{0} \in \R^d$.
These choices correspond to the simplest initial iterate $\mtx{X}_1 := \mtx{0}$.

\subsubsection{Primal updates}

At iteration $t = 1, 2, 3, \dots$, we compute a unit-norm vector $\vct{v}_t$
that is an approximate minimum eigenvector of the gradient $\mtx{D}_t$
of the smoothed objective:
\begin{equation} \label{eqn:sketchy-cgal-lmo-approx}
\xi_t :=
\vct{v}_t^* \mtx{D}_t \vct{v}_t \leq \lambda_{\min}(\mtx{D}_t) + \frac{\norm{\mtx{D}_t}}{\sqrt{t+1}}
\quad\text{where}\quad
\mtx{D}_t := \mtx{C} + \mathcal{A}^* (\vct{y}_t + \beta_t (\vct{z}_t - \vct{b})).
\end{equation}
This calculation corresponds with~\cref{eqn:cgal-lmo-approx}.
Form the primal update direction $\vct{h}_t = \mathcal{A}(\alpha \vct{v}_t \vct{v}_t^*)$, %
and then update the primal state variable $\vct{z}_t$ and the sketch $\mtx{S}_t$:
\begin{align} \label{eqn:sketchy-cgal-update}
\vct{z}_{t+1} := (1 - \eta_t) \vct{z}_t + \eta_t \vct{h}_t \quad\text{and}\quad %
\mtx{S}_{t+1} := (1 - \eta_t) \mtx{S}_t + \eta_t \alpha \vct{v}_t (\vct{v}_t^* \mtx{\Omega}).
\end{align}
We obtain the update rule for the primal state variable $\vct{z}_t$
by applying the linear map $\mathcal{A}$ to the primal update rule~\cref{eqn:cgal-primal-update}.
The update rule for the sketch $\mtx{S}_t$ follows from~\cref{eqn:sketch-update}.

\subsubsection{Dual updates}

The update to the dual variable takes the form
\begin{equation} \label{eqn:sketchy-cgal-dual-update}
\vct{y}_{t+1} = \vct{y}_t + \gamma_t (\vct{z}_{t+1} - \vct{b})
\end{equation}
where we choose the largest $\gamma_t$ %
that satisfies the conditions
\begin{equation} \label{eqn:sketchy-cgal-dual-step}
\gamma_t \normsq{ \vct{z}_{t+1} - \vct{b} }
	\leq \frac{4 \alpha^2 \beta_0}{(t+1)^{3/2}} \normsq{ \mathcal{A} }
	\quad\text{and}\quad
	0 \leq \gamma_t \leq \beta_0.
\end{equation}
If needed, we set $\gamma_t = 0$ to prevent $\norm{\vct{y}_{t+1}} > K$.
This is the \sCGAL\ iteration.

\begin{algorithm}[t]%
  \caption{\sCGAL\ for the model problem~\cref{eqn:model-problem}
  \label{alg:sketchy-cgal}}
  \begin{algorithmic}[1]
  \vspace{0.5pc}

	\footnotesize
	\Require{Problem data for~\cref{eqn:model-problem} implemented via the primitives~\cref{eqn:primitives},
	sketch size $R$, number $T$ of iterations}
    \Ensure{Rank-$R$ approximate solution to~\cref{eqn:model-problem}
    in factored form $\widehat{\mtx{X}}_T = \mtx{U\Lambda U}^*$,
    where $\mtx{U} \in \F^{n \times R}$ has orthonormal columns and $\mtx{\Lambda} \in \R^{R\times R}$
    is nonnegative diagonal}
	\vspace{0.5pc}

	\Function{\sCGAL}{$R$; $T$}

	\State	Scale problem data (\cref{sec:numerics-scaling})
		\Comment{\textcolor{dkblue}{\textbf{[opt]}} Recommended!}
	\State	$\beta_0 \gets 1$ and $K \gets +\infty$
		\Comment{Default parameters}
	\State	\textsf{NystromSketch.Init}($n$, $R$)
	\State	$\vct{z} \gets \vct{0}_{d}$ and $\vct{y} \gets \vct{0}_d$
	\For{$t \gets 1, 2, 3, \dots, T$}
  		\State	$\beta \gets \beta_0 \sqrt{t+1}$ and $\eta \gets 2 /(t+1)$
		\State	$[\xi, \vct{v}] \gets \textsf{ApproxMinEvec}( \mtx{C} + \mathcal{A}^* (\vct{y} + \beta(\vct{z} - \vct{b})); q_t )$
			\Comment \cref{alg:rand-lanczos} with $q_t = t^{1/4} \log n$
			\Statex	\Comment Implement with primitives~\cref{eqn:primitives}\primone\primtwo!

		\State	$\vct{z} \gets (1 - \eta) \, \vct{z} + \eta \, \mathcal{A}( \alpha \vct{vv}^* )$
			\Comment	Use primitive~\cref{eqn:primitives}\primthree
		\State	$\vct{y} \gets \vct{y} + \gamma (\vct{z} - \vct{b})$
			\Comment	$\gamma$ is the largest solution to~\cref{eqn:sketchy-cgal-dual-step}
		\State	\textsf{NystromSketch.RankOneUpdate}($\sqrt{\alpha} \vct{v}$, $\eta$)
	\EndFor
	\State	$[\mtx{U}, \mtx{\Lambda}] \gets \textsf{NystromSketch.Reconstruct}()$

	\State	$\mtx{\Lambda} \gets \mtx{\Lambda} + (\alpha - \trace(\mtx{\Lambda})) \Id_R / R$
		\Comment	\textcolor{dkblue}{\textbf{[opt]}} Enforce trace constraint in \cref{eqn:model-problem}
	\EndFunction

  \vspace{0.25pc}

\end{algorithmic}
\end{algorithm}

\subsection{Connection with \CGAL}
\label{sec:loop-invariants}

There is a tight connection between the iterates of
\sCGAL\ and \CGAL.
Let $\mtx{X}_1 := \mtx{0}$.  Using the vectors $\vct{v}_t$ computed
in~\cref{eqn:sketchy-cgal-lmo-approx}, define matrices
\begin{equation} \label{eqn:sketchy-cgal-correspondence}
\mtx{H}_t := \alpha \vct{v}_t \vct{v}_t^*
\quad\text{and}\quad
\mtx{X}_{t+1} := (1 - \eta_t) \mtx{X}_t + \eta_t \mtx{H}_t.
\end{equation}
With these definitions, the following loop invariants are in force:
\begin{equation} \label{eqn:loop-invariant}
\vct{h}_t = \mathcal{A} \mtx{H}_t
\quad\text{and}\quad
\vct{z}_t = \mathcal{A} \mtx{X}_t
\quad\text{and}\quad
\mtx{S}_t = \mtx{X}_t \mtx{\Omega}.
\end{equation}
By comparing \cref{sec:sketchy-cgal-iteration,sec:cgal-strategy},
we see that the trajectory $\{ (\mtx{X}_t, \mtx{H}_t, \vct{y}_t) : t = 1, 2, 3, \dots \}$
could also have been generated by running the \CGAL\ iteration.
In other words, the variables in \sCGAL\
track the variables of some invocation of \CGAL\
and inherit their behavior.  We refer to the matrices
$\mtx{X}_t$ as the \emph{implicit} \CGAL\ iterates.

\subsubsection{Approximating the \CGAL\ iterates}

We do not have access to the implicit \CGAL\ iterates %
described above.  Nevertheless, the sketch permits us
to approximate them! %
After iteration $t$ of \sCGAL, we can form a rank-$R$
approximation $\widehat{\mtx{X}}_t$
of the implicit iterate $\mtx{X}_t$ by invoking the
formula~\cref{eqn:nystrom} with $\mtx{S} = \mtx{S}_t$.
According to~\cref{fact:nystrom}, for each $r < R-1$,
\begin{equation} \label{eqn:Xt-hat-err}
\Expect_{\mtx{\Omega}} \norm{ \mtx{X}_t - \widehat{\mtx{X}}_t }_*
	\leq \left( 1 + \frac{r}{R-r-1} \right) \cdot \norm{\mtx{X}_t - \lowrank{\mtx{X}_t}{r} }_*.
\end{equation}
In other words, the computed approximation $\widehat{\mtx{X}}_t$ is a good proxy
for the implicit iterate $\mtx{X}_t$ whenever the latter
matrix is well-approximated by a low-rank matrix.
The same bound~\cref{eqn:Xt-hat-err} holds
if we replace $\widehat{\mtx{X}}_t$ by the truncated
rank-$r$ matrix $\lowrank{\widehat{\mtx{X}}_t}{r}$.

\begin{remark}[Trace correction]
The model problem~\cref{eqn:model-problem}
requires the matrix variable to have trace $\alpha$,
but the computed solution $\widehat{\mtx{X}}_t$ %
rarely satisfies this constraint. 
\Cref{alg:sketchy-cgal} includes an optional projection step (line 13)
that corrects the trace.  This step never increases the error in the Nystr{\"o}m
approximation by more than a factor of two (\cref{sec:trace-correct-supp}). %
Our analysis of \sCGAL\ \emph{does not} include the projection,
but it is valuable in practice.
\end{remark}

\subsubsection{Assessing solution quality}

Given quantities computed by \sCGAL,
we can assess how well the implicit iterate
$\mtx{X}_t$
solves the model problem~\cref{eqn:model-problem}.
The infeasibility is just $\vct{z}_t - \vct{b}$.
To bound suboptimality, we track the objective $p_t := \ip{ \mtx{C} }{ \mtx{X}_t }$
by setting $p_1 = 0$ and using the update
$p_{t+1} = (1-\eta_t) p_t + \eta_t \alpha \vct{v}_t^* (\mtx{C} \vct{v}_t)$.
The surrogate duality gap~\cref{eqn:cgal-gap} takes the form %
$$
g_t(\mtx{X}_t) = p_t + \ip{ \vct{y}_t + \beta_t (\vct{z}_t - \vct{b}) }{ \vct{z}_t } - \lambda_{\min}(\mtx{D}_t).
$$
In practice, we use~\cref{eqn:sketchy-cgal-lmo-approx} to approximate
the minimum eigenvalue: $\lambda_{\min}(\mtx{D}_t) \approx \xi_t$.
In light of \cref{eqn:cgal-posterior-error-bd},
we can bound the suboptimality of $\mtx{X}_t$ 
via the expression
\begin{equation} \label{eqn:scgal-posterior-error-bd}
\ip{ \mtx{C} }{ \mtx{X}_t } - \ip{\mtx{C}}{ \mtx{X}_{\star}}
	\leq g_t(\mtx{X}_t) - \ip{ \vct{y}_t }{ \vct{z}_t - \vct{b} } - \frac{1}{2} \beta_t \norm{ \vct{z}_t - \vct{b} }^2,
\end{equation}
where $\mtx{X}_{\star}$ is a primal optimal point.
See \cref{sec:scgal-soln-quality-supp} for more details.

\subsection{Convergence of \sCGAL}

The implicit iterates $\mtx{X}_t$ converge to a solution
of the model problem~\cref{eqn:model-problem} at the same rate
as the iterates of \CGAL\ would.  On average, the rank-$R$
iterates $\widehat{\mtx{X}}_t$ track the implicit iterates.
The discrepancy between them depends on how well the implicit iterates
are approximated by low-rank matrices.
Here is a simple convergence result
that reflects this intuition;
see~\cref{sec:scgal-converge-supp}
for the easy proof and further results.

\begin{theorem}[\sCGAL: Convergence]
Assume problem~\cref{eqn:model-problem} satisfies strong duality,
and let $\mathsf{\Psi}_{\star}$ be its solution set.
For each {$r < R-1$}, the iterates $\widehat{\mtx{X}}_t$ computed by \sCGAL\ (\cref{sec:sketchy-cgal-iteration,sec:loop-invariants})
satisfy
$$
\limsup_{t\to\infty} \Expect_{\mtx{\Omega}} \dist_*(\widehat{\mtx{X}}_t, \mathsf{\Psi}_{\star})
	\leq \left(1+ \frac{r}{R-r-1}\right) \cdot
	\max_{\mtx{X} \in \mathsf{\Psi}_{\star}} \norm{\mtx{X} - \lowrank{\mtx{X}}{r} }_*.
$$
The same bound holds if we replace $\widehat{\mtx{X}}_t$ with the truncated approximation $\lowrank{\widehat{\mtx{X}}_t}{r}$.
Here, $\dist_*$ is the nuclear-norm distance between a matrix and a set of matrices.
\end{theorem}

\subsection{Resource usage}

The main working variables in the \sCGAL\ iteration
are the primal state $\vct{z}_t \in \R^d$, the computed eigenvector
$\vct{v}_t \in \R^n$, the update direction $\vct{h}_t \in \R^d$,
the test matrix $\mtx{\Omega} \in \R^{n \times R}$, the
sketch $\mtx{S}_t \in \R^{n \times R}$, and the dual variable $\vct{y}_t \in \R^d$.
The total cost for storing these variables is $\Theta(d + Rn)$.

The arithmetic bottleneck in \sCGAL\ comes from the approximate eigenvector
computation~\cref{eqn:sketchy-cgal-lmo-approx};
see \cref{sec:rand-lanczos}
for resource requirements.
The remaining computation takes place in the variable updates.
To form the primal update direction $\vct{h}_t$, %
we invoke the primitive~\cref{eqn:primitives}\primthree.
To update the primal state variable $\vct{z}_t$ and
the sketch $\mtx{S}_t$ in~\cref{eqn:sketchy-cgal-update},
we need $\mathcal{O}(d + Rn)$ arithmetic operations.
No further storage is required at this stage.

Table~\ref{tab:scgal-cost} documents the cost of performing $T$ iterations
of the \sCGAL\ method.
The first column %
summarizes the resources consumed in the outer iteration. %
The second column tabulates the total resources spent to solve
the eigenvalue problem~\cref{eqn:sketchy-cgal-lmo-approx}
via the randomized Lanczos method (\cref{alg:rand-lanczos}).

In light of \cref{fact:cgal-converge,sec:loop-invariants},
we can be confident that the implicit iterate $\mtx{X}_T$ is $\eps$-optimal
within $T = \mathcal{O}(\eps^{-2})$ iterations.
The formula~\cref{eqn:Xt-hat-err} gives \emph{a priori}
guarantees on the quality of the approximation $\widehat{\mtx{X}}_T$.

\begin{center}
\begin{table}
\caption{Resource usage for $T$ iterations of \sCGAL\
with sketch size $R$. %
The algorithm returns a rank-$R$ approximation to
an $\eps$-optimal solution (\cref{sec:approx-soln}) %
to the model problem~\cref{eqn:model-problem}.
\textbf{In theory, $T \sim \eps^{-2}$.}
\textbf{In practice, $T \sim \eps^{-1}$.}
Constants and logarithms are omitted.  See~\cref{sec:scgal-cost}.} %
\label{tab:scgal-cost}
\begin{center}
\begin{tabular}{l|r||r}
& \sCGAL\			& with Lanczos \\ %
& \cref{alg:sketchy-cgal}	& \cref{alg:rand-lanczos} \\
\hline
\textbf{Storage} (floats)		& $d + Rn$ & $n$ \\ %
\hline
\textbf{Arithmetic} (flops)		& $T (d + Rn)$ & $T^{5/4} n$ \\ %
\textbf{Primitives} (calls) && \\
\quad\cref{eqn:primitives}\primone\primtwo & --- & $T^{5/4}$ \\ %
\quad \cref{eqn:primitives}\primthree	& $T$ & ---  \\ %
\hline
\end{tabular}
\end{center}
\end{table}
\end{center}

\subsection{Theoretical performance of \sCGAL}
\label{sec:scgal-cost}

We can package up this discussion in a theorem that describes the theoretical performance
of the \sCGAL\ method. %

\begin{theorem}[\sCGAL] \label{thm:scgal} %
Assume the model problem~\cref{eqn:model-problem} satisfies strong duality.
Fix a rank $r$, and set the sketch size $R \geq \zeta^{-1} r + 1$.
After $T = \mathcal{O}(\eps^{-2})$ iterations,
\sCGAL\ (\cref{sec:sketchy-cgal-iteration,sec:loop-invariants})
returns a rank-$r$ approximation $\widehat{\mtx{X}} = \lowrank{\widehat{\mtx{X}}_T}{r}$
to an $\eps$-optimal point of~\cref{eqn:model-problem}
that satisfies~\cref{eqn:goal} with constant probability.

Suppose we use the storage-optimal Lanczos method (\cref{alg:rand-lanczos})
to solve~\cref{eqn:sketchy-cgal-lmo-approx}.  Then the total storage cost for \sCGAL\ is $\mathcal{O}( d + Rn )$ floats.
The arithmetic requirements are $\mathcal{O}( \eps^{-2}(d + Rn) + \eps^{-5/2} n \log(n/\eps) )$ flops,
$\mathcal{O}(\eps^{-5/2} \log(n/\eps))$ invocations of the primitives \cref{eqn:primitives}\primone\primtwo,
and $\mathcal{O}(\eps^{-2})$ applications of the primitive \cref{eqn:primitives}\primthree.
The constants depend on the problem data $(\mtx{C}, \mathcal{A}, \vct{b}, \alpha)$
and algorithm parameters $(\beta_0, K)$.
\end{theorem}

In practice, \sCGAL\ performs better than \cref{thm:scgal}
suggests: empirically, we find that $T = \mathcal{O}(\eps^{-1})$
for real-world problem instances.

\section{Numerical examples}
\label{sec:numerics}

This section showcases computational experiments
that establish that \sCGAL\ is a practical method
for solving large SDPs.
We show that the algorithm is {flexible} by applying it to several classes of SDPs,
and we show it is {reliable} by solving a large number of instances of each type.
We give empirical evidence that the (implicit) iterates
{converge} to optimality much faster than \cref{thm:scgal} suggests.
Comparisons with other general-purpose SDP solvers demonstrate
that \sCGAL\ is {competitive} for small SDPs,
while it {scales} to problems that standard
methods cannot handle.

\subsection{Setup}

All experiments are performed in \textsc{Matlab\_R2018}a
with double-precision arithmetic. %
Source code is included with the supplementary material.
To simulate the processing power of a personal laptop computer,
we use a single Intel Xeon CPU E5-2630 v3, clocked at 2.40 GHz,
with RAM usage capped at 16 GB.

Arithmetic costs are measured in terms of actual run time. %
\textsc{Matlab} does not currently offer a %
memory profiler, so we externally monitor the total memory allocated. %
We approximate the storage cost by reporting the peak value minus the storage at \textsc{Matlab}'s idle state.

\subsubsection{Problem scaling}
\label{sec:numerics-scaling}

The bounds in the convergence theorem \cref{fact:cgal-converge}
for the implicit iterates of \sCGAL\ depend on problem scaling.  Our analysis motivates us to set
\begin{equation}
\label{eqn:problem-scaling}
\fnorm{\mtx{C}} = \norm{\mathcal{A}} = \alpha = 1
\quad\text{and}\quad
\fnorm{\mtx{A}_1} = \fnorm{\mtx{A}_2} = \dots = \fnorm{\mtx{A}_d}. %
\end{equation}
In our experiments, we enforce the scalings~\cref{eqn:problem-scaling},
except where noted.

\subsubsection{Implementation}

Our experiments require a variant of \sCGAL\ that can handle
inequality constraints; see~\cref{sec:extensions}.
We implement the algorithm with the default parameters ($\beta_0 = 1$ and $K = +\infty$).
The sketch uses a Gaussian test matrix.
The eigenvalue subproblem is solved via randomized Lanczos (\cref{alg:rand-lanczos}).
We include the optional trace normalization (line 13 in \cref{alg:sketchy-cgal})
whenever appropriate.  \textbf{No other tuning is done.}

\subsection{The MaxCut SDP}
\label{sec:maxcut-numerics}

We begin with \textsf{MaxCut} SDPs~\cref{eqn:maxcut-sdp}.
Our goal is to assess the storage and arithmetic costs of
\sCGAL\ for a standard testbed.
We compare with provable solvers for general SDPs:
\textsf{SEDUMI} \cite{S98guide}, \textsf{SDPT3} \cite{TTT99:SDPT3}, \textsf{MOSEK} \cite{mosek}, and \textsf{SDPNAL+} \cite{YST15:SDPNAL}.

\subsubsection{Rounding}
\label{sec:maxcut-rounding}

To extract a cut from an approximate solution to~\cref{eqn:maxcut-sdp},
we apply a simple rounding procedure.
\sCGAL\ returns a matrix $\widehat{\mtx{X}} = \mtx{U\Lambda U}^*$,
where $\mtx{U} \in \R^{n \times R}$ has orthonormal columns.
The columns of the entrywise signum, $\sgn(\mtx{U})$, are the signed indicators of $R$ cuts.
We compute the weights of all $R$ cuts and select the largest.
The other solvers return a full-dimensional solution $\widehat{\mtx{X}}$;
we compute the top $R$ eigenvectors of $\widehat{\mtx{X}}$ using the \textsc{Matlab} command \texttt{eigs}
and invoke the same rounding procedure.

\subsubsection{Datasets}

We consider datasets from two different benchmark groups:
\begin{enumerate} %
\item \textsc{Gset}: 67 binary-valued matrices generated by an autonomous random graph generator and published online \cite{Gset}. The dimension $n$ varies from $800$ to $10\,000$.
\item \textsc{Dimacs10}: This benchmark consists of 150 symmetric matrices (with $n$ varying from $39$ to $50\,912\,018$) chosen for the $10$th~\textsc{Dimacs} Implementation Challenge \cite{DIMACS10}. We consider 148 datasets with dimension $n \leq 24\,000\,000$.
Two problems (\texttt{rgg\_n\_2\_24\_s0} and \texttt{Europe\_osm}) exceeded the 16 GB storage limit.
\end{enumerate}

\subsubsection{Storage and arithmetic comparisons}
\label{sec:maxcut-fig1}

We run each solver for each dataset and measure the storage cost and the runtime. We invoke \sCGAL\ with rank parameter $R = 10$,
and we stop the algorithm when the error bound~\cref{eqn:scgal-posterior-error-bd}
guarantees that the implicit iterates have both {relative} suboptimality and infeasibility below {$10^{-1}$}.
{For other solvers, we set the relative error tolerance to $10^{-1}$.}
See \cref{sec:maxcut-details-supp} for implementation details. 

\Cref{fig:MaxCut-1} displays the results of this experiment. %
The conventional convex solvers do not scale beyond $n = 10^4$ due to their high storage costs.
In contrast, \sCGAL\ handles problems with $n \approx 2.2 \cdot 10^7$.
For 9 datasets, \sCGAL\ did not reach the error tolerance within 5 days.
Further details and results appear in~\cref{sec:maxcut-supp}. %

\begin{figure}[t!]
    \includegraphics[width=0.325\textwidth]{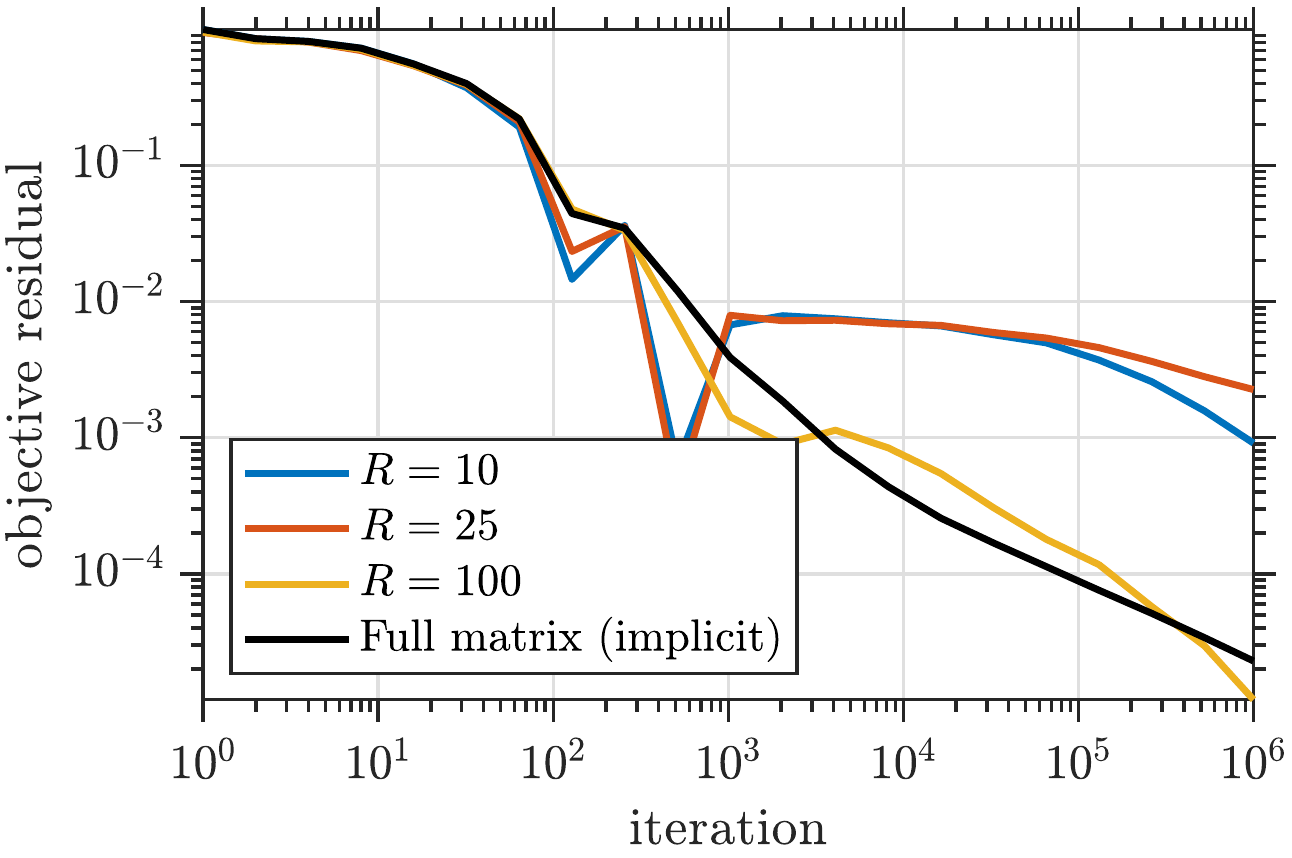}
    \includegraphics[width=0.325\textwidth]{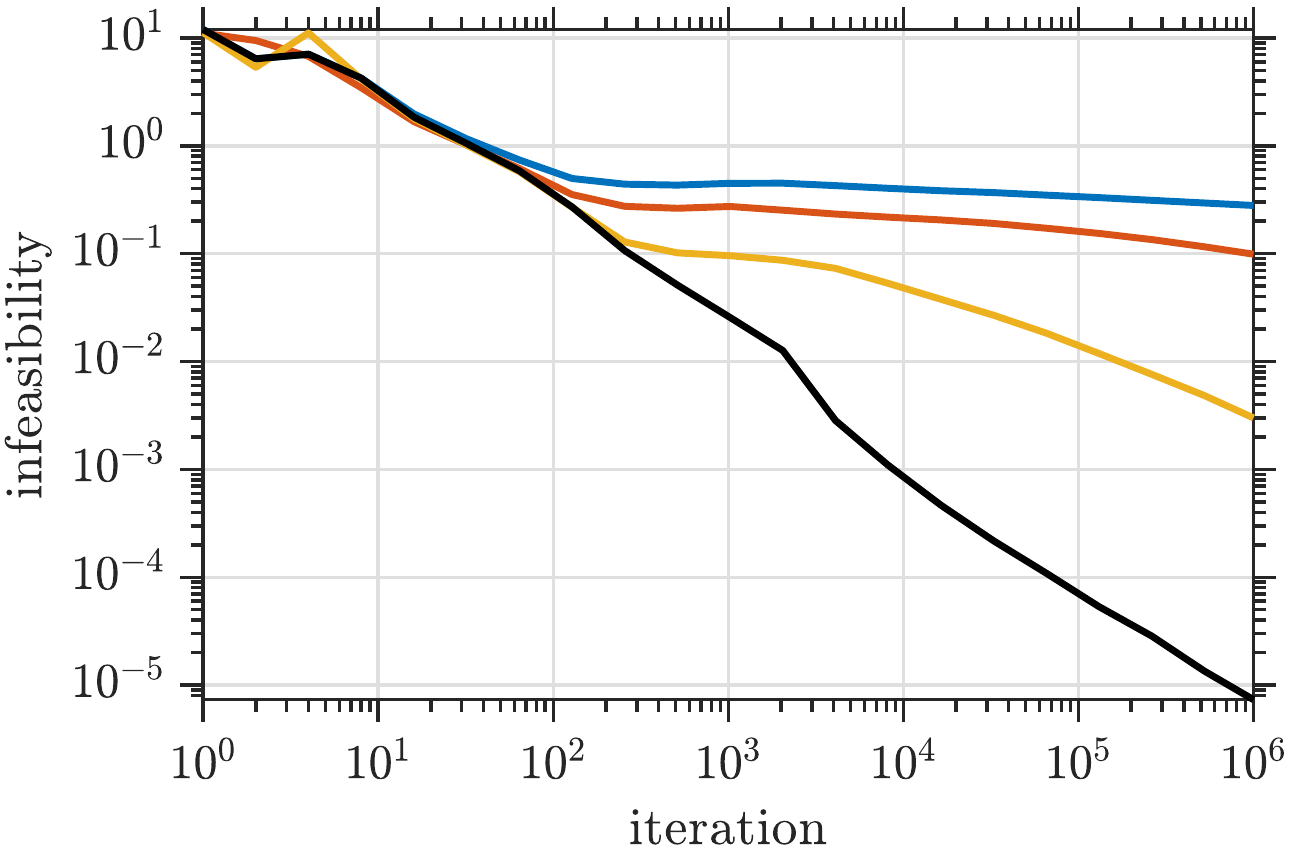}
    \includegraphics[width=0.325\textwidth]{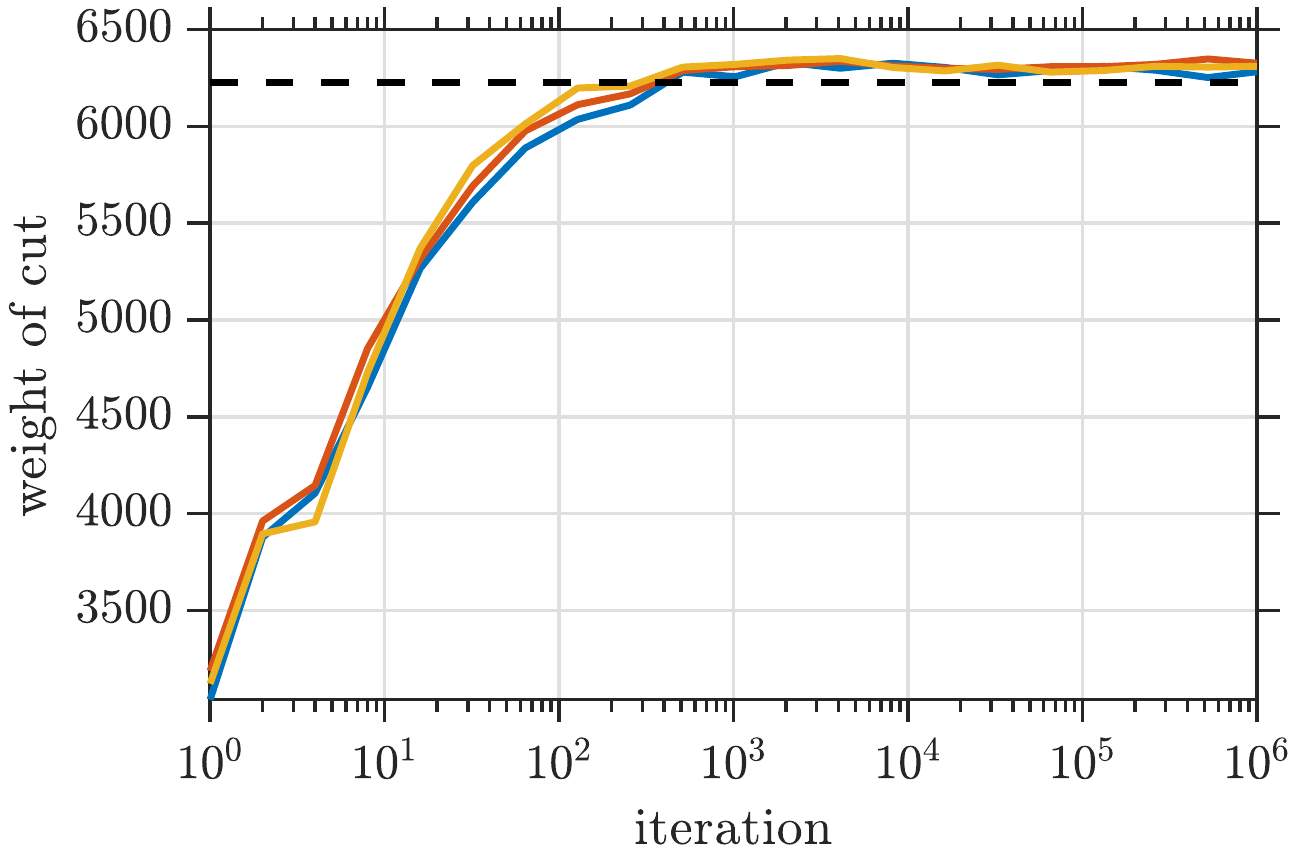} \\[0.5em]
    \includegraphics[width=0.325\textwidth]{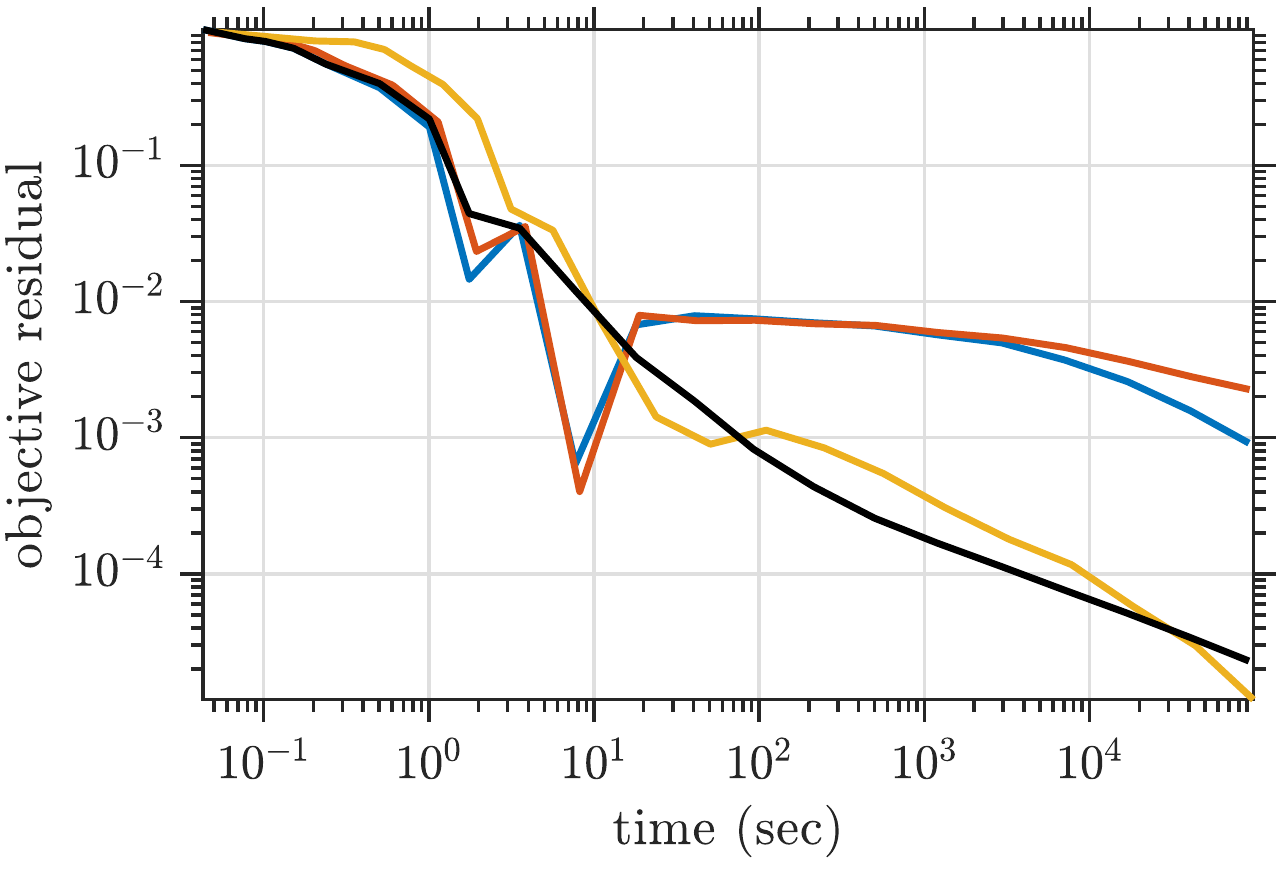}
    \includegraphics[width=0.325\textwidth]{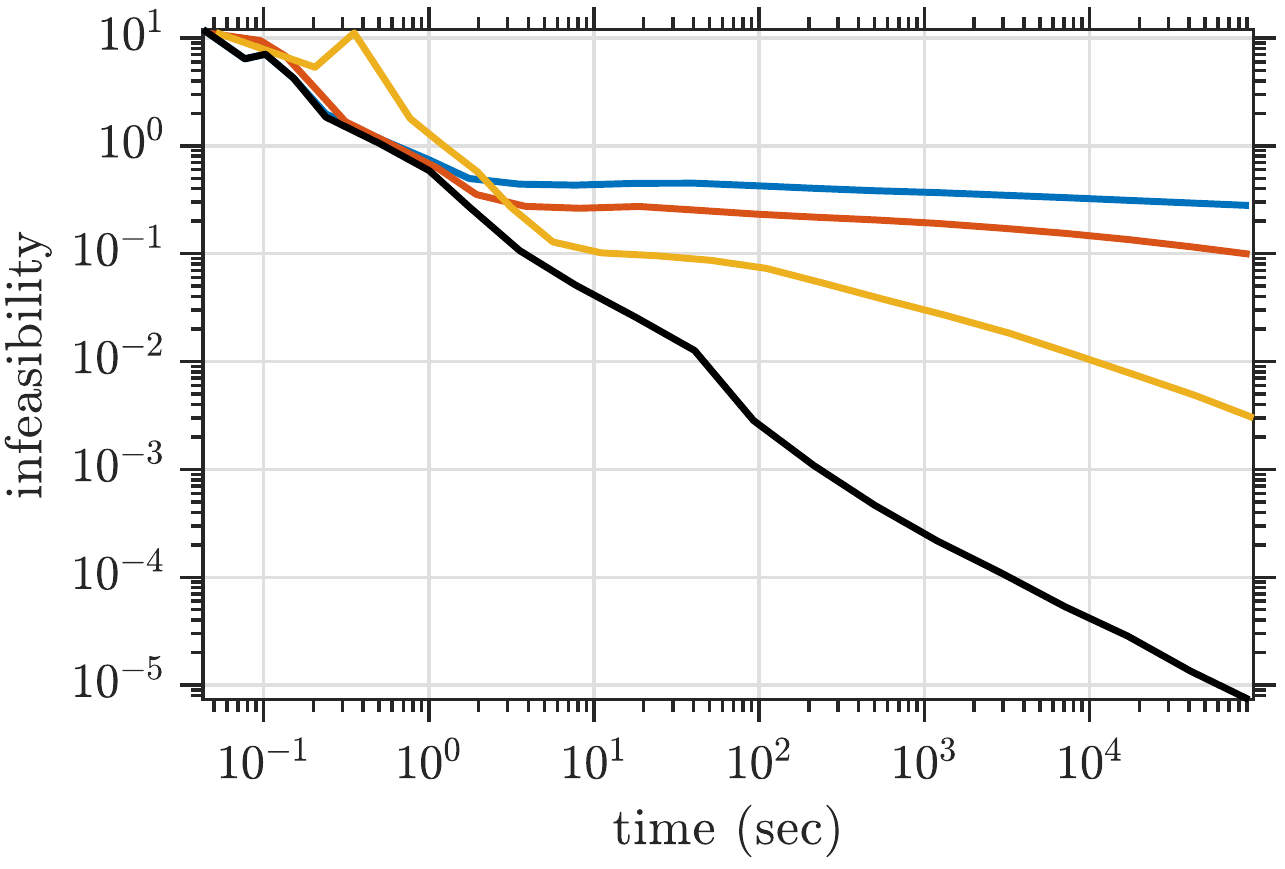}
    \includegraphics[width=0.325\textwidth]{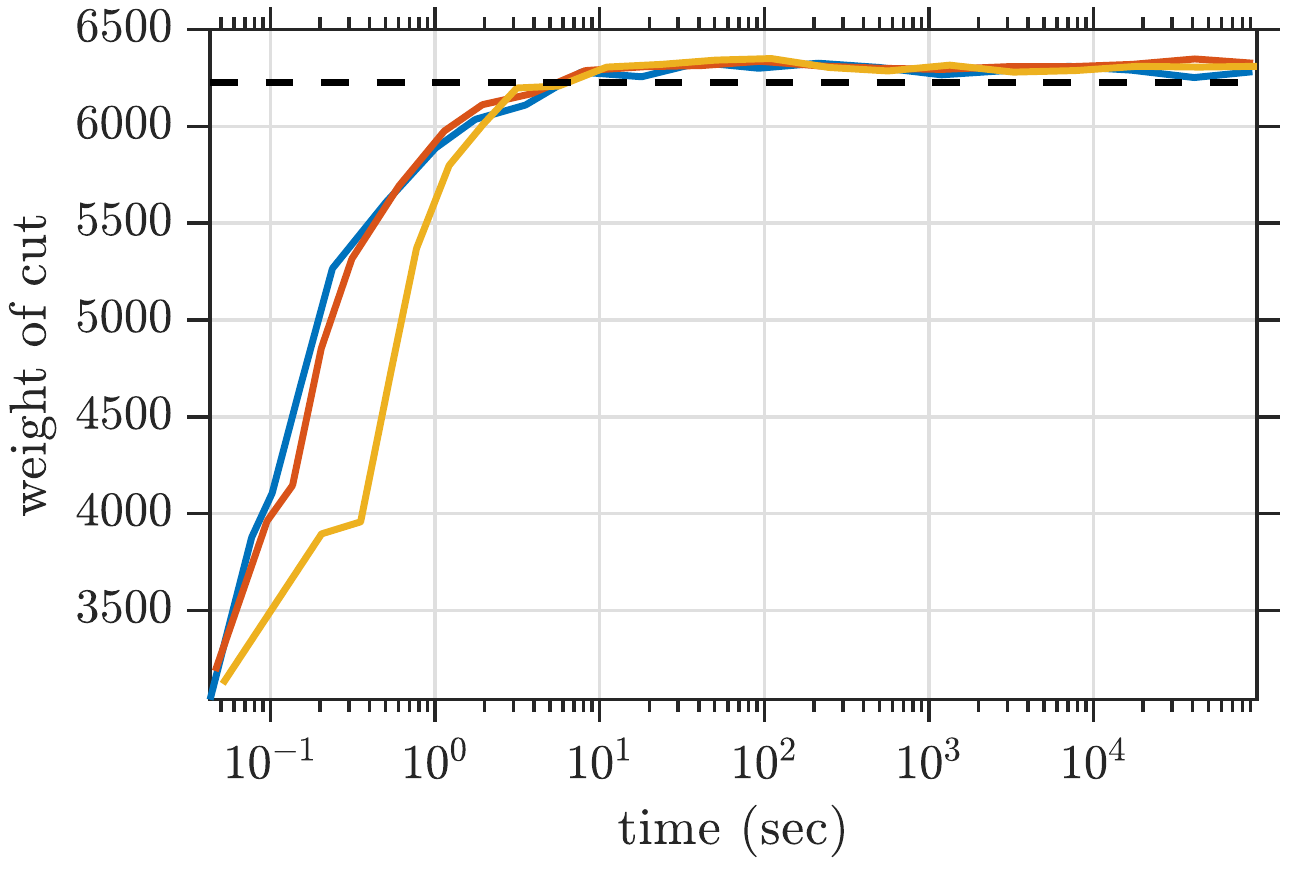}
\caption{\textsf{\textbf{\textsf{MaxCut} SDP: Convergence.}}
We solve the \textsf{MaxCut} SDP for the \textsf{G67} dataset ($n = 10\,000$) with \sCGAL.  The subplots show the suboptimality [left], infeasibility [center], and cut value [right] of the implicit iterates and low-rank approximations for sketch size $R \in \{10,25,100\}$ as a function of iteration [top] and run time [bottom].  The dashed line is the cut value of a high-accuracy \textsf{SDPT3} solution.  See~\cref{sec:maxcut-fig2}.}
\label{fig:MaxCut-2}
\end{figure}

\subsubsection{Empirical convergence rates}
\label{sec:maxcut-fig2}

Next, we investigate the empirical convergence of \sCGAL\ and the effect of the sketch size parameter $R$.  We consider the \textsf{MaxCut} SDP~\cref{eqn:maxcut-sdp} for the \textsf{G67} dataset ($n = 10\,000$), the largest instance in \textsc{Gset}.
We run $10^6$ iterations of \sCGAL\ for each $R \in \{10, 25, 100\}$.

We use a high-accuracy solution from \textsf{SDPT3} to approximate an optimal point $\mtx{X}_{\star}$
of~\cref{eqn:model-problem}.  Given a prospective solution $\mtx{X}$,
we compute its relative suboptimality and feasibility as
$$
\texttt{objective residual} = \frac{\abs{\ip{\mtx{C}}{\mtx{X}} - \ip{\mtx{C}}{\mtx{X}_{\star}}}}{1 + \abs{\ip{\mtx{C}}{\mtx{X}_{\star}}}}
\quad\text{and}\quad
\texttt{infeasibility} = \frac{\norm{\mathcal{A}\mtx{X} - \vct{b}}}{1 + \norm{\vct{b}}}.
$$
It is standard to increment the denominator by one to handle small values gracefully.
These quantities are evaluated with respect to the original (not rescaled) problem data.

\Cref{fig:MaxCut-2} illustrates the convergence of \sCGAL\ for this problem.
After $t$ iterations, the residual and infeasibility decay like $\mathcal{O}(t^{-1})$,
which is far better than the $\mathcal{O}(t^{-1/2})$ bound in \cref{thm:scgal}.
Similar behavior is manifest for other datasets and other problems.

\cref{fig:MaxCut-2} also displays the weight of the cut obtained after rounding, compared with the weight of the cut obtained from the \textsf{SDPT3} solution.
Observe that sketch size $R = 10$ is sufficient, and the \sCGAL\ solutions yield excellent cuts after a few hundred iterations.

\subsubsection{Primal--dual convergence}

We have observed empirically that convergence occurs for the implicit primal
sequence $(\mtx{X}_t)$, the dual sequence $(\vct{y}_t)$,
and the posterior error bound~\cref{eqn:scgal-posterior-error-bd} generated by \sCGAL.
See \cref{sec:primal-dual-conv-supp} for numerical evidence.

\subsubsection{Hard \textsf{MaxCut} instances}

Waldspurger \& Waters~\cite{WW18:RankOptimality} construct instances of
the \textsf{MaxCut} SDP~\cref{eqn:maxcut-sdp} that are challenging for
algorithms based on the Burer--Monteiro~\cite{Burer2003} factorization heuristic
(\cref{sec:bm}).
Each instance has a unique solution and the solution has rank $1$, but Burer--Monteiro
methods require factorization rank $R = \Theta(\sqrt{n})$, resulting
in storage cost $\Theta(n^{3/2})$.
In~\cref{sec:bm-hard}, we give numerical evidence that \sCGAL\
can solve these instances with sketch size $R = 2$, achieving
the optimal storage $\Theta(n)$.  We confirm that Burer--Monteiro
usually fails in the optimal-storage regime.

\subsection{Abstract phase retrieval}
\label{sec:numerics-phase-retrieval}

Phase retrieval is the problem of reconstructing a complex-valued signal from intensity-only measurements.
It arises in interferometry~\cite{Fie82:Phase-Retrieval}, speech processing~\cite{BCE06:Signal-Reconstruction},
array imaging~\cite{CMP10:Array-Imaging}, microscopy~\cite{Horstmeyer_2015}, and many other applications.
We will outline a standard method~\cite{CMP10:Array-Imaging,Horstmeyer_2015} for performing phase retrieval
by means of an SDP. %

This section uses synthetic instances of a phase retrieval SDP to compare the scaling behavior
of \sCGAL\ and \CGAL.  We also consider a third algorithm \textsf{ThinCGAL},
inspired by~\cite{YHC15:ScalablePhaseRetrieval},
that maintains a thin SVD of the matrix variable via rank-one updates~\cite{Bra06:Fast-Low-Rank}.

\subsubsection{Phase retrieval SDPs}

Let $\vct{\chi}_\natural \in \C^n$ be an unknown (discrete) signal. %
For known vectors $\vct{a}_i \in \C^n$, we acquire measurements of the form
\begin{equation}
\label{eqn:phase-retrieval-measurements}
b_i = \abs{\ip{\vct{a}_i}{\vct{\chi}_\natural}}^2
\quad \text{for} \quad i=1,2,3,\dots,d.
\end{equation}
Abstract phase retrieval is the challenging problem of recovering $\vct{\chi}_{\natural}$
from $\vct{b}$.

Let us summarize a lifting approach introduced by Balan et al.~\cite{BBCE09:Painless-Reconstruction}. %
The key idea is to replace the signal vector $\vct{x}_{\natural}$
by the matrix $\mtx{X}_{\natural} = \vct{\chi}_{\natural} \vct{\chi}_{\natural}^*$.
Then rewrite~\cref{eqn:phase-retrieval-measurements} as
\begin{equation} \label{eqn:phase-retrieval-meas-linear}
b_i = \vct{a}_i^*{\mtx{X}_\natural}\vct{a}_i =  \ip{\mtx{A}_i }{\mtx{X}_\natural}
\quad \text{where} \quad \mtx{A}_i = \vct{a}_i \vct{a}_i^*
\quad\text{for $i = 1, 2, 3, \dots, d$.}
\end{equation}
Promoting the implicit constraints on $\mtx{X}_{\natural}$ and forming the $\mtx{A}_i$
into a linear map $\mathcal{A}$, we can express the problem of finding $\mtx{X}_{\natural}$
as a %
feasibility problem with a matrix variable:
$$
\mathrm{find} \quad \mtx{X} \in \Sym_n \quad
\subjto \quad \mathcal{A}\mtx{X} = \vct{b},\quad \mtx{X}\text{~is~psd},\quad \rank(\mtx{X}) = 1.
$$
To reach a tractable convex formulation, we pass to a trace minimization SDP~\cite{Faz02:Matrix-Rank}:
\begin{equation}
\label{eqn:numerics-phase-retrieval-sdp}
\minimize \quad \trace{\mtx{X}} \quad
\subjto \quad \mathcal{A}\mtx{X} = \vct{b}, \quad \mtx{X}\text{~is~psd}, \quad \trace{\mtx{X}} \leq \alpha.
\end{equation}
The parameter $\alpha$ is an upper bound on the signal energy $\norm{\vct{\chi}_\natural}^2$,
which can be estimated from the observed data $\vct{b}$; see \cite{YHC15:ScalablePhaseRetrieval} for the details.
We can solve the SDP~\cref{eqn:numerics-phase-retrieval-sdp} via \sCGAL.

\subsubsection{Rounding}
\label{sec:abstract-phase-retrieval-rounding}

Suppose that we have obtained an approximate solution $\mtx{X}$ to~\cref{eqn:numerics-phase-retrieval-sdp}.
To estimate the signal $\vct{\chi}_{\natural}$,
we form the vector $\vct{\chi} = \sqrt{\lambda} \vct{u}$ where $(\lambda, \vct{u})$ is a maximum
eigenpair of $\mtx{X}$.
Both \sCGAL\ and \textsf{ThinCGAL} return eigenvalue decompositions,
so this step is trivial.  For \CGAL, we use the \textsc{Matlab} function \texttt{eigs} to perform this computation.

\subsubsection{Dataset and evaluation}
\label{sec:apr-dataset}

We consider synthetic phase retrieval instances.
For each $n \in \{ 10^2, 10^3, \dots, 10^6 \}$, we generate $20$ independent datasets as follows. %
First, draw $\vct{\chi}_\natural \in \C^n$ from the complex standard normal distribution.
Then acquire $d = 12n$ phaseless measurements \cref{eqn:phase-retrieval-measurements}
using the coded diffraction pattern model \cite{CLS15:CodedDiffractionPhaseLift};
see \cref{sec:synthetic-phase-supp}.  The induced linear maps $\mathcal{A}$ and $\mathcal{A}^*$
can be applied via the fast Fourier transform (FFT).

The relative error in a signal reconstruction $\vct{\chi}$ is given by
$\min_{\phi\in\R} \norm{\econst^{\mathrm{i}\phi}\vct{\chi} - \vct{\chi}_\natural}/\norm{\vct{\chi}_\natural}$.
In the SDP~\cref{eqn:numerics-phase-retrieval-sdp}, we set $\alpha = 3n$
to demonstrate insensitivity of the algorithm to the choice of $\alpha$. %

\begin{figure} %
    \centering
        \includegraphics[height=4.8cm]{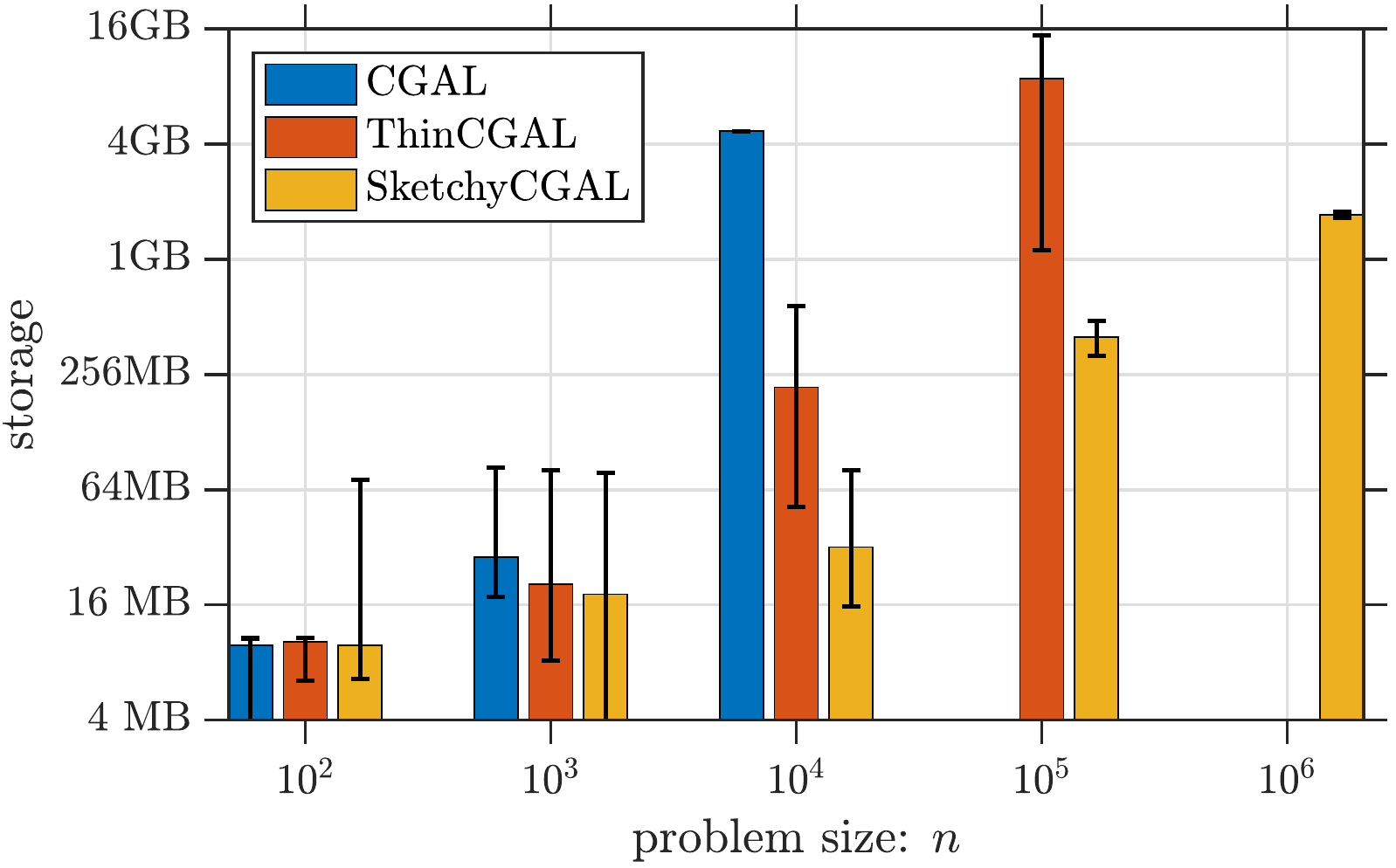}
        \hfill
        \includegraphics[height=4.8cm]{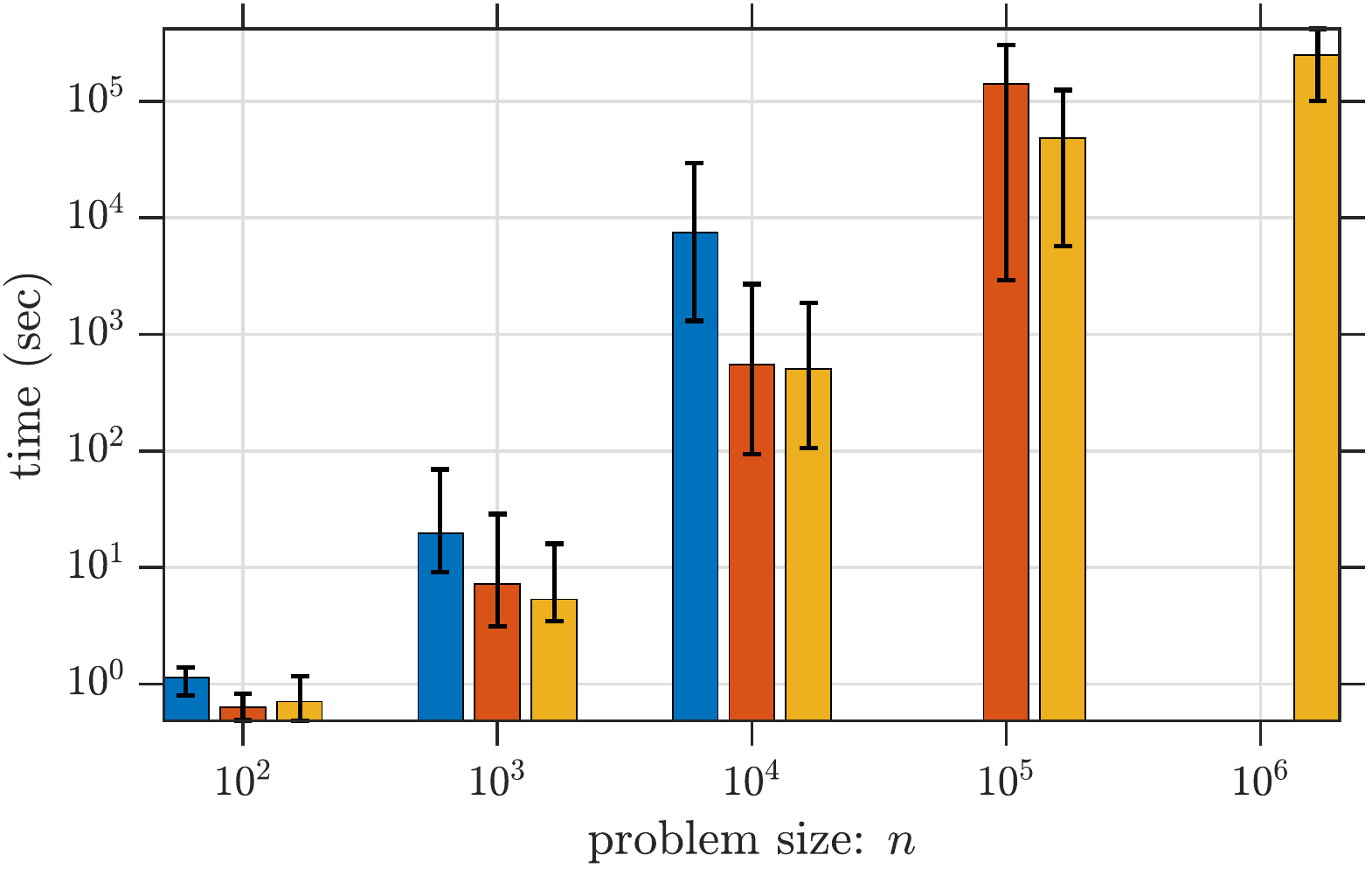}
\caption{\textsf{\textbf{Phase retrieval SDP: Scalability.}}
Storage cost {[left]} and runtime {[right]} to solve random instances
with algorithms \CGAL, \textsf{ThinCGAL}, and \sCGAL.  The height of each bar is the mean; the interval marks the minimum and maximum over 20 trials. Missing bars indicate total failure. See~\cref{sec:apr}.}
\label{fig:abstract-phase-retrieval}
\end{figure}

\subsubsection{Storage and arithmetic comparisons}
\label{sec:apr}

For each algorithm, we report the storage cost and runtime required to produce
a signal estimate $\vct{\chi}$ with (exact) relative error below $10^{-2}$.
We invoke \sCGAL\ with sketch size parameter $R = 5$.

\Cref{fig:abstract-phase-retrieval} displays the outcome.
We witness the benefit of sketching for both storage and arithmetic.
\CGAL~fails for all large-scale instances $(n = 10^5$ and $10^6)$ due to storage allocation.
For the same reason, \textsf{ThinCGAL} fails for $n = 10^6$.
\sCGAL\ successfully solves all problem instances to the target accuracy.

\subsection{Phase retrieval in microscopy}

Next, we study a more realistic phase retrieval problem that arises
from a type of microscopy system~\cite{ZHY13:Wide-Field-High-Resolution}
called Fourier ptychography (FP).  Phase retrieval SDPs
offer a potential approach to FP imaging~\cite{Horstmeyer_2015}.
This section shows that \sCGAL\ can successfully
solve the difficult phase retrieval SDPs that arise from FP.

\subsubsection{Fourier ptychography}

FP microscopes circumvent the physical limits
of a simple lens to achieve high-resolution and wide field-of-view
simultaneously~\cite{ZHY13:Wide-Field-High-Resolution}.
To do so, an FP microscope illuminates a sample
from many angles and uses a simple lens to collect
low-resolution intensity-only images.
The measurements are low-pass filters, whose transfer functions %
depend on the lens and the angle of illumination~\cite{Horstmeyer_2015}.
From the data, we form a high-resolution
image by solving %
a phase retrieval problem;
e.g., via the SDP~\cref{eqn:numerics-phase-retrieval-sdp}.

The high-resolution image of the sample
is represented by a Fourier-domain vector
$\vct{\chi}_{\natural} \in \C^n$.
We acquire $d$ intensity-only measurements of
the form~\cref{eqn:phase-retrieval-measurements},
where $d$ is the total number of pixels in the
low-resolution illuminations.
The low-pass measurements are encoded in vectors $\vct{a}_i$. %
The operators $\mathcal{A}$ and $\mathcal{A}^*$,
built from the matrices $\mtx{A}_i = \vct{a}_i\vct{a}_i^*$, %
can be applied via the FFT.

\subsubsection{Dataset and evaluation}

The authors of \cite{Horstmeyer_2015} provided transmission matrices $\mtx{A}_i$ of a working FP system.
We simulate this system in the computer environment to acquire noiseless intensity-only measurements
of a high-resolution target image~\cite{Con19:Marburg-Virus}. %
In this setup, $\mtx{\chi}_\natural \in \C^n$ corresponds to the Fourier transform of
an $n = 320^2 = 102\,400$ pixel grayscale image.
We normalize $\mtx{\chi}_\natural$ so that $\norm{\mtx{\chi}_\natural} = 1$.
We acquire $225$ low-resolution illuminations of the original image, each with $64^2 = 4\,096$ pixels.
The total number of measurements is $d = 921\,600$.  %

We evaluate the error of an estimate $\vct{\chi}$ as in \cref{sec:apr-dataset}.
Although we know that the signal energy equals $1$, it is more realistic
to approximate the signal energy by setting $\alpha = 1.5$
in the SDP~\cref{eqn:numerics-phase-retrieval-sdp}.

\begin{figure} %
    \centering
    \subfloat[$t = 10\quad$ (69 sec)]{\label{fig:Ptychography-10}\includegraphics[scale=0.44]{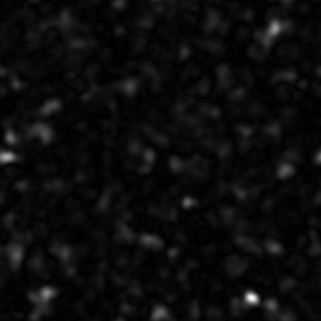} } ~
    \subfloat[$t = 100\quad$ (1\,063 sec)]{\label{fig:Ptychography-100}\includegraphics[scale=0.44]{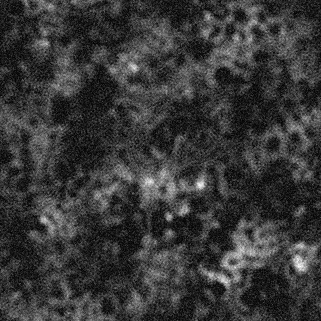} } \\
    \subfloat[$t = 1\,000\quad$ (18\,398 sec)]{\label{fig:Ptychography-1000}\includegraphics[scale=0.44]{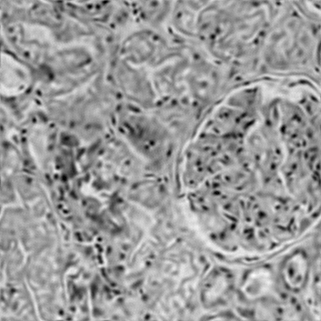} } ~
    \subfloat[$t = 10\,000\quad$ (209\,879 sec)]{\label{fig:Ptychography-10000}\includegraphics[scale=0.44]{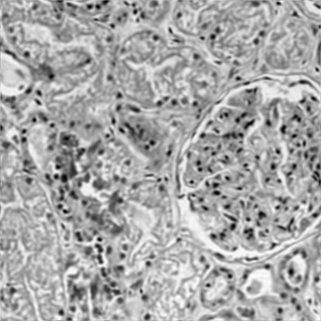} } ~
    \subfloat[original]{\fboxsep=0mm\fboxrule=1pt\label{fig:Ptychography-original}\fcolorbox{red}{white}{\includegraphics[scale=0.44]{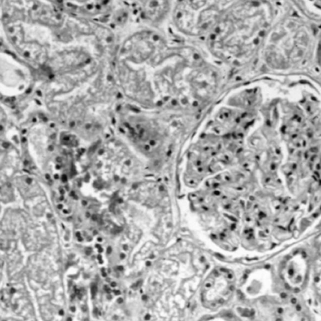}}} \\
    \caption{\textbf{\textsf{Phase retrieval SDP: Imaging.}}
    Reconstruction of an $n = 320^2$ pixel image from Fourier ptychography data.  We solve an $n \times n$ phase retrieval SDP via \sCGAL\ with rank parameter $R = 5$
    and show the images obtained at iterations $t = 10^3, 10^4$. %
    The last subfigure is the original.  See~\cref{sec:FP}.}
\label{fig:Ptychography}
\end{figure}

\subsubsection{FP imaging} %
\label{sec:FP}

We solve the phase retrieval SDP~\cref{eqn:numerics-phase-retrieval-sdp}
by performing $10\,000$ iterations of \sCGAL\ with rank parameter $R = 5$.
The top eigenvector of the output gives an approximation
$\vct{\chi} \in \C^n$ of the signal. %
The inverse Fourier transform of $\vct{\chi}$ is the desired image.

\Cref{fig:Ptychography} displays the image reconstruction after $t \in \{10, 10^2, 10^3, 10^4\}$ iterations.
We obtain a good-quality result in $5$ hours after $1\,000$ iterations;
a sharper image emerges in $59$ hours after $10\,000$ iterations are complete.
We believe the computational time can be reduced substantially with a parallel or GPU implementation.
Nevertheless, it is gratifying that we have solved a difficult SDP whose matrix variable
has over $10^{10}$ entries.  See~\cref{sec:FP-supp} for a larger FP instance with $n = 640^2$ pixels.

\subsection{The quadratic assignment problem}
\label{sec:numerics-qap}

The quadratic assignment problem (QAP) is a very difficult combinatorial
optimization problem that includes the traveling salesman, max-clique, bandwidth problems,
and many others as special cases~\cite{LMO+07:Survey-Quadratic}. %
SDP relaxations offer a powerful approach for obtaining good solutions
to large QAP problems~\cite{Zhao1998}.  In this section, we demonstrate that
\sCGAL\ can solve these challenging SDPs.

\subsubsection{QAP}

We begin with the simplest form of the QAP.
Fix symmetric $\mathtt{n} \times \mathtt{n}$ matrices $\mtx{A}, \mtx{B} \in \mathbb{S}_{\mathtt{n}}$
where $\mathtt{n}$ is a natural number.  We wish to ``align'' the matrices by solving
\begin{equation} \label{eqn:qap}
\minimize \quad \trace( \mtx{A\Pi B\Pi}^* ) \quad
\subjto \quad \text{$\mtx{\Pi}$ is an $\mathtt{n} \times \mathtt{n}$ permutation matrix.}
\end{equation}
Recall that a permutation matrix $\mtx{\Pi}$ has precisely one nonzero entry in each row and
column, and that nonzero entry equals one.  A brute force search over the $\mathtt{n}!$
permutation matrices of size $\mathtt{n}$ quickly becomes intractable as $\mathtt{n}$ grows.
Unsurprisingly, the QAP problem~\cref{eqn:qap}
is~\textsf{NP}-hard~\cite{SG76:P-Complete-Approximation}.
Instances with $\mathtt{n} > 30$ usually cannot
be solved in reasonable time.

\subsubsection{Relaxations}

There is an extensive literature on SDP relaxations for QAP,
beginning with the work~\cite{Zhao1998} of Zhao et al.
We consider a weaker relaxation inspired
by~\cite{HCG14:Scalable-Semidefinite,BravoFerreira2018}:
\begin{equation}
\label{eqn:qap-sdp}
\begin{aligned}
\minimize \quad &\trace[ (\mtx{B} \otimes \mtx{A}) \mathsf{Y} ] \\
\subjto \quad &\trace_1(\mathsf{Y}) = \Id, \quad \trace_2(\mathsf{Y}) = \Id,
	\quad \mathcal{G}(\mathsf{Y}) \geq \mathsf{0}, \\
	&\mathrm{vec}(\mtx{P}) = \diag(\mathsf{Y}), \quad
	\mtx{P}\vct{1} = \vct{1}, \quad \vct{1}^* \mtx{P} = \vct{1}^*,
	\quad \mtx{P} \geq \vct{0}, \\
	&\begin{bmatrix} 1 & \mathrm{vec}(\mtx{P})^* \\
	\mathrm{vec}(\mtx{P}) & \mathsf{Y} \end{bmatrix} \psdge \mtx{0},
	\quad \trace \mathsf{Y} = n. %
\end{aligned}
\end{equation}
We have written $\otimes$ for the Kronecker product.

The constraint $\mathcal{G}(\mathsf{Y}) \geq \mathsf{0}$ enforces
nonnegativity of a subset of the entries in $\mathsf{Y}$.
In the Zhao et al.~relaxation, $\mathcal{G}$ is the identity map,
so it yields $\mathcal{O}(\mathtt{n}^4)$ constraints.
We reduce the complexity %
by choosing $\mathcal{G}$ more carefully.
In our formulation, $\mathcal{G}$ extracts precisely the
nonzero entries of the matrix $\mtx{B} \otimes \vct{11}^*$.
This is beneficial because $\mtx{B}$ is sparse in many applications.
For example, in traveling salesman and bandwidth problems,
$\mtx{B}$ has $\mathcal{O}(\mathtt{n})$ nonzero entries,
so the map $\mathcal{G}$ produces only $\mathcal{O}(\mathtt{n}^3)$
constraints.

The main variable $\mathsf{Y}$ in~\cref{eqn:qap-sdp} has dimension $\mathtt{n}^2 \times \mathtt{n}^2$.
As a consequence, the problem has $\mathcal{O}(\mathtt{n}^4)$ degrees of freedom,
together with $\mathcal{O}(\mathtt{n}^3)$ to $\mathcal{O}(\mathtt{n}^4)$ constraints (depending on $\mathcal{G}$).
The explosive growth of this relaxation scuttles most algorithms
by the time $\mathtt{n} > 50$.  To solve larger instances,
many researchers resort to even weaker relaxations.

In contrast, we can solve the relaxation~\cref{eqn:qap-sdp} directly using \sCGAL,
up to $\mathtt{n} = 150$.
By limiting the number of inequality constraints, via the operator $\mathcal{G}$,
we achieve substantial reductions in resource usage.  We validate our algorithm
on QAPs where the exact solution is known, and we compare the performance with
algorithms \cite{ZBV09:PathFollowing,BravoFerreira2018} for other relaxations.

\subsubsection{Rounding}
\label{sec:qap-round}

Given an approximate solution to~\cref{eqn:qap-sdp}, we use a rounding method to construct a permutation. %
\sCGAL\ returns a matrix $\widehat{\mtx{X}} = \mtx{U\Lambda U}^*$ where $\mtx{U}$ has dimension $(\mathtt{n}^2 + 1) \times R$.
We extract the first column of $\mtx{U}$, discard its first entry
and reshape the remaining part into an $\mathtt{n} \times \mathtt{n}$ matrix.
Then we project this matrix onto the set of permutation matrices via the Hungarian method~\cite{Kuh55:Hungarian-Method,Mun57:Algorithms-Assignment,JK87:Shortest-Augmenting}.
This yields a feasible point $\mtx{\Pi}$ for the problem \cref{eqn:qap}.
We repeat this procedure for all $R$ columns of $\mtx{U}$,
and we pick the one that minimizes $\trace(\mtx{A\Pi B\Pi}^*)$.
This permutation gives an upper bound on the optimal value of \cref{eqn:qap}.

\subsubsection{Datasets and evaluation}
We consider instances from \textsf{QAPLIB} \cite{QAPLIB} and \textsf{TSPLIB} \cite{TSPLIB} that are used in \cite{BravoFerreira2018}.
The optimal values are known, and the permutation size $\mathtt{n}$ varies between $12$ and $150$. %
(Recall that the SDP matrix dimension $n = \mathtt{n}^2 + 1$.)  We report
\begin{equation} \label{eqn:qap-gap}
\texttt{relative gap }\% = \frac{\texttt{upper bound obtained - optimum}}{\texttt{optimum}} \times 100
\end{equation}

\subsubsection{Solving QAPs}
\label{sec:qap-solns}

To solve~\cref{eqn:qap-sdp}, we cannot use the scaling~\cref{eqn:problem-scaling}
because $\norm{\mathcal{A}}$ is not available;
see the source code for our approach.
We apply \sCGAL\ with sketch size $R = \mathtt{n}$,
so the sketch uses storage $\Theta(\mathtt{n}^3)$.
After rounding, a low- or medium-accuracy solution of \cref{eqn:qap-sdp}
often provides a better permutation than a high-accuracy solution.
Therefore, we applied the rounding step at iterations $2,4,8,16,\ldots$\
and tracked the quality of the best permutation attained on the solution path.
We stopped after (the first of) $10^6$ iterations or $72$ hours. %

The results of this experiment appear in \Cref{fig:qap}. %
We compare against the best value reported by Bravo Ferreira et al.\ in \cite[Tables~4 and~6]{BravoFerreira2018}
for their \textsf{CSDP} method with clique size $k \in \{2,3,4\}$.
We also include the results that \cite{BravoFerreira2018} lists for the \textsf{PATH} method~\cite{ZBV09:PathFollowing}.

\sCGAL\ allows us to solve a tighter SDP relaxation of QAP than the other methods (\textsf{CSDP}, \textsf{PATH}).
As a consequence, we obtain significantly smaller gaps for most instances.

\begin{figure}[t!]
    \centering
    \includegraphics[width=\textwidth]{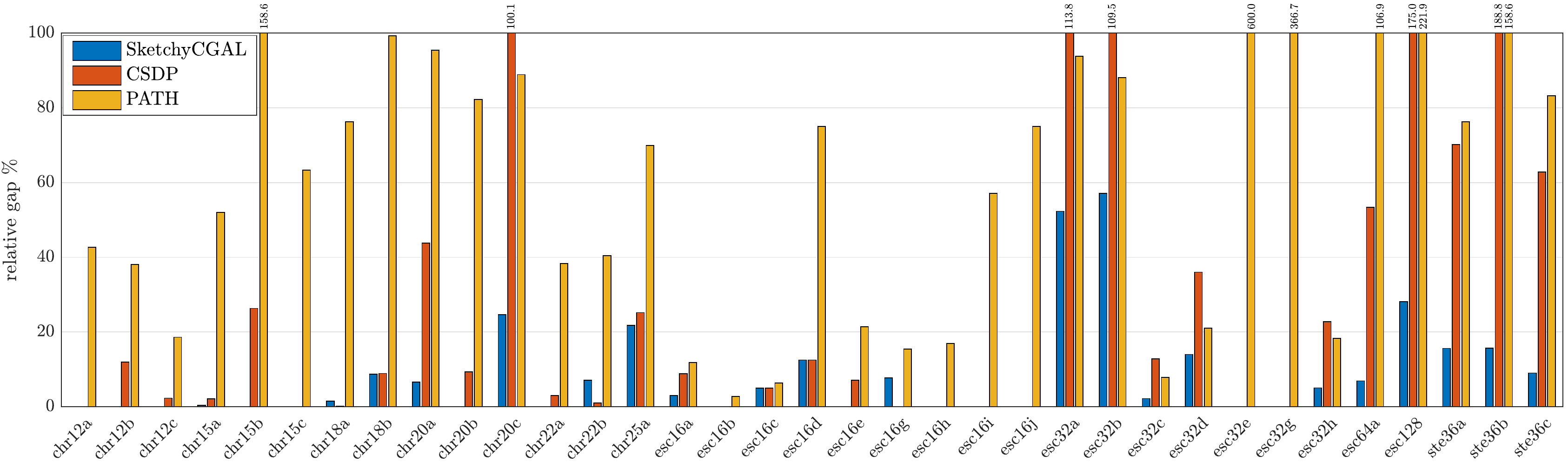}
    \\[0.5em]
    \includegraphics[width=\textwidth]{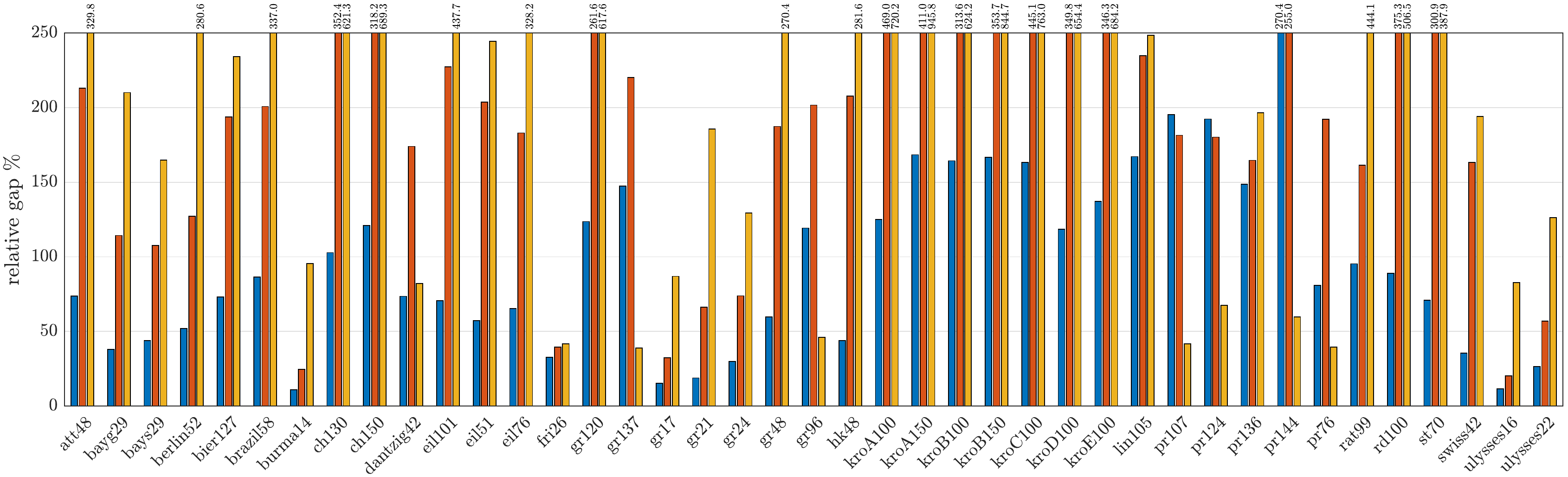}

    \caption{\textsf{\textbf{QAP relaxation: Solution quality.}}
    Using \sCGAL, the \textsf{CSDP} method~\cite{BravoFerreira2018}, and the \textsf{PATH} method~\cite{ZBV09:PathFollowing},
    we solve SDP relaxations of QAP instances from QAPLIB {[top]} and TSPLIB {[bottom]}.  The bars compare the cost of the computed solution,
    against the optimal value; shorter is better.
    See~\cref{sec:qap-solns}.}
\label{fig:qap}
\end{figure}

\section{Related work}
\label{sec:related-work}

There is a large body of literature on numerical methods for solving SDPs;
see~\cite{MHA19:Survey-Recent} for a recent survey.
Here we describe the SDP literature just enough to contextualize our contributions.

\subsection{Standard methodologies}

First, we outline the approaches that drive most of the reliable,
general-purpose SDP software packages that are currently available.

Interior-point methods (IPMs) %
\cite{NN90:Self-concordant,NN94:IP-polynomial,Ali91:IPM-SDP,Ali95:IPM-SDP-combinatorial,KK92:Upper-bounds,KK93:OnL-algorithm}
reformulate the SDP as an unconstrained problem and take an (approximate) Newton step at each iteration.
In exchange, they deliver quadratic convergence. Hence IPMs are widely used to solve SDPs to high precision
\cite{KT05:IPM-SDP-survey,AL12:Handbook-SDP,Nes04:IntroductoryLectures,BN01:ModernConvexOptimization}.
Software packages include \textsf{SeDuMi} \cite{S98guide}, \textsf{MoSeK} \cite{mosek}, and \textsf{SDPT3} \cite{TTT99:SDPT3}.
Alas, IPMs do not scale to large problems:
to solve the Newton system
we must store and factor large, dense matrices. %
A typical IPM for the SDP~\cref{eqn:model-problem}
requires $\Theta(n^3 + d^2n^2 + d^3)$ arithmetic operations per iteration
and $\Theta(n^2 + dn + d^2)$ memory \cite{AHO98:primal-dual-IPM-SDP}.

Several effective SDP solvers are based on the augmented Lagrangian (AL) paradigm
\cite{BV06:Lift-project-relaxations,Wen09:Row-by-row-method,ZSD10:SDPNAL,WGY10:Alternating-direction-SDP}.
In particular, Zhao et al.~\cite{ZSD10:SDPNAL} employ a semi-smooth Newton method and the conjugate
gradient method to solve the subproblems.
Their method is enhanced and implemented in the software package \textsf{SDPNAL+}~\cite{YST15:SDPNAL}.
AL methods for SDPs typically require storage $\Omega(n^2)$.

\subsection{First-order methods}
First-order methods use only gradient information to solve the SDP
to reduce runtime and storage requirements.
We focus on \emph{projection-free} algorithms, suitable large SDPs,
that do not require full SVD computations.
Major first-order methods for SDP include approaches based on
the conditional gradient method (\textsf{CGM}) \cite{FW56:FrankWolfe,LP66:Minimization-Methods,Jon92:Simple-Lemma}
and extensions that handle more complex constraints \cite{Clarkson:2010:CSG:1824777.1824783,Haz08:Sparse-Approximate,pmlr-v28-jaggi13},
primal--dual subgradient algorithms~\cite{Nesterov2009,YTC15:UniversalPD},
and
the matrix multiplicative weight (MMW) method~\cite{Tsuda05matrixexponentiated,AHK05:Fast-Algorithms},
or equivalently, the mirror-prox algorithm with the quantum entropy mirror map~\cite{Nem04:Prox-Method-Rate}.

The standard \textsf{CGM} algorithm does not apply to the model problem \cref{eqn:model-problem} because of the affine constraint $\mathcal{A}\mtx{X} = \vct{b}$.
Several variants~\cite{GPL18:FrankWolfeSplitting,LLM19:NonErgodicAL,silvetifalls2019generalized,YFL+18:HomotopyCGM}
of \textsf{CGM} can handle affine constraints.
In particular, \CGAL~\cite{YFC19:Conditional-Gradient-Based}
does so by applying \textsf{CGM} to an AL formulation.
We have chosen to extend \CGAL\
because of its strong empirical performance
and its robustness to inexact eigenvector computations; see~\cite[Sec.~4-5]{YFC19:Conditional-Gradient-Based}.

Primal--dual subgradient methods perform subgradient ascent on the dual problem;
the cost of each iteration is dominated by an eigenvector computation.
Nesterov~\cite{Nesterov2009} constructs primal iterates by proximal mapping.
The algorithm in Yurtsever et al.~\cite{YTC15:UniversalPD}
constructs primal iterates by an averaging technique, similar to the updates in \textsf{CGM};
this method involves a line search, so it requires very accurate eigenvector calculations. %

The MMW method is derived by
reducing the SDP to a sequence of feasibility problems.
These are reformulated as a primal--dual game
whose dual is an eigenvalue optimization problem.
The resulting MMW algorithm can be interpreted as
performing gradient descent in a dual space and
using the matrix exponential map to transfer information back to the primal space.
To scale this approach to larger problems,
researchers have proposed linearization, random projection, sparsification techniques,
and stochastic Lanczos quadrature to approximate the matrix exponential
\cite{Arora:2016:CPA:2906142.2837020,allenzhu2015using,Peng:2012:FSW:2312005.2312026,GH11:SDP-sublinear-time,GH16:Sublinear-Time,allen-zhu2017follow,baes2013randomized,LP19:Algorithm-Semidefinite,pmlr-v99-carmon19a}.
Even so, the reduction to a sequence of feasibility problems makes this technique impractical for general SDPs.
We are aware of only one computational evaluation of the MMW idea \cite{baes2013randomized}.

\subsection{Storage considerations}

Almost all provably correct %
SDP algorithms store and operate on a full-dimensional matrix variable,
so they are not suitable for very large SDPs.

Some primal--dual subgradient methods and \textsf{CGM} variants build
an approximate solution as a convex combination of rank-one updates,
so the rank of the solution does not exceed the number of iterations.
This fact has led researchers to call these methods ``storage-efficient,''
but this claim is misleading because the algorithms require many iterations to converge.

In the conference paper~\cite{YUTC17:Sketchy-Decisions}, written by a subset of the authors,
we observed that certain types of optimization algorithms can be combined with sketching
to control storage costs.  As a first example, we augmented \textsf{CGM} with sketching to obtain a new algorithm
called \textsf{SketchyCGM}.  This method solves a special class of low-rank matrix
optimization problems that arise in statistics and machine learning.
We believe that \textsf{SketchyCGM} is the first algorithm
for this class of problems that provably succeeds with optimal storage.

To develop \sCGAL, we changed the base optimization algorithm (to \CGAL)
so that we can solve standard-form SDPs.  We switched to a
simpler sketching technique (the Nystr{\"o}m sketch) that has better empirical performance.
We also analyzed how accurately to solve the eigenvalue
problems to ensure that \sCGAL succeeds (\cref{sec:cgal-theory-supp}),
and we deployed an approximate eigenvalue computation method (randomized Lanczos)
that meets our needs.  Altogether, this effort leads to a storage-optimal algorithm
that works for all standard-form SDPs with a minimum of tuning.
Reducing the storage has the ancillary benefit of reducing arithmetic
and communication costs, which also improves scalability.

In concurrent work with Ding \cite{DYC+19:ApproximateComplementarity},
a subset of the authors developed a new \emph{approximate complementarity principle}
that also yields a storage-optimal algorithm for standard-form SDPs.
This approach uses a suboptimal dual point to approximate the range
of the primal solution to the SDP.
By compressing the primal problem to this subspace, we can solve the primal SDP with limited storage.
This method, however, has more limited guarantees than \sCGAL.
A numerical evaluation is in progress.

\subsection{Nonconvex Burer--Monteiro methods} \label{sec:bm}

The most famous approach to low-storage semidefinite programming is the factorization heuristic proposed by Homer and Peinado~\cite{HP97:Design-Performance}
and refined by Burer and Monteiro (BM)~\cite{Burer2003}.
The main idea is to reformulate the model problem~\cref{eqn:model-problem} by
expressing the psd matrix variable $\mtx{X} = \mtx{FF}^*$ in terms of a factor $\mtx{F} \in \R^{n \times R}$,
where the rank parameter $R \ll n$.  That is,
\begin{equation} \label{eqn:factorized-problem}
\minimize\quad \ip{\mtx{C}}{\mtx{FF}^*}
\quad\subjto\quad \mathcal{A}(\mtx{FF}^*) = \vct{b},
\quad \mtx{FF}^* \in \alpha \mtx{\Delta}_n.
\end{equation}
This approach controls storage by sacrificing convexity and the associated guarantees.

Many nonlinear programming methods have been applied to optimize \cref{eqn:factorized-problem}.
AL methods are commonly used~\cite{Burer2003,BM05:LocalMinimaConvergence,pmlr-v2-kulis07a,sahin2019inexact,SGV18:SDPApproxOperatorSplitting}.
The most popular research software based on BM factorization is \textsf{Manopt} \cite{JMLR:v15:boumal14a},
which implements manifold optimization algorithms
including Riemannian gradient and Riemannian trust region methods~\cite{Absil2007}.
Consequently, \textsf{Manopt} is limited to problems where the factorized formulation \cref{eqn:factorized-problem}
defines a smooth manifold.

There has been an intense effort to establish theoretical results for the BM factorization approach.
It is clear that every solution to the SDP \cref{eqn:model-problem} of rank $R$ or less
is also a solution to the factorized problem~\cref{eqn:factorized-problem}.
On the other hand,~\cref{eqn:factorized-problem} may admit local minima
that are not global minima of~\cref{eqn:model-problem}.
Some guarantees are available.  For example,
if $\mtx{C}$ is generic and the constraint set of~\cref{eqn:factorized-problem} is a smooth manifold and
$R \geq \sqrt{2(d+1)}$,
then each second-order critical point of \cref{eqn:factorized-problem} is a global optimum \cite{BVB16:BMSmoothSDP}.
A second-order critical point can be located using a Riemannian trust region method.
See~\cite{BBJ+18:SmootheAnalysisSDP,PJB18:SmoothedAnalysisSDP,Cif19:BurerMonteiroGuarantees} for additional theoretical analysis.

The storage and arithmetic costs of solving~\cref{eqn:factorized-problem} depend on the factorization rank $R$.
Unfortunately, the BM method may fail when $R = o(\sqrt{d})$.
Below this threshold, the BM formulation~\cref{eqn:factorized-problem} can have spurious solutions
(second-order critical points that are not globally optimal),
and the bad problem instances can form a set of positive measure \cite{WW18:RankOptimality}.
Hence the Burer--Monteiro approach cannot support provably correct algorithms
with storage costs better than $\Omega(n\sqrt{d})$.
See~\cref{sec:bm-hard} for numerical evidence.

\section{Conclusion}

We have presented a practical, new approach for solving SDPs at scale.
Our algorithm, \sCGAL, combines a primal--dual optimization method with randomized linear algebra
techniques to achieve unprecedented guarantees when the problem is weakly constrained
and the solution is approximately low rank.
We hope that our ideas lead to further algorithmic advances
and support new applications of semidefinite programming.

\sCGAL\ is currently limited by the arithmetic
cost of solving large eigenvalue problems to increasing accuracy.
It also falters for SDPs with a large number of constraints
because it depends on a primal--dual approach.
Moreover, our analysis does not fully explain
the observed behavior of the algorithm,
including the rate of convergence of the primal variable
or the convergence of the dual variable and the surrogate duality gap.
These topics merit further research.

\appendix

\section{Analysis of the \CGAL\ Algorithm}
\label{sec:cgal-theory-supp}

This section contains a complete analysis of the convergence of the
\CGAL\ algorithm and its arithmetic costs.  For simplicity,
we have specialized this presentation to the model problem that
is the focus of this paper; many extensions are possible.
The convergence results here are adapted from the initial paper~\cite{YFC19:Conditional-Gradient-Based}
on the \CGAL\ algorithm.  The analysis of the approximate eigenvector computation
and the detailed results for the model problem are new.
Empirical work suggests that the analysis is still qualitatively suboptimal,
which is a direction for future research.

\subsection{The model problem}

We focus on solving the optimization template
\begin{equation} \label{eqn:model-problem-supp}
\minimize \quad \ip{\mtx{C}}{\mtx{X}}
\quad\subjto\quad \mathcal{A}\mtx{X} = \vct{b},
\quad \mtx{X} \in \alpha \mtx{\Delta}_n.
\end{equation}
The constraint set $\alpha \mtx{\Delta}_n$ consists
of $n \times n$ psd matrices with trace $\alpha$.  The objective function
depends on a matrix $\mtx{C} \in \Sym_n$.  The linear
constraints are determined by the linear map $\mathcal{A} : \Sym_n \to \R^d$
and the right-hand side vector $\vct{b} \in \R^d$.

\subsection{Elements of Lagrangian duality}

Introduce the Lagrangian function
\begin{equation} \label{eqn:lagrangian-supp}
L(\mtx{X}; \vct{y}) := \ip{\mtx{C}}{\mtx{X}} + \ip{\vct{y}}{\mathcal{A}\mtx{X} - \vct{b}}
\quad\text{for $\mtx{X} \in \alpha \mtx{\Delta}_n$ and $\vct{y} \in \R^d$.}
\end{equation}
We assume that the Lagrangian admits at least one saddle point
$(\mtx{X}_{\star}, \vct{y}_{\star}) \in \alpha \mtx{\Delta}_n \times \R^d$:
\begin{equation} \label{eqn:saddle-point-supp}
L(\mtx{X}_{\star}, \vct{y}) \leq L(\mtx{X}_{\star}, \vct{y}_{\star})
	\leq L(\mtx{X}; \vct{y}_{\star})
	\quad\text{for all $\mtx{X} \in \alpha \mtx{\Delta}_n$ and $\vct{y} \in \R^d$.}
\end{equation}
This hypothesis is guaranteed under Slater's condition.
Write $p_{\star}$ for the shared extremal value of the dual and primal problems:
\begin{equation} \label{eqn:optimal-value-supp}
\max_{\vct{y} \in \R^d} \min_{\mtx{X} \in \alpha \mtx{\Delta}_n} L(\mtx{X}, \vct{y})
	= p_{\star}
	= \min_{\mtx{X} \in \alpha \mtx{\Delta}_n} \sup_{\vct{y} \in \R^d} L(\mtx{X}, \vct{y}).
\end{equation}
In particular, note that $p_{\star} = \ip{\mtx{C}}{\mtx{X}_{\star}}$.

\subsection{The \CGAL\ iteration}\label{sec:cgal-iteration-supp}

The \CGAL iteration solves the model problem~\cref{eqn:model-problem-supp}
using an augmented Lagrangian method where the primal step is inspired
by the conditional gradient method.
See~\cref{alg:cgal-supp} for pseudocode.

Let $\beta_0 > 0$ be an initial smoothing parameter.
Fix a schedule for the step size parameter $\eta_t$ and the smoothing parameter $\beta_t$:
\begin{equation} \label{eqn:cgal-step-supp}
\eta_t := \frac{2}{t+1}
\quad\text{and}\quad
\beta_t := \beta_0\sqrt{t+1}
\quad\text{for $t = 1,2,3, \dots$.}
\end{equation}
Define the augmented Lagrangian $L_t$ with smoothing parameter $\beta_t$
\begin{equation} \label{eqn:al-supp}
L_{t}(\mtx{X}; \vct{y}) :=
	\ip{\mtx{C}}{\mtx{X}} + \ip{\vct{y}}{\mathcal{A}\mtx{X} - \vct{b}}
	+ \frac{1}{2} \beta_t\normsq{\mathcal{A}\mtx{X} - \vct{b}}.
\end{equation}
The \CGAL\ algorithm solves the model problem~\cref{eqn:model-problem-supp}
by alternating between primal and dual update steps on~\cref{eqn:al-supp},
while increasing the smoothing parameter.

Fix the initial iterates
\begin{equation} \label{eqn:cgal-init-supp}
\mtx{X}_1 = \mtx{0} \in \Sym_n
\quad\text{and}\quad
\vct{y}_1 = \vct{0} \in \R^d.
\end{equation}
At each iteration $t = 1, 2, 3, \dots$, we implicitly compute the partial derivative
\begin{equation} \label{eqn:cgal-grad-supp}
\mtx{D}_t := \partial_{\mtx{X}} L_t(\mtx{X}_t; \vct{y}_t)
	= \mtx{C} + \mathcal{A}^*\big( \vct{y}_t + \beta_t (\mathcal{A}\mtx{X}_t - \vct{b}) \big).
\end{equation}
Then we identify a primal update direction $\mtx{H}_t \in \alpha \mtx{\Delta}_n$ that satisfies
\begin{equation} \label{eqn:cgal-direction-supp}
\ip{ \mtx{D}_t }{ \mtx{H}_t }
	\leq \min_{\mtx{H} \in \alpha \mtx{\Delta}_n}
	\ip{ \mtx{D}_t }{ \mtx{H} }
	+ \frac{\alpha\beta_0}{\beta_t} \norm{\mtx{D}_t}.
\end{equation}
In other words, the smoothing parameter $\beta_t$ also controls
the amount of inexactness we are willing
to tolerate in the linear minimization at iteration $t$.
We construct the next primal iterate via the rule
\begin{equation} \label{eqn:cgal-primal-update-supp}
\mtx{X}_{t+1} = \mtx{X}_t + \eta_t (\mtx{H}_t - \mtx{X}_t) \in \alpha \mtx{\Delta}_n.
\end{equation}
The dual update takes the form
\begin{equation} \label{eqn:cgal-dual-update-supp}
\vct{y}_{t+1} = \vct{y}_t + \gamma_t (\mathcal{A}\mtx{X}_{t+1} - \vct{b}).
\end{equation}
We select the largest dual step size parameter $0 \leq \gamma_t \leq \beta_0$
that satisfies the condition %
\begin{equation} \label{eqn:cgal-dual-step-supp}
\gamma_t \normsq{\mathcal{A}\mtx{X}_{t+1} - \vct{b}}
	\leq \beta_t \eta_t^2 \alpha^2 \normsq{\mathcal{A}}.
\end{equation}
The latter inequality always holds when $\gamma_t = 0$.
We will also choose the dual step size to maintain a bounded travel condition:
\begin{equation} \label{eqn:cgal-bdd-travel-supp}
\norm{\vct{y}_t} \leq K.
\end{equation}
If the bounded travel condition holds at iteration $t - 1$,
then the choice $\gamma_t = 0$ ensures that it holds at iteration $t$.
In practice, it is not necessary to enforce~\cref{eqn:cgal-bdd-travel-supp}.
The iteration continues until it reaches a maximum iteration count $T_{\max}$.

\begin{algorithm}[t]%
  \caption{\CGAL\ for the model problem~\cref{eqn:model-problem-supp} %
  \label{alg:cgal-supp}}
  \begin{algorithmic}[1]
  \vspace{0.5pc}

	\Require{Problem data for~\cref{eqn:model-problem-supp} implemented via the primitives~\cref{eqn:primitives};
	number $T$ of iterations}
    \Ensure{Approximate solution $\mtx{X}_T$ to~\cref{eqn:model-problem}}
	\Recommend{Set $T \approx \eps^{-1}$ to achieve $\eps$-optimal solution~\cref{eqn:eps-optimal-supp}}

	\Statex

	\Function{\CGAL}{$T$}

	\State	Scale problem data (\cref{sec:numerics-scaling})
		\Comment	\textcolor{dkblue}{\textbf{[opt]}} Recommended!
	\State	$\beta_0 \gets 1$ and $K \gets + \infty$
		\Comment	Default parameters

	\State	$\mtx{X} \gets \mtx{0}_{n \times n}$ and $\vct{y} \gets \vct{0}_d$
	\For{$t \gets 1, 2, 3, \dots, T$}
  		\State	$\beta \gets \beta_0 \sqrt{t+1}$ and $\eta \gets 2 /(t+1)$
		\State	$(\xi, \vct{v}) \gets \textsf{ApproxMinEvec}( \mtx{C} + \mathcal{A}^* (\vct{y} + \beta(\mathcal{A} \mtx{X} - \vct{b})); q_t )$
		\Statex	\Comment \cref{alg:rand-lanczos} with $q_t = t^{1/4} \log n$
		\Statex \Comment Implement with primitives~\cref{eqn:primitives}\primone\primtwo!
		\State	$\vct{X} \gets (1 - \eta) \, \vct{X} + \eta \, (\alpha \, \vct{vv}^*)$
		\State	$\vct{y} \gets \vct{y} + \gamma (\mathcal{A} \mtx{X} - \vct{b})$
			\Comment	Step size $\gamma$ satisfies~\cref{eqn:cgal-dual-step-supp,eqn:cgal-bdd-travel-supp}
	\EndFor
	\EndFunction

  \vspace{0.25pc}

\end{algorithmic}
\end{algorithm}

\subsection{Distributed computation}

Since \CGAL\ builds on the augmented Lagrangian framework,
we can apply it even when the problem is too large to solve on one computational node.
In particular, when $d$ is large, it may be advantageous to partition
the constraint matrices $\mtx{A}_i$ and the associated dual variables $y_i$
among several workers.
Distributed \CGAL has a structure similar to the alternating directions method of multipliers (\textsf{ADMM}) \cite{BPC+11:Distributed-Optimization}.

\subsection{Theoretical analysis of \CGAL}

We develop two results on the behavior of \CGAL.
The first concerns its convergence to optimality,
and the second concerns the computational resource usage.

\subsubsection{Convergence}

The first result demonstrates that the \CGAL\ algorithm always converges
to a primal optimal solution of the model problem~\cref{eqn:model-problem-supp}.
This result is adapted from~\cite{YFC19:Conditional-Gradient-Based};
a complete proof appears below in~\cref{sec:cgal-pf-supp}.

\begin{theorem}[\CGAL: Convergence] \label{thm:cgal-supp}
Define
$$
E := 6 \alpha^2 \normsq{\mathcal{A}} + \alpha (\norm{\mtx{C}} + K \norm{\mathcal{A}}).
$$
The primal iterates $\{ \mtx{X}_t : t = 2, 3, 4, \dots \}$ generated by the
\CGAL iteration satisfy the {a priori} bounds
\begin{align}
&\ip{\mtx{C}}{\mtx{X}_t} - p_{\star}
	\leq \frac{2\beta_0E}{\sqrt{t}}
	+ K \cdot \norm{ \mathcal{A}\mtx{X}_{t} - \vct{b}}; \label{eqn:cgal-subopt-supp} \\
- \norm{\vct{y}_{\star}}\cdot\norm{ \mathcal{A}\mtx{X}_{t} - \vct{b} }
	\leq\ &\ip{\mtx{C}}{\mtx{X}_{t}} - p_{\star}; \label{eqn:cgal-superopt-supp} \\
\norm{ \mathcal{A}\mtx{X}_{t} - \vct{b} }
	\leq\ &\frac{2\beta_0^{-1} (K + \norm{\vct{y}_{\star}}) + 2 \sqrt{E}}{\sqrt{t}}.
	\label{eqn:cgal-infeas-supp}
\end{align}
The \emph{a priori} bounds %
ensure that
\begin{equation} \label{eqn:eps-optimal-supp}
\norm{ \mathcal{A} \mtx{X}_t - \vct{b} } \leq \eps %
\quad\text{and}\quad
\abs{ \ip{ \mtx{C} }{ \mtx{X}_t } - p_{\star} } \leq \eps %
\end{equation}
within $\mathcal{O}(\eps^{-2})$ iterations.
The big-O suppresses constants that depend on the problem data $(\alpha, \norm{\mathcal{A}}, \norm{\mtx{C}})$
and the algorithm parameters $\beta_0$ and $K$.
\end{theorem}

\subsubsection{Problem scaling}
\label{sec:prob-scaling-supp}

\Cref{thm:cgal-supp} indicates that it is valuable to scale the model problem~\cref{eqn:model-problem-supp}
so that $\alpha = \norm{\mtx{C}} = \norm{\mathcal{A}} = 1$.  In this case, a good choice for the
smoothing parameter is $\beta_0 = 1$.  Nevertheless, the algorithm converges, regardless of
the scaling and regardless of the parameter choices $\beta_0$ and $K$.
We use a slightly different scaling in practice; see~\cref{sec:numerics-scaling}.

\subsubsection{Bounded travel?\nopunct}

\Cref{thm:cgal-supp} suggests that the optimal choice for the travel bound
is $K = 0$.  In other words, the dual vector $\vct{y}_t = \vct{0}$, and it
does not evolve.  The algorithm that results from this choice is called 
\textsf{HomotopyCGM}~\cite{YFL+18:HomotopyCGM}.
The numerical work on \CGAL, reported in~\cite{YFC19:Conditional-Gradient-Based},
makes it clear that updating the dual variable, as \CGAL\ does, allows for substantial
performance improvements over \textsf{HomotopyCGM}. 
Unfortunately, the theory developed in~\cite{YFC19:Conditional-Gradient-Based},
and echoed here, does not comprehend the reason for this phenomenon.
This is an obvious direction for further research.

\subsection{Proof of \cref{thm:cgal-supp}}
\label{sec:cgal-pf-supp}

In this section, we establish the convergence guarantee
stated in~\cref{thm:cgal-supp}.

\subsubsection{Analysis of the primal update}

The first steps in the proof address the role of the primal update rule~\cref{eqn:cgal-primal-update-supp}.
The arguments parallels the standard convergence analysis~\cite{pmlr-v28-jaggi13} of \textsf{CGM},
applied to the function
\begin{equation} \label{eqn:ft-supp}
f_t(\mtx{X}) := L_t(\mtx{X}; \vct{y}_t)
	= \ip{ \mtx{C} }{ \mtx{X} } + \ip{\vct{y}_t}{\mathcal{A} \mtx{X} - \vct{b}}
	+ \frac{1}{2} \beta_t \normsq{ \mathcal{A} \mtx{X} - \vct{b} }.
\end{equation}
Observe that the gradient $\grad f_t(\mtx{X}_t)$ coincides with the partial derivative~\cref{eqn:cgal-grad-supp}.

To begin, we exploit the smoothness of $f_t$ to control the change in its value at adjacent
primal iterates.  The function $f_t$ is convex on $\Sym_n$, and its gradient has Lipschitz
constant $\beta_t \normsq{\mathcal{A}}$, so
$$
f_t(\mtx{X}_+) - f_t(\mtx{X}) \leq
	\ip{ \grad f_t(\mtx{X}) }{ \mtx{X}_+ - \mtx{X} }
	+ \frac{1}{2} \beta_t \normsq{\mathcal{A} } \fnormsq{ \mtx{X}_+ - \mtx{X} }
	\quad\text{for $\mtx{X},\mtx{X}_+ \in \Sym_n$.}
$$
In particular, with $\mtx{X}_+ = \mtx{X}_{t+1}$ and $\mtx{X} = \mtx{X}_t$, we obtain
\begin{equation} \label{eqn:ft-smooth-supp}
\begin{aligned}
f_t(\mtx{X}_{t+1})
	&\leq f_t(\mtx{X}_t) + \ip{ \mtx{D}_t }{\mtx{X}_{t+1} - \mtx{X}_t}
		+ \frac{1}{2} \beta_t \normsq{\mathcal{A} } \fnormsq{ \mtx{X}_{t+1} - \mtx{X}_t } \\
	&= f_t(\mtx{X}_t) + \eta_t \ip{ \mtx{D}_t }{ \mtx{H}_t - \mtx{X}_t }
		+ \frac{1}{2} \beta_t \eta_t^2 \normsq{\mathcal{A}} \fnormsq{\mtx{H}_{t} - \mtx{X}_t} \\
	&\leq f_t(\mtx{X}_t) + \eta_t \ip{ \mtx{D}_t }{ \mtx{H}_t - \mtx{X}_t }
		+ \beta_t \eta_t^2 \alpha^2 \normsq{\mathcal{A}}.
\end{aligned}
\end{equation}
The second identity follows from the update rule~\cref{eqn:cgal-primal-update-supp}.
The bound on the last term depends on the fact that the constraint set $\alpha \mtx{\Delta}_n$
has Frobenius-norm diameter $\alpha \sqrt{2}$.

Next, we use the construction of the primal update to control the linear term in the last display.
The update direction $\mtx{H}_t$ satisfies the inequality~\cref{eqn:cgal-direction-supp}, so
\begin{equation} \label{eqn:ft-update-supp}
\begin{aligned}
\ip{ \mtx{D}_t }{ \mtx{H}_t - \mtx{X}_t }
	&\leq \min_{\mtx{H} \in \alpha \mtx{\Delta}_n} \ip{ \mtx{D}_t }{\mtx{H} - \mtx{X}_t} + \frac{\alpha \beta_0}{\beta_t} \norm{\mtx{D}_t} \\
	&\leq \ip{ \mtx{D}_t }{ \mtx{X}_{\star} - \mtx{X}_t } + \frac{\alpha \beta_0}{\beta_t} \norm{ \mtx{D}_t }.
\end{aligned}
\end{equation}
The second inequality depends on the fact that $\mtx{X}_{\star} \in \alpha \mtx{\Delta}_n$.

We can use the explicit formula~\cref{eqn:cgal-grad-supp} for the derivative $\mtx{D}_t$
to control the two terms in~\cref{eqn:ft-update-supp}.  First,
\begin{equation} \label{eqn:ft-update-supp-2}
\begin{aligned}
\ip{ \mtx{D}_t }{ \mtx{X}_{\star} - \mtx{X}_t }
	&= \ip{ \mtx{C} + \mathcal{A}^* \big( \vct{y}_t + \beta_t (\mathcal{A}\mtx{X}_t - \vct{b})\big) }
	{ \mtx{X}_{\star} - \mtx{X}_t } \\ %
	&= \ip{ \mtx{C} }{ \mtx{X}_{\star} - \mtx{X}_t }
	- \ip{ \vct{y}_t }{ \mathcal{A}\mtx{X}_t - \vct{b} }
	- \beta_t \normsq{\mathcal{A}\mtx{X}_t - \vct{b} } \\ %
	&= \ip{ \mtx{C} }{\mtx{X}_{\star}} - f_t(\mtx{X}_t)
	- \frac{1}{2} \beta_t \normsq{ \mathcal{A}\mtx{X}_t - \vct{b} }. %
\end{aligned}
\end{equation}
We have invoked the definition of the adjoint $\mathcal{A}^*$
and the fact that $\mathcal{A}\mtx{X}_{\star} = \vct{b}$.
Last, we used the definition~\cref{eqn:ft-supp} to identify the quantity $f_t(\mtx{X}_t)$.
Second,
\begin{equation} \label{eqn:ft-update-supp-3}
\begin{aligned}
\norm{\mtx{D}_t}
	&= \norm{\mtx{C} + \mathcal{A}^* \big(\vct{y}_t + \beta_t (\mathcal{A}\mtx{X}_t - \vct{b}) \big) } \\
	&\leq \norm{ \mtx{C} } + K \norm{ \mathcal{A}} + \beta_t \norm{\mathcal{A}} \norm{ \mathcal{A}\mtx{X}_t - \vct{b} }.
\end{aligned}
\end{equation}
We have used the assumption that $\vct{y}_t$ satisfies
the bounded travel condition~\cref{eqn:cgal-bdd-travel-supp}.

Combine the last three displays to obtain the estimate
\begin{multline} \label{eqn:ft-update-supp-comb}
\ip{ \mtx{D}_t }{ \mtx{H}_t - \mtx{X}_t }
	\leq \ip{\mtx{C}}{\mtx{X}_{\star}} - f_t(\mtx{X}_t) \\
		+ \Big( \alpha\beta_0 \norm{\mathcal{A}} \norm{\mathcal{A}\mtx{X}_t - \vct{b}}
		- \frac{1}{2} \beta_t \normsq{\mathcal{A}\mtx{X}_t - \vct{b} } \Big)
		+ \frac{\alpha \beta_0}{\beta_t}(\norm{\mtx{C}} + K \norm{\mathcal{A}}).
\end{multline}

We can now control the decrease in the function $f_t$ between adjacent primal iterates.
Combine the displays~\cref{eqn:ft-smooth-supp,eqn:ft-update-supp-comb} to arrive at the bound
\begin{multline*}
f_t(\mtx{X}_{t+1}) \leq (1 - \eta_t) f_t(\mtx{X}_t) + \eta_t \ip{ \mtx{C} }{ \mtx{X}_{\star} } \\
	+ \Big( \eta_t \alpha \beta_0 \norm{\mathcal{A}} \norm{\mathcal{A}\mtx{X}_t - \vct{b}}
	- \frac{1}{2} \beta_t \eta_t \normsq{\mathcal{A}\mtx{X}_t - \vct{b}} \Big)
		+ \Big( \beta_t \eta_t^2 \alpha^2 \normsq{\mathcal{A}}
		+ \frac{\eta_t \alpha\beta_0}{\beta_t}(\norm{\mtx{C}} + K \norm{\mathcal{A}}) \Big).
\end{multline*}
Subtract $p_{\star} = \ip{\mtx{C}}{\mtx{X}_{\star}}$ from both sides to arrive at
\begin{multline*}
f_t(\mtx{X}_{t+1}) - p_{\star} \leq (1 - \eta_t)\Big( f_t(\mtx{X}_t) - p_{\star} \Big) \\
	+ \Big( \eta_t  \alpha \beta_0 \norm{\mathcal{A}} \norm{\mathcal{A}\mtx{X}_t - \vct{b}}
	- \frac{1}{2} \beta_t \eta_t \normsq{\mathcal{A}\mtx{X}_t - \vct{b}} \Big)
		+ \Big( \beta_t \eta_t^2 \alpha^2 \normsq{\mathcal{A}}
		+ \frac{\eta_t \alpha\beta_0}{\beta_t}(\norm{\mtx{C}} + K \norm{\mathcal{A}}) \Big).
\end{multline*}
Finally, use the definition~\cref{eqn:ft-supp} again to pass back to the augmented Lagrangian:
\begin{multline} \label{eqn:Lt+primal-supp}
L_t(\mtx{X}_{t+1}; \vct{y}_t) - p_{\star}
	\leq (1 - \eta_t ) \Big( L_t(\mtx{X}_{t}; \vct{y}_t) - p_{\star} \Big) \\
	+ \Big( \eta_t  \alpha \beta_0 \norm{\mathcal{A}} \norm{\mathcal{A}\mtx{X}_t - \vct{b}}
	- \frac{1}{2} \beta_t \eta_t \normsq{\mathcal{A}\mtx{X}_t - \vct{b}} \Big) \\
		+ \Big( \beta_t \eta_t^2 \alpha^2 \normsq{\mathcal{A}}
		+ \frac{\eta_t \alpha\beta_0}{\beta_t}(\norm{\mtx{C}} + K \norm{\mathcal{A}}) \Big).
\end{multline}
This bound describes the evolution of the augmented Lagrangian as the primal iterate advances.
But we still need to include the effects of increasing the smoothing parameter
and updating the dual variable.

\subsubsection{Analysis of the smoothing update}

To continue, observe that updates to the smoothing parameter have a controlled
impact on the augmented Lagrangian~\cref{eqn:al-supp}:
$$
L_{t}(\mtx{X}_{t}; \vct{y}_t) - L_{t-1}(\mtx{X}_t; \vct{y}_t)
	= \frac{1}{2} \big( \beta_{t} - \beta_{t-1} \big)  \normsq{\mathcal{A}\mtx{X}_t - \vct{b}}.
$$
Add and subtract $L_{t-1}(\mtx{X}_t; \vct{y}_t)$ in the large parenthesis in the first line of~\cref{eqn:Lt+primal-supp}
and invoke the last identity to obtain
\begin{multline} \label{eqn:Lt+smooth-supp}
L_t(\mtx{X}_{t+1}; \vct{y}_t) - p_{\star}
	\leq (1 - \eta_t ) \Big( L_{t-1}(\mtx{X}_{t}; \vct{y}_t) - p_{\star} \Big) \\
	+ \Big( \eta_t \alpha\beta_0 \norm{\mathcal{A}} \norm{\mathcal{A}\mtx{X}_t - \vct{b}}
	+ \frac{1}{2} \left[ (1 - \eta_t)( \beta_t - \beta_{t-1} ) - \beta_t \eta_t\right] \normsq{\mathcal{A}\mtx{X}_t - \vct{b}} \Big) \\
	+ \Big( \beta_t \eta_t^2 \alpha^2 \normsq{\mathcal{A}}
	+ \frac{\eta_t \alpha\beta_0}{\beta_t}(\norm{\mtx{C}} + K \norm{\mathcal{A}}) \Big).
\end{multline}

The next step is to develop a uniform bound on the terms in the second line
so that we can ignore the role of the feasibility gap $\norm{\mathcal{A}\mtx{X}_t - \vct{b}}$
in the subsequent calculations.  The choice~\cref{eqn:cgal-step-supp}
of the parameters ensures that
$$
\begin{aligned}
(1 - \eta_t)(\beta_t - \beta_{t-1}) - \beta_t \eta_t
	< \frac{-\beta_0^2}{\beta_t}.
\end{aligned}
$$
Introduce this bound into the second line of~\cref{eqn:Lt+smooth-supp} and maximize
the resulting concave quadratic function to reach
$$
\begin{aligned}
\eta_t \alpha \beta_0 \norm{\mathcal{A}} \norm{\mathcal{A}\mtx{X}_t - \vct{b}}
	&+ \frac{1}{2} \left[ (1 - \eta_t)( \beta_t - \beta_{t-1} ) - \beta_t \eta_t\right] \normsq{\mathcal{A}\mtx{X}_t - \vct{b}} \\
	&\leq \eta_t \alpha \beta_0 \norm{\mathcal{A}} \norm{\mathcal{A}\mtx{X}_t - \vct{b}}
	- \frac{\beta_0^2}{2 \beta_t} \normsq{\mathcal{A}\mtx{X}_t - \vct{b}}
	\leq \beta_t \eta_t^2 \alpha^2 \normsq{\mathcal{A}}.
\end{aligned}
$$
Substitute the last display into~\cref{eqn:Lt+smooth-supp} to determine that
\begin{multline} \label{eqn:Lt+smooth-supp-2}
L_t(\mtx{X}_{t+1}; \vct{y}_t) - p_{\star}
	\leq (1 - \eta_t ) \Big( L_{t-1}(\mtx{X}_{t}; \vct{y}_t) - p_{\star} \Big) \\
	+ \Big( 2 \beta_t \eta_t^2 \alpha^2 \normsq{\mathcal{A}}
	+ \frac{\eta_t \alpha \beta_0}{\beta_t}(\norm{\mtx{C}} + K \norm{\mathcal{A}}) \Big).
\end{multline}
To develop a recursion, we need to assess how the left-hand side changes
when we update the dual variable.

\subsubsection{Analysis of the dual update}

To incorporate the dual update, observe that
$$
\begin{aligned}
L_t(\mtx{X}_{t+1}; \vct{y}_{t+1})
	&= L_t(\mtx{X}_{t+1}; \vct{y}_t) + \ip{ \vct{y}_{t+1} - \vct{y}_t }{ \mathcal{A} \mtx{X}_{t+1} - \vct{b} } \\
	&= L_t(\mtx{X}_{t+1}; \vct{y}_t) + \gamma_t \normsq{ \mathcal{A}\mtx{X}_{t+1} - \vct{b} } \\
	&\leq L_t(\mtx{X}_{t+1}; \vct{y}_t) + \beta_t \eta_t^2 \alpha^2 \normsq{\mathcal{A}}.
\end{aligned}
$$
The first relation is simply the definition~\cref{eqn:al-supp} of the augmented Lagrangian,
while the second relation depends on the dual update rule~\cref{eqn:cgal-dual-update-supp}.
The last step follows from the selection rule~\cref{eqn:cgal-dual-step-supp} for
the dual step size parameter.

Introduce the latter display into~\cref{eqn:Lt+smooth-supp} to discover that
\begin{multline} \label{eqn:Lt-recursion-supp}
L_t(\mtx{X}_{t+1}; \vct{y}_{t+1}) - p_{\star}
	\leq (1 - \eta_t ) \Big( L_{t-1}(\mtx{X}_{t}; \vct{y}_t) - p_{\star} \Big) \\
	+ \Big( 3 \beta_t \eta_t^2 \alpha^2 \normsq{\mathcal{A}}
	+ \frac{\eta_t \alpha \beta_0}{\beta_t}(\norm{\mtx{C}} + K \norm{\mathcal{A}}) \Big).
\end{multline}
We have developed a recursion for the value of the augmented Lagrangian
as the iterates and the smoothing parameter evolve.

\subsubsection{Solving the recursion}

Next, we solve the recursion~\cref{eqn:Lt-recursion-supp}.  We assert that
\begin{equation} \label{eqn:cgal-err-supp}
\begin{aligned}
L_t(\mtx{X}_{t+1}; \vct{y}_{t+1}) - p_{\star}
	&< \frac{2 \beta_0}{\sqrt{t+1}}
	\left[6 \alpha^2 \normsq{\mathcal{A}} + \alpha (\norm{\mtx{C}} + K \norm{\mathcal{A}}) \right] \\
	&=: \frac{2\beta_0E}{\sqrt{t+1}}
\quad\text{for $t = 1, 2, 3, \dots$.}
\end{aligned}
\end{equation}
For the case $t = 1$, the definition~\cref{eqn:cgal-step-supp}
ensures that $\eta_1 = 1$ and $\beta_1 = \beta_0 \sqrt{2}$,
so the bound~\cref{eqn:cgal-err-supp} follows instantly from~\cref{eqn:Lt-recursion-supp}.
When $t > 1$, an inductive argument using the recursion~\cref{eqn:Lt-recursion-supp}
and the bound~\cref{eqn:cgal-err-supp} for $t - 1$ ensures that
\begin{align*}
L_t(\mtx{X}_{t+1}; \vct{y}_{t+1}) - p_{\star}
	&\leq \left[ \frac{t-1}{t+1} \cdot \frac{2\beta_0}{\sqrt{t}} + \frac{2\beta_0}{(t+1)^{3/2}} \right]
	\left[ 6 \alpha^2\normsq{\mathcal{A}} + \alpha (\norm{\mtx{C}} + K \norm{\mathcal{A}}) \right] \\
	&< \frac{2\beta_0}{\sqrt{t+1}}
	\left[ 6 \alpha^2\normsq{\mathcal{A}} + \alpha (\norm{\mtx{C}} + K \norm{\mathcal{A}}) \right].
\end{align*}
We have introduced the stated values~\cref{eqn:cgal-step-supp}
of the step size and smoothing parameters.  The induction proceeds,
and we conclude that~\cref{eqn:cgal-err-supp} is valid.

\subsubsection{Bound for the suboptimality of the objective}

We are prepared to develop an upper bound on the suboptimality of the objective
of the model problem~\cref{eqn:model-problem-supp}.  The definition~\cref{eqn:al-supp}
of the augmented Lagrangian directly implies that
\begin{multline} \label{eqn:subopt-decomp-supp}
\ip{\mtx{C}}{\mtx{X}_{t+1}} - p_{\star}
	= L_t(\mtx{X}_{t+1}; \vct{y}_{t+1}) - p_{\star}
	- \ip{ \vct{y}_{t+1} }{ \mathcal{A}\mtx{X}_{t+1}-\vct{b}} \\
	- \frac{1}{2}\beta_t \normsq{ \mathcal{A} \mtx{X}_{t+1} - \vct{b} }.
\end{multline}
Continuing from here,
$$
\ip{\mtx{C}}{\mtx{X}_{t+1}} - p_{\star}
	\leq \frac{2\beta_0E}{\sqrt{t+1}}
	+ K \cdot \norm{ \mathcal{A}\mtx{X}_{t+1} - \vct{b}}.
$$
The first identity follows from definition~\cref{eqn:al-supp} of the augmented Lagrangian.
The bound relies on~\cref{eqn:cgal-err-supp} and the Cauchy--Schwarz inequality.
We have also used the bounded travel condition~\cref{eqn:cgal-bdd-travel-supp}.
This establishes~\cref{eqn:cgal-subopt-supp}.

\subsubsection{Bound for the superoptimality of the objective}
\label{sec:bound-superopt-supp}

The \CGAL iterates $\mtx{X}_t$ are generally infeasible
for~\cref{eqn:model-problem-supp}, so they can be superoptimal.
Nevertheless, we can easily control the superoptimality.
By the saddle-point properties~\cref{eqn:saddle-point-supp,eqn:optimal-value-supp},
the Lagrangian~\cref{eqn:lagrangian-supp} satisfies
\begin{equation} \label{eqn:subopt-lagrange-supp}
p_{\star} = \max_{\vct{y}} \min_{\mtx{X} \in \alpha \mtx{\Delta}_n} L(\mtx{X}; \vct{y})
	\leq L(\mtx{X}_{t+1}; \vct{y}_{\star})
	= \ip{ \mtx{C} }{ \mtx{X}_{t+1} } + \ip{ \vct{y}_{\star} }{ \mathcal{A} \mtx{X}_{t+1} - \vct{b} }.
\end{equation}
Invoke the Cauchy--Schwarz inequality and rearrange to determine that %
$$
- \norm{\vct{y}_{\star}} \cdot\norm{\mathcal{A}\mtx{X}_{t+1} -\vct{b}}
	\leq \ip{ \mtx{C} }{ \mtx{X}_{t+1} } - p_{\star}.
$$
This establishes~\cref{eqn:cgal-superopt-supp}.

\subsubsection{Bound for the infeasibility of the iterates}

Next, we demonstrate that the iterates converge toward the feasible set
of~\cref{eqn:model-problem-supp}.  Combine~\cref{eqn:subopt-decomp-supp,eqn:subopt-lagrange-supp}
and rearrange to see that
$$
\ip{ \vct{y}_{t+1} - \vct{y}_{\star} }{ \mathcal{A}\mtx{X}_{t+1}-\vct{b} }
	\leq L_t(\mtx{X}_{t+1}; \vct{y}_{t+1}) - p_{\star} - \frac{1}{2} \beta_t \normsq{\mathcal{A}\mtx{X}_{t+1} - \vct{b}}.
$$
Bound the left-hand side below using Cauchy--Schwarz and the right-hand side above
using~\cref{eqn:cgal-err-supp}:
$$
- \norm{ \vct{y}_{t+1} - \vct{y}_{\star} } \cdot \norm{ \mathcal{A}\mtx{X}_{t+1} - \vct{b}}
	\leq \frac{2\beta_0 E}{\sqrt{t+1}}
	- \frac{1}{2} \beta_t \normsq{ \mathcal{A}\mtx{X}_{t+1} - \vct{b}}.
$$
Solve this quadratic inequality to obtain the bound
$$
\begin{aligned}
\norm{\mathcal{A}\mtx{X}_{t+1} - \vct{b}}
	&\leq \beta_t^{-1} \left( \norm{\vct{y}_{t+1} - \vct{y}_{\star}} + \sqrt{ \normsq{\vct{y}_{t+1} - \vct{y}_{\star}}
	+ \frac{4 \beta_t \beta_0 E}{\sqrt{t+1}} } \right) \\
	&\leq \beta_t^{-1} \Big( 2 \norm{\vct{y}_{t+1} - \vct{y}_{\star}} + \sqrt{4\beta_0^2 E} \Big) \\
	&= \frac{2\beta_0^{-1} \norm{\vct{y}_{t+1} - \vct{y}_{\star}} + 2 \sqrt{E}}{\sqrt{t+1}}.
\end{aligned}
$$
The second inequality depends on the definition~\cref{eqn:cgal-step-supp}
of the smoothing parameter $\beta_t$ and the subadditivity of the square root.

Finally, we control the dual error using the bounded travel condition~\cref{eqn:cgal-bdd-travel-supp}:
$$
\norm{\vct{y}_{t+1} - \vct{y}_{\star}}
	\leq \norm{\vct{y}_{t+1}} + \norm{\vct{y}_{\star}}
	\leq K + \norm{\vct{y}_{\star}}.
$$
The last two displays yield
$$
\norm{\mathcal{A}\mtx{X}_{t+1} - \vct{b}}
	\leq \frac{2\beta_0^{-1} (K + \norm{\vct{y}_{\star}}) + 2 \sqrt{E}}{\sqrt{t+1}}.
$$
This confirms~\cref{eqn:cgal-infeas-supp}.

\subsection{Computable bounds for suboptimality}
\label{sec:bound-subopt-supp}

In this section, we assume that the linear minimization~\cref{eqn:cgal-direction-supp}
is performed exactly at iteration $t$. 
{%
That is, there is no error depending on $\norm{\mtx{D}_t}$.} 
Introduce the duality gap surrogate
\begin{equation} \label{eqn:cgal-gap-supp}
g_t(\mtx{X}) := \max_{\mtx{H} \in \alpha \mtx{\Delta}_n} \ip{ \grad f_t(\mtx{X}) }{ \mtx{X} - \mtx{H} }.
\end{equation}
The function $f_t$ is defined in~\cref{eqn:ft-supp}.
Note that the gap $g(\mtx{X}_t)$ can be evaluated with the information we have at hand:
\begin{equation} \label{eqn:cgal-dual-gap-supp}
\begin{aligned}
g_t(\mtx{X}_t) &= \ip{ \mtx{D}_t }{ \mtx{X}_t } - \ip{ \mtx{D}_t }{ \mtx{H}_t } \\
	&= \ip{ \mtx{C} }{ \mtx{X}_t } + \ip{ \vct{y}_t + \beta_t (\mathcal{A}\mtx{X}_t -\vct{b}) }{ \mathcal{A}\mtx{X}_t }
	- \ip{ \mtx{D}_t }{ \mtx{H}_t }.
\end{aligned}
\end{equation}
The last term is just the value of the linear minimization~\cref{eqn:cgal-direction-supp}

The gap gives us computable bounds on the suboptimality of the current iterate $\mtx{X}_t$.
Indeed, the convexity of $f_t$ implies that
$$
f_t(\mtx{X}_t) - f_t(\mtx{X}_{\star})
	\leq \ip{ \grad f_t(\mtx{X}_t) }{ \mtx{X}_t - \mtx{X}_{\star} }
	\leq g_{t}(\mtx{X}_t).
$$
Using the definition~\cref{eqn:ft-supp} and the fact that
$\mtx{X}_{\star}$ is feasible for~\cref{eqn:model-problem-supp},
we find that
$$
f_t(\mtx{X}_t) - p_{\star}
	\, = \, f_t(\mtx{X}_t) - f_t(\mtx{X}_{\star}) %
	\leq g_t(\mtx{X}_t).
$$
Invoking the definition~\cref{eqn:ft-supp} again and rearranging, we find that
\begin{equation} \label{eqn:cgal-subopt-pf-supp}
\begin{aligned}
\ip{\mtx{C}}{\mtx{X}_t} - p_{\star}
	&\leq g_t(\mtx{X}_t) - \ip{ \vct{y}_t }{ \mathcal{A} \mtx{X}_t - \vct{b} } - \frac{1}{2} \beta_t \normsq{\mathcal{A}\mtx{X}_t - \vct{b} }. %
\end{aligned}
\end{equation}
In other words, the suboptimality of the primal iterate $\mtx{X}_t$
is bounded in terms of the dual gap $g_t(\mtx{X}_t)$, the feasibility gap
$\mathcal{A} \mtx{X}_t - \vct{b}$, and the dual variable $\vct{y}_t$.

\begin{remark}[Bounds with approximate linear minimization]
We can extend the bound \cref{eqn:cgal-subopt-pf-supp} for the approximate linear minimization in \cref{eqn:cgal-direction-supp} by taking the error term into account.
Based on the definition \cref{eqn:cgal-gap-supp} of $g_t(\mtx{X}_t)$ and the oracle \cref{eqn:cgal-direction-supp},
\begin{equation*}
g_t(\mtx{X}_t) \leq \ip{ \mtx{D}_t }{ \mtx{X}_t } - \ip{ \mtx{D}_t }{ \mtx{H}_t } + \frac{\alpha\beta_0}{\beta_t} \norm{\mtx{D}_t}.
\end{equation*}
This leads to an extended version of \cref{eqn:cgal-subopt-pf-supp}:
\begin{equation}\label{eqn:cgal-subopt-pf-approx-supp}
\ip{\mtx{C}}{\mtx{X}_t} - p_{\star}
	\leq g_t(\mtx{X}_t) - \ip{ \vct{y}_t }{ \mathcal{A} \mtx{X}_t - \vct{b} } - \frac{1}{2} \beta_t \normsq{\mathcal{A}\mtx{X}_t - \vct{b} } + \frac{\alpha\beta_0}{\beta_t} \norm{\mtx{D}_t}.
\end{equation}
Note that the additional error term converges to zero. 
We can simplify this bound further by using \cref{eqn:ft-update-supp-3}.

In practice, we have observed that the bound~\cref{eqn:cgal-subopt-pf-approx-supp}
is less informative than simply using the bound~\cref{eqn:cgal-subopt-pf-supp},
which assumes that the eigenvalue computation is exact.
\end{remark}

{%
\begin{remark}[Superoptimality]
Note that the suboptimality $\ip{\mtx{C}}{\mtx{X}_t} - p_{\star}$ can attain negative values because the \CGAL\ iterates $\mtx{X}_t$ are generally infeasible. Similarly, the upper bound in \cref{eqn:cgal-subopt-pf-supp} can be negative. In other words, the iterates $\mtx{X}_t$ can be superoptimal.  Nevertheless, the superoptimality is controlled by the feasibility gap, as shown in \cref{sec:bound-superopt-supp}. 
\end{remark}
}

\section{Sketching and reconstruction of a positive-semidefinite matrix}
\label{sec:nystrom-supp}

This section reviews and gives additional details about the
Nystr{\"o}m sketch~\cite{HMT11:Finding-Structure,Git13:Topics-Randomized,LLS+17:Algorithm-971,TYUC17:Fixed-Rank-Approximation}.
This sketch tracks an evolving psd matrix and then reports a provably accurate low-rank approximation.
The material on error estimation is new.

\subsection{Sketching and updates}

Consider a psd input matrix $\mtx{X} \in \Sym_n$.
Let $R$ be a parameter that
modulates the storage cost of the sketch and the quality of the matrix approximation.

To construct the sketch, we draw and fix a standard normal %
test matrix $\mtx{\Omega} \in \F^{n \times R}$.
The sketch of the matrix $\mtx{X}$ takes the form
\begin{equation} \label{eqn:sketch-supp}
\mtx{S} = \mtx{X\Omega} \in \F^{n \times R}
\quad\text{and}\quad
\tau = \trace(\mtx{X}) \in \R.
\end{equation}
The sketch supports linear rank-one updates to $\mtx{X}$.
Indeed, we can track the evolution
\begin{equation} \label{eqn:sketch-update-supp}
\begin{aligned}
\mtx{X} &\gets (1 - \eta)\, \mtx{X} + \eta \,\vct{vv}^*
\qquad\qquad\text{for $\eta \in[0,1]$ and $\vct{v}\in\F^n$} \\
\quad\text{via}\quad
\mtx{S} &\gets (1 - \eta)\,\mtx{S} + \eta \, \vct{v} (\vct{v}^* \mtx{\Omega}) \\
\quad\text{and}\quad
\tau &\gets (1 - \eta)\, \tau + \eta \, \normsq{ \vct{v} }.
\end{aligned}
\end{equation}
The test matrix $\mtx{\Omega}$ and the sketch $(\mtx{S}, \tau)$ require
storage of $2Rn + 1$ numbers in $\F$.  The arithmetic cost of the linear
update~\cref{eqn:sketch-update} to the sketch is $\Theta(Rn)$ numerical operations.

\begin{remark}[Structured random matrices]
We can reduce storage costs by a factor of two by using a structured
random matrix in place of $\mtx{\Omega}$.
 For example, see~\cite[Sec.~3]{TYUC19:Streaming-Low-Rank}
or~\cite{SGTU18:Tensor-Random}.
This modification requires
additional care with implementation (e.g., use of sparse arithmetic
or fast trigonometric transforms), but the improvement can be
significant for very large problems.
\end{remark}

\subsection{The reconstruction process}

Given the test matrix $\mtx{\Omega}$
and the sketch $\mtx{S} = \mtx{X\Omega}$,
we can form a rank-$R$ approximation $\widehat{\mtx{X}}$
of the matrix $\mtx{X}$ contained in the sketch.
This approximation is defined by the formula
\begin{equation} \label{eqn:nystrom-supp}
\widehat{\mtx{X}} := \mtx{S} (\mtx{\Omega}^* \mtx{S})^\dagger \mtx{S}^*
	= (\mtx{X\Omega}) (\mtx{\Omega}^* \mtx{X \Omega})^{\dagger} (\mtx{X \Omega})^*.
\end{equation}
This reconstruction is called a \emph{Nystr{\"o}m approximation}.
We often truncate $\widehat{\mtx{X}}$ %
by replacing it with
its best rank-$r$ approximation $\lowrank{\widehat{\mtx{X}}}{r}$ for some $r \leq R$.

See~\cref{alg:nystrom-sketch-supp} for a numerically stable
implementation of the Nystr{\"o}m reconstruction process~\cref{eqn:nystrom},
including error estimation steps. %
The algorithm takes $\mathcal{O}(R^2 n)$ numerical operations
and $\Theta(Rn)$ storage.

\begin{algorithm}[t]%
  \caption{\textsf{NystromSketch} implementation (see \cref{sec:nystrom-supp})
  \label{alg:nystrom-sketch-supp}}
  \begin{algorithmic}[1]
  \vspace{0.5pc}

	\Require{Dimension $n$ of input matrix, size $R$ of sketch } %
    \Ensure{Rank-$R$ approximation $\widehat{\mtx{X}}$ of sketched matrix in factored form
    $\widehat{\mtx{X}} = \mtx{U\Lambda U}^*$,
    where $\mtx{U} \in \R^{n \times R}$ has orthonormal columns and $\mtx{\Lambda} \in \R^{R \times R}$
    is nonnegative diagonal, and the Schatten 1-norm approximation errors $\mathrm{err}(\lowrank{\widehat{\mtx{X}}}{r})$
    for $1 \leq r \leq R$, as defined in~\cref{eqn:sketch-err-supp}}
    \Recommend{Choose $R$ as large as possible} %

	\vspace{0.5pc}

	\Function{\textsf{NystromSketch.Init}}{$n$, $R$}
	\State	$\mtx{\Omega} \gets \texttt{randn}(n, R)$
		\Comment	Draw and fix random test matrix
	\State	$\mtx{S} \gets \texttt{zeros}(n, R)$ and $\tau \gets 0$
		\Comment	Form sketch of zero matrix
  	\EndFunction

	\vspace{0.5pc}

	\Function{\textsf{NystromSketch.RankOneUpdate}}{$\vct{v}$, $\eta$}
		\Comment	Implements~\cref{eqn:sketch-update}
	\State	$\mtx{S} \gets	(1-\eta)\, \mtx{S} + \eta\, \vct{v} (\vct{v}^* \mtx{\Omega})$
		\Comment	Update sketch of matrix
	\State	$\tau \gets (1 - \eta) \, \tau + \eta \, \normsq{\vct{v}}$
		\Comment	Update the trace
	\EndFunction

	\vspace{0.5pc}

	\Function{\textsf{NystromSketch.Reconstruct}}{\,} %
	\State	$\sigma \gets \sqrt{n} \, \texttt{eps}(\texttt{norm}(\mtx{S}))$ %
		\Comment	Compute a shift parameter
	\State	$\mtx{S} \gets \mtx{S} + \sigma \, \mtx{\Omega}$
		\Comment	Implicitly form sketch of $\mtx{X} + \sigma \, \Id$
	\State	$\mtx{L} \gets \texttt{chol}(\mtx{\Omega}^* \mtx{S})$
	\State	$[\mtx{U}, \mtx{\Sigma}, \sim] \gets \texttt{svd}(\mtx{S}/\mtx{L})$
		\Comment	Dense SVD
	\State	$\mtx{\Lambda} \gets \max\{0, \mtx{\Sigma}^2 - \sigma \, \Id \}$
		\Comment	Remove shift
	\State	$\mathbf{err} \gets \tau - \texttt{cumsum}(\texttt{diag}(\mtx{\Lambda}))$
		\Comment	Compute approximation errors
	\EndFunction

  \vspace{0.25pc}

\end{algorithmic}
\end{algorithm}

\subsection{The approximation error}

We can easily determine the exact Schatten 1-norm error in
the truncated Nystr{\"o}m approximation:
\begin{equation} \label{eqn:sketch-err-supp}
\mathrm{err}(\lowrank{\widehat{\mtx{X}}}{r})
	:= \norm{ \mtx{X} - \lowrank{\widehat{\mtx{X}}}{r} }_{*}
	= \tau - \trace(\lowrank{\widehat{\mtx{X}}}{r})
	\quad\text{for each $r \leq R$.}
\end{equation}
Furthermore, we can use~\cref{eqn:sketch-err-supp}
to ascertain whether the unknown input matrix $\mtx{X}$
is (almost) low-rank.  Indeed, the best rank-$r$ approximation of $\mtx{X}$
satisfies
\begin{equation} \label{eqn:sketch-input-err-bd-supp}
\norm{ \mtx{X} - \lowrank{\mtx{X}}{r} }_* \leq
	\mathrm{err}(\lowrank{\widehat{\mtx{X}}}{r})
	\quad\text{for each $r \leq R$.}
\end{equation}
Thus, large drops in the function $r \mapsto \mathrm{err}(\lowrank{\widehat{\mtx{X}}}{r})$
signal large drops in the eigenvalues of $\mtx{X}$.
See \cref{sec:approx-analysis-supp} for the details.

\subsection{Trace correction}
\label{sec:trace-correct-supp}

The trace of the Nystr{\"o}m approximation $\widehat{\mtx{X}}$
does not exceed the trace of the input matrix $\mtx{X}$.  It can sometimes be helpful
to replace the Nystr{\"o}m approximation by another matrix whose trace matches
the trace of the input matrix.

Suppose that the Nystr{\"o}m approximation has the form $\widehat{\mtx{X}} = \mtx{U\Lambda U}^*$,
where $\mtx{U} \in \F^{n \times R}$ is orthonormal and $\mtx{\Lambda} \in \Sym_R^+$ is diagonal.  
Assume that $\trace( \mtx{X} ) = \alpha$ and $\trace( \widehat{\mtx{X}} ) = \widehat{\alpha}$.
Since $\widehat{\mtx{X}}$ is a Nystr{\"o}m approximation of $\mtx{X}$,
it holds that $\widehat{\alpha} \leq \alpha$.
Now, we can construct a second approximation
\begin{equation} \label{eqn:trace-correct-supp}
\widetilde{\mtx{X}} := \mtx{U} \, \big( \mtx{\Lambda} + (\alpha - \widehat{\alpha}) \Id_R / R \big)\, \mtx{U}^*.
\end{equation}
It is evident that $\widehat{\mtx{X}} \psdle \widetilde{\mtx{X}}$ and that $\trace \widetilde{\mtx{X}} = \alpha$.

One might hope that the trace-corrected approximation $\widetilde{\mtx{X}}$
is better than the original Nystr{\"o}m approximation $\widehat{\mtx{X}}$,
but this is not necessarily the case.  Fortunately, we always have the
relation
$$
\norm{ \mtx{X} - \widetilde{\mtx{X}} }_*
	\leq 2 \norm{ \mtx{X} - \widehat{\mtx{X}} }_*.
$$
In other words, correcting the trace doubles the error, at worst.
See \cref{sec:trace-correct-pf-supp} for the proof.

\subsection{Statistical properties of the Nystr{\"o}m sketch}

The truncated Nystr{\"o}m approximation has a number of attractive statistical properties.
For a fixed input matrix $\mtx{X}$, the expected approximation error
$\Expect_{\mtx{\Omega}} \norm{ \mtx{X} - \lowrank{\widehat{\mtx{X}}}{r} }_*$
is monotone decreasing in both the sketch size $R$ and the truncation rank $r$.
Furthermore, if $\rank(\mtx{X}) = r$ for $r \leq R$,
then $\norm{ \mtx{X} - \lowrank{\widehat{\mtx{X}}}{r} }_{*} = 0$ with probability one.
We establish these results below in \cref{sec:nystrom-stat-supp}.

\subsection{\textit{A priori} error bounds}
\label{sec:nystrom-apriori-supp}

The Nystr{\"o}m approximation $\widehat{\mtx{X}}$ yields a provably good estimate
for the matrix $\mtx{X}$ contained in the sketch~\cite[Thms.~4.1]{TYUC17:Fixed-Rank-Approximation}.

\begin{fact}[Nystr{\"o}m sketch: Error bound] \label{fact:nystrom-supp}
Fix a psd matrix $\mtx{X} \in \Sym_n$.  Let $\mtx{S} = \mtx{X\Omega}$
where $\mtx{\Omega} \in \F^{n \times R}$ is standard normal.
For each $r < R - 1$, the Nystr{\"o}m approximation~\cref{eqn:nystrom-supp} satisfies
\begin{equation} \label{eqn:nystrom-error-supp}
\Expect_{\mtx{\Omega}} \norm{ \mtx{X} - \widehat{\mtx{X}} }_{*} %
	\leq \left(1 + \frac{r}{R-r-1}\right) \norm{ \mtx{X} - \lowrank{\mtx{X}}{r} }_*.
\end{equation}
If we replace $\widehat{\mtx{X}}$
with the rank-$r$ truncation $\lowrank{ \widehat{\mtx{X}} }{r}$,
the error bound~\cref{eqn:nystrom-error-supp} remains valid.
Similar results hold with high probability.
\end{fact}

The truncated Nystr{\"o}m approximations
satisfy a stronger error bound %
when the input matrix $\mtx{X}$ exhibits spectral decay.

\begin{fact}[Nystr{\"o}m sketch: Error bound II] \label{fact:nystrom-ii-supp}
Fix a psd matrix $\mtx{X} \in \Sym_n$.  Let $\mtx{S} = \mtx{X\Omega}$
where $\mtx{\Omega} \in \F^{n \times R}$ is standard normal.
For each $r < R - 1$, the Nystr{\"o}m approximation~\cref{eqn:nystrom-supp} satisfies
\begin{equation} \label{eqn:nystrom-error-ii-supp}
\Expect_{\mtx{\Omega}} \norm{ \mtx{X} - \lowrank{\widehat{\mtx{X}}}{r} }_{*} %
	\leq \norm{ \mtx{X} - \lowrank{\mtx{X}}{r} }_*
	+ \left(1 + \frac{r}{R-r-1}\right) \norm{ \mtx{X} - \lowrank{\mtx{X}}{r} }_*.
\end{equation}
If we replace $\widehat{\mtx{X}}$ with the rank-$r$ truncation
$\lowrank{ \widehat{\mtx{X}} }{r}$, the error bound~\cref{eqn:nystrom-error-ii-supp} remains valid.
Similar results hold with high probability.
\end{fact}

\begin{proof}
Combine the proofs
of~\cite[Thm.~4.2]{TYUC17:Fixed-Rank-Approximation} and~\cite[Thm.~9.3]{HMT11:Finding-Structure}.
\end{proof}

\subsection{Discussion}
\label{sec:nystrom-disc}

In practice, it is best to minimize the error attributable to sketching.
To that end, we recommend choosing the sketch size parameter $R$
as large as possible, given resource constraints,
so that we can obtain the highest-quality
Nystr{\"o}m approximation.

In some problems, e.g., \textsf{MaxCut} with eigenvector rounding,
the desired rank $r$ of the truncation is known in advance.
In this case, \cref{fact:nystrom} offers guidance about
how to select $R$ to achieve a specific error tolerance $(1 + \zeta)$
in~\cref{eqn:goal}.
For example, when $5r + 1 \leq R$, the expected Schatten 1-norm error
in the rank-$r$ approximation $\lowrank{\widehat{\mtx{X}}}{r}$
is at most $1.25 \times$ the error in the best rank-$r$
approximation of $\mtx{X}$.

When the input matrix $\mtx{X}$ has decaying eigenvalues,
the error in the truncated approximation may be far smaller
than \cref{fact:nystrom} predicts; see~\cite[Thm.~4.2]{TYUC17:Fixed-Rank-Approximation}.
This happy situation is typical when $\mtx{X}$ is generated by
the \CGAL\ iteration.

\subsection{Representation of the truncated Nystr{\"o}m approximation}
\label{sec:nystrom-rep-supp}

The key tool in the analysis is a simple representation for the truncated approximation.
These facts are extracted from~\cite[Supp.]{TYUC17:Fixed-Rank-Approximation}.

Let $\mtx{X} \in \Sym_n$ be a fixed psd matrix,
and let $\mtx{\Omega} \in \R^{n \times R}$ be an arbitrary test matrix.
Let $\mtx{P} \in \Sym_n$ be the orthoprojector onto the range of $\mtx{X}^{1/2} \mtx{\Omega}$.
Then we can write the Nystr{\"o}m approximation~\cref{eqn:nystrom-supp} as
$$
\widehat{\mtx{X}} = \mtx{X}^{1/2} \mtx{P} \mtx{X}^{1/2}.
$$
For each $r \leq R$, define $\mtx{P}_r$ to be the orthoprojector onto the co-range of
the matrix $\lowrank{ \mtx{X}^{1/2} \mtx{P} }{r}$.  By construction,
$\lowrank{ \mtx{X}^{1/2} \mtx{P} }{r} = \mtx{X}^{1/2} \mtx{P}_r$.
As a consequence,
$$
\lowrank{\widehat{\mtx{X}}}{r} = (\lowrank{\mtx{X}^{1/2} \mtx{P}}{r})(\lowrank{\mtx{X}^{1/2} \mtx{P}}{r})^*
	= \mtx{X}^{1/2} \mtx{P}_r \mtx{X}^{1/2}.
$$
These results allow us to relate the truncated approximations to each other:

\subsection{Approximation errors: Analysis}
\label{sec:approx-analysis-supp}

We may now obtain explicit formulas for the error in each Nystr{\"o}m approximation.
For $r \leq R$, note that
$$
\mtx{X} - \lowrank{\widehat{\mtx{X}}}{r} = \mtx{X}^{1/2} (\Id - \mtx{P}_r) \mtx{X}^{1/2} \psdge \mtx{0}.
$$
It follows immediately that
$$
\mathrm{err}( \lowrank{\widehat{\mtx{X}}}{r} )
	= \norm{ \mtx{X} - \lowrank{\widehat{\mtx{X}}}{r} }_*
	= \trace( \mtx{X} - \lowrank{\widehat{\mtx{X}}}{r} )
	= \trace( \mtx{X} ) - \trace( \lowrank{\widehat{\mtx{X}}}{r} ).
$$
This is the relation~\cref{eqn:sketch-err-supp}.

Assume that $r \leq r'$.  Since $\trace( \lowrank{\widehat{\mtx{X}}}{r} ) \leq \trace( \lowrank{\widehat{\mtx{X}}}{r'} )$,
we have the bound
$$
\mathrm{err}( \lowrank{\widehat{\mtx{X}}}{r} ) \geq \mathrm{err}( \lowrank{ \widehat{\mtx{X}}}{r'} )
\quad\text{for $r \leq r'$.}
$$
In other words, for fixed sketch size $R$,
the error in the truncated Nystr{\"o}m approximation
is monotone decreasing in the approximation rank.

To obtain~\cref{eqn:sketch-input-err-bd-supp}, observe that
$$
\norm{ \mtx{X} - \lowrank{\mtx{X}}{r} }_*
	\leq \norm{ \mtx{X} - \lowrank{\widehat{\mtx{X}}}{r} }_*
	= \mathrm{err}( \lowrank{\widehat{\mtx{X}}}{r} ).
$$
The inequality holds because $\lowrank{\mtx{X}}{r}$ is a best
rank-$r$ approximation of $\mtx{X}$ in Schatten 1-norm,
while $\lowrank{\widehat{\mtx{X}}}{r}$ is another rank-$r$ matrix.

\subsection{Trace correction: Analysis}
\label{sec:trace-correct-pf-supp}

Next, we study the trace-corrected Nystr{\"o}m approximation,
introduced in \cref{sec:trace-correct-supp}.  Recall that
$\mtx{X}$ is a psd matrix with trace $\alpha$ and
$\widehat{\mtx{X}}$ is a Nystr{\"o}m approximation
of $\mtx{X}$ with $\trace(\widehat{\mtx{X}}) = \widehat{\alpha}$.
Owing to the projection representation of $\widehat{\mtx{X}}$ in \cref{sec:nystrom-rep-supp},
it must be the case that $\widehat{\alpha} \leq \alpha$.
The trace-corrected approximation $\widetilde{\mtx{X}}$ is defined
in~\cref{eqn:trace-correct-supp}.
By construction, this matrix satisfies $\widetilde{\mtx{X}} \psdge \widehat{\mtx{X}}$
and that $\trace(\widetilde{\mtx{X}}) = \alpha$.

First, we develop a variational interpretation of the trace-corrected approximation:
\begin{equation} \label{eqn:tracenorm-min-supp}
\widetilde{\mtx{X}} \in \argmin \{ \norm{ \mtx{Y} - \widehat{\mtx{X}} }_* : \text{$\trace(\mtx{Y}) = \alpha$ and $\mtx{Y} \psdge \mtx{0}$} \}.
\end{equation}
Indeed, for any feasible point $\mtx{Y}$,
$$
\norm{ \mtx{Y} - \widehat{\mtx{X}} }_*
	\geq \trace( \mtx{Y} - \widehat{\mtx{X}} )
	= \alpha - \widehat{\alpha}.
$$
In the first relation, equality holds if and only if $\mtx{Y} \psdge \widehat{\mtx{X}}$.
We have already seen that the matrix $\widetilde{\mtx{X}}$
is a feasible point that satisfies the equality condition.
There are many solutions to the optimization problem~\cref{eqn:tracenorm-min-supp};
we have singled out $\widetilde{\mtx{X}}$ as the one that simultaneously minimizes
the Frobenius-norm error $\fnorm{ \mtx{Y} - \widehat{\mtx{X}} }$ over the same
feasible set.

Second, we need to argue that trace correction has a controlled impact
on the error in the Nystr{\"o}m approximation.  This point follows
from a standard calculation:
$$
\norm{ \widetilde{\mtx{X}} - \mtx{X} }_*
	\leq \norm{ \widetilde{\mtx{X}} - \widehat{\mtx{X}} }_* + \norm{ \mtx{X} - \widehat{\mtx{X}} }_*
	\leq 2 \norm{ \mtx{X} - \widehat{\mtx{X}} }_*.
$$
The first relation is the triangle inequality.
The second relation holds because $\widetilde{\mtx{X}}$
solves the variational problem~\cref{eqn:tracenorm-min-supp},
while $\mtx{X}$ is also feasible for this optimization problem.

\subsection{Statistical properties: Analysis}
\label{sec:nystrom-stat-supp}

Next, let us verify the statistical properties of the error.
In this section, the test matrix
$\mtx{\Omega} \in \F^{n \times R}$ is standard normal.

Assuming that $\rank(\mtx{X}) = r \leq R$,
let us prove that {$\norm{ \mtx{X} - \lowrank{\widehat{\mtx{X}}}{r} }_* = 0$}
with probability one.  To that end, we observe
$$
\range(\mtx{P}) = \range(\mtx{X}^{1/2} \mtx{\Omega}) = \range(\mtx{X}^{1/2})
\quad\text{with probability one.}
$$
It follows that
$$
\widehat{\mtx{X}} = \mtx{X}^{1/2} \mtx{P} \mtx{X}^{1/2} = \mtx{X}
\quad\text{with probability one.}
$$
Moreover, $\rank(\mtx{X}^{1/2} \mtx{P}) = r$ with probability one.
Conditional on this event,
$$
\lowrank{\widehat{\mtx{X}}}{r} = \mtx{X}^{1/2} \mtx{P}_r \mtx{X}^{1/2}
	= (\lowrank{ \mtx{X}^{1/2} \mtx{P} }{r} )(\lowrank{ \mtx{X}^{1/2} \mtx{P} }{r} )^*
	= (\mtx{X}^{1/2} \mtx{P}) (\mtx{X}^{1/2} \mtx{P})^* = \mtx{X}.
$$
This is the stated result.

Next, we show that the expected error in the truncated Nystr{\"o}m approximation is monotone
decreasing with respect to the sketch size $R$.  Fix the truncation rank $r$.
Let $\mtx{\Omega}_+ = \begin{bmatrix} \mtx{\Omega} & \vct{\omega} \end{bmatrix}$,
where $\vct{\omega} \in \F^n$ is a standard normal vector independent from $\mtx{\Omega}$.
Define $\mtx{P}_+ \in \Sym_n$ to be the orthoprojector
onto $\range(\mtx{X}^{1/2} \mtx{\Omega}_+)$.
It is clear that $\range(\mtx{P}) \subseteq \range(\mtx{P}_+)$,
and so
$$
\mtx{X}^{1/2} \mtx{P} \mtx{X}^{1/2} \psdle \mtx{X}^{1/2} \mtx{P}_+ \mtx{X}^{1/2}.
$$
As a consequence,
$$
\trace( \lowrank{ \mtx{X}^{1/2} \mtx{P} \mtx{X}^{1/2} }{r} )
	\leq \trace( \lowrank{ \mtx{X}^{1/2} \mtx{P}_+ \mtx{X}^{1/2} }{r} ).
$$
Equivalently,
$$
\norm{ \mtx{X} - \lowrank{ \mtx{X}^{1/2} \mtx{P} \mtx{X}^{1/2} }{r} }_*
	\geq \norm{ \mtx{X} - \lowrank{ \mtx{X}^{1/2} \mtx{P}_+ \mtx{X}^{1/2} }{r} }_*.
$$
Take the expectation with respect to $\mtx{\Omega}_+$.  The left-hand side
is the average error in the $r$-truncated Nystr{\"o}m approximation with
a standard normal sketch of size $R$.  The right-hand side is the same thing,
except the sketch has size $R+1$.  This is the required result.

\section{\sCGAL: Additional results}

This section contains some additional material about the \sCGAL\
algorithm.

\subsection{Assessing solution quality}
\label{sec:scgal-soln-quality-supp}

We can develop estimates for the quality of the \sCGAL\
solution by adapted the approach that we used for \CGAL.

To do so, we need to track the primal objective value
at the current iterate:
$$
p_t = \ip{\mtx{C}}{\mtx{X}_t}.
$$
At each iteration, we can easily update this estimate using the computed
approximate eigenvector $\vct{v}_t$:
$$
p_{t+1} = (1 - \eta_t) p_t + \eta_t \alpha \ip{ \vct{v}_t }{ \mtx{C} \vct{v}_t }.
$$
This update rule is applied with the help of the primitive \cref{eqn:primitives}\primone.

When we wish to estimate the error, say in iteration $t$, we can solve the eigenvalue subproblem
to very high accuracy:
$$
\xi_t = \vct{v}_t^* \mtx{D}_t \vct{v}_t = \min_{\norm{\vct{v}} = 1} \vct{v}^* \mtx{D}_t \vct{v}.
$$
Then, we can compute the surrogate duality gap:
$$
\begin{aligned}
g_t(\mtx{X}_t) = p_t + \ip{ \vct{y}_t + \beta_t (\vct{z}_t -\vct{b}) }{ \vct{z}_t }
	- \xi_t.
\end{aligned}
$$
This expression follows directly from the formula~\cref{eqn:cgal-dual-gap-supp}
using the loop invariant that $\vct{z}_t = \mathcal{A}\mtx{X}_t$.
We arrive at a computable error estimate:
$$
\begin{aligned}
p_t - p_{\star}
	&\leq g_t(\mtx{X}_t) - \ip{ \vct{y}_t }{ \vct{z}_t - \vct{b} } - \frac{1}{2} \beta_t \norm{\vct{z}_t - \vct{b}}^2. %
\end{aligned}
$$
This bound follows directly from~\cref{eqn:cgal-subopt-pf-supp}.

\subsection{Convergence theory}
\label{sec:scgal-converge-supp}

In this section, we establish two simple results on the convergence
properties of the \sCGAL\ algorithm.

\begin{theorem}[\sCGAL: Convergence I]
Let $\mathsf{\Psi}_{\star}$ be the solution set of the model problem~\cref{eqn:model-problem}.
For each $r < R - 1$, the iterates $\widehat{\mtx{X}}_t$ computed by \sCGAL\ (\cref{sec:sketchy-cgal-iteration,sec:loop-invariants})
satisfy
$$
\limsup_{t\to\infty} \Expect_{\mtx{\Omega}} \dist_*(\widehat{\mtx{X}}_t, \mathsf{\Psi}_{\star})
	\leq \left(1+ \frac{r}{R-r-1}\right) \cdot
	\max_{\mtx{X} \in \mathsf{\Psi}_{\star}} \norm{\mtx{X} - \lowrank{\mtx{X}}{r} }_*.
$$
The same bound holds for the truncated approximations $\lowrank{\widehat{\mtx{X}}_t}{r}$.
Here, $\dist_*$ is the nuclear-norm distance between a matrix and a set of matrices.
\end{theorem}

\begin{proof}
The implicit iterates $\mtx{X}_t$ satisfy the conclusions of~\cref{fact:cgal-converge},
so they converge toward the compact set $\mathsf{\Psi}_{\star}$.
Therefore, we can choose a sequence $\{\mtx{X}_{t\star}\} \subset \mathsf{\Psi}_{\star}$ with the property that
$\norm{\mtx{X}_t - \mtx{X}_{t\star}}_* \to 0$.
By the triangle inequality and~\cref{eqn:Xt-hat-err},
$$
\begin{aligned}
\Expect_{\mtx{\Omega}} \dist_*(\widehat{\mtx{X}}_t, \mathsf{\Psi}_{\star})
	&\leq \Expect_{\mtx{\Omega}} \norm{\widehat{\mtx{X}}_t - \mtx{X}_t}_* + \dist_*(\mtx{X}_t, \mathsf{\Psi}_{\star}) \\
	&\leq \left(1+ \frac{r}{R-r-1}\right) \cdot \norm{ \mtx{X}_t - \lowrank{\mtx{X}_t}{r} }_* + \norm{ \mtx{X}_t - \mtx{X}_{t\star} }_*.
\end{aligned}
$$
The rank-$r$ approximation error in Schatten 1-norm is 1-Lipschitz
with respect to the Schatten 1-norm (cf.~\cite[Sec.~SM2.2]{TYUC19:Streaming-Low-Rank}),
so
$$
\abs{ \norm{ \mtx{X}_t - \lowrank{\mtx{X}_t}{r} }_* - \norm{ \mtx{X}_{t\star} - \lowrank{\mtx{X}_{t\star}}{r} }_* }
	 \leq \norm{ \mtx{X}_{t} - \mtx{X}_{t\star} }_*.
$$
Therefore,
$$
\begin{aligned}
\norm{ \mtx{X}_t - \lowrank{\mtx{X}_t}{r} }_*
	& \leq \norm{ \mtx{X}_{t\star} - \lowrank{\mtx{X}_{t\star}}{r} }_*
	+ \norm{ \mtx{X}_{t} - \mtx{X}_{t\star} }_* \\
	& \leq \max_{\mtx{X} \in \mathsf{\Psi}_{\star}} \norm{ \mtx{X} - \lowrank{\mtx{X}}{r} }_*
	+ \norm{ \mtx{X}_{t} - \mtx{X}_{t\star} }_*.
\end{aligned}
$$
Combine the last two displays, and extract the superior limit. %
\end{proof}

If the implicit iterates generated by \sCGAL\ happen to converge
to a limit, we have a more precise result.

\begin{theorem}[\sCGAL: Convergence II]
Assume the implicit iterates $\mtx{X}_t$
induced by \sCGAL\ (\cref{sec:sketchy-cgal-iteration,sec:loop-invariants})
converge to a matrix $\mtx{X}_{\mathrm{cgal}}$.  For each $r < R - 1$,
the computed iterates $\widehat{\mtx{X}}_t$ satisfy
\begin{equation*} %
\begin{aligned}
\limsup_{t \to \infty} \Expect_{\mtx{\Omega}} \norm{ \mathcal{A}\widehat{\mtx{X}}_t - \vct{b} }
	&\leq \left(1+ \frac{r}{R-r-1}\right) \cdot \norm{\mathcal{A}} \cdot \norm{\mtx{X}_{\mathrm{cgal}} - \lowrank{\mtx{X}_{\mathrm{cgal}}}{r} }_*; \\
\limsup_{t \to \infty} \Expect_{\mtx{\Omega}} \abs{ \ip{\mtx{C}}{\widehat{\mtx{X}}_t} - \ip{\mtx{C}}{\mtx{X}_{\star}} }
	&\leq \left(1+ \frac{r}{R-r-1}\right) \cdot \norm{\mtx{C}} \cdot \norm{\mtx{X}_{\mathrm{cgal}} - \lowrank{\mtx{X}_{\mathrm{cgal}}}{r} }_*.
\end{aligned}
\end{equation*}
The same bound holds for the truncated approximations $\lowrank{\widehat{\mtx{X}}_t}{r}$.
If $\rank(\mtx{X}_{\mathrm{cgal}}) \leq R$, then the computed iterates $\widehat{\mtx{X}}_t$
converge to the solution set of~\cref{eqn:model-problem}.
\end{theorem}

\begin{proof}
The implicit iterates $\mtx{X}_t$ satisfy the conclusions of \cref{fact:cgal-converge},
so the limit $\mtx{X}_{\mathrm{cgal}}$ solves~\cref{eqn:model-problem}.
Using the triangle inequality, the operator norm bound, and~\cref{eqn:Xt-hat-err}, we obtain
nonasymptotic error bounds
\begin{equation*} %
\begin{aligned}
\Expect_{\mtx{\Omega}} \norm{ \mathcal{A}\widehat{\mtx{X}}_t - \vct{b} }
	&\leq \frac{\mathrm{Const}}{\sqrt{t}}
		+ \left(1+ \frac{r}{R-r-1}\right) \cdot \norm{\mathcal{A}} \cdot \norm{\mtx{X}_t - \lowrank{\mtx{X}_t}{r} }_*; \\
\Expect_{\mtx{\Omega}} \abs{\ip{\mtx{C}}{\widehat{\mtx{X}}_t} - \ip{\mtx{C}}{\mtx{X}_{\star}} }
	&\leq \frac{\mathrm{Const}}{\sqrt{t}}
		+ \left(1+ \frac{r}{R-r-1}\right) \cdot \norm{\mtx{C}} \cdot \norm{\mtx{X}_t - \lowrank{\mtx{X}_t}{r} }_*.
\end{aligned}
\end{equation*}
Extract the limit as $t \to \infty$. %
The last conclusion follows from the facts outlined in \cref{sec:nystrom-apriori-supp}.
\end{proof}

\section{Beyond the model problem}
\label{sec:extensions}

The \CGAL\ algorithm~\cite{YFC19:Conditional-Gradient-Based} applies
to a more general problem template than~\cref{eqn:model-problem}.
Likewise, the \sCGAL\ algorithm can solve a wider class of problems
in a scalable fashion.  This section outlines some of the opportunities.

\subsection{A more general template}

Consider the optimization problem %
\begin{equation} \label{eqn:general-problem}
\minimize\quad	\ip{\mtx{C}}{\mtx{X}}
\quad\subjto\quad	\mathcal{A}\mtx{X} \in \mathsf{K}
\quad\text{and}\quad
\mtx{X} \in \mathsf{X}, \quad
\text{$\mtx{X}$ is psd.}
\end{equation}
In this expression, %
$\mathsf{K} \subset \R^d$ is a closed, convex set and $\mathsf{X} \subset \Sym_n(\F)$
is a compact, convex set of matrices.  The rest of this section
describes some problems that fall within the compass of~\cref{eqn:general-problem},
as well as new computational challenges that appear.
\Cref{alg:full-sketchy-cgal-supp} contains pseudocode for a version of
\sCGAL\ tailored to~\cref{eqn:general-problem}.

\begin{remark}[Other matrix optimization problems]
We can also extend \sCGAL\ to optimization problems involving
matrices that are symmetric (but not psd) or that are rectangular.
For example, matrix completion via nuclear-norm minimization~\cite{Sre04:Learning-Matrix}
falls in this framework.  In this case, we need to replace
the Nystr{\"o}m sketch with a more general technique,
such as~\cite{TYUC17:Practical-Sketching,TYUC19:Streaming-Low-Rank}.
Further extensions are also possible; see~\cite{YFC19:Conditional-Gradient-Based}.
We omit these developments.
\end{remark}

\subsection{The convex constraint set}
\label{sec:convex-constraint-set}

To handle the convex constraint $\mathsf{X}$ that appears in~\cref{eqn:general-problem},
we must develop a subroutine for the linear minimization problem
\begin{equation} \label{eqn:lmo-general}
\underset{\mtx{H} \in \Sym_n}{\minimize}\quad
\ip{\mtx{D}_t}{ \mtx{H} }
\quad\subjto\quad \mtx{H} \in \mathsf{X}, \quad \text{$\mtx{H}$ is psd.}
\end{equation}
To implement \sCGAL\ efficiently, we need the problem~\cref{eqn:lmo-general}
to admit a structured (e.g., low-rank) approximate solution.
Here are some situations where this is possible.

\vspace{0.5pc}

\begin{enumerate}
\item	\textbf{Trace-bounded psd matrices.} $\mathsf{X} := \{ \mtx{X} \in \Sym_n : \text{$\trace \mtx{X} \leq \alpha$ and $\mtx{X}$ is psd} \}$.
For solving a standard-form SDP, this constraint is more natural than $\mathsf{X} = \alpha \mtx{\Delta}_n$. Given an (approximate) minimum eigenpair $(\xi_t, \vct{v}_t)$ of $\mtx{D}_t$,
the solution of~\cref{eqn:lmo-general} is
$$
\mtx{H}_t = \begin{cases}
	\alpha \vct{v}_t \vct{v}_t^*, & \xi_t < 0, \\
	\vct{0}, & \xi_t \geq 0.
\end{cases}
$$
As before, we can solve the eigenvector problem with \cref{alg:rand-lanczos}.

\item	\textbf{Relaxed orthoprojectors.} $\mathsf{X} := \{ \mtx{X} \in \Sym_n : \text{$\trace \mtx{X} = \alpha$ and $\mtx{0} \psdle \mtx{X} \psdle \Id$} \}$.  This is the best convex relaxation of the set of orthogonal projectors with rank $\alpha$; see~\cite{OW92:Sum-Largest}.  When $\alpha$ is small, we can provably solve the linear minimization with randomized subspace iteration~\cite{HMT11:Finding-Structure} or randomized block Lanczos methods~\cite[Sec.~10.3.6]{GVL13:Matrix-Computations-4ed}.

\end{enumerate}

\begin{remark}[Standard-form SDP] 
It is often easy to find an upper bound for the trace of a solution for an SDP. 
In many applications, the constraint $\mathcal{A}\mtx{X} = \vct{b}$ already enforces a constant trace (e.g., MaxCut or QAP).
In many other applications (e.g., phase retrieval), $\mtx{C}$ is identity matrix,
so the objective is to minimize the trace. 
In this setting, the trace of an arbitrary feasible point is an upper bound for $\alpha$. 
If neither situation is in force,
we can still solve a standard-form SDP by solving a small number of trace-bounded SDPs:
\begin{enumerate}
\item Start with an arbitrary bound $\alpha_0 > 0$, say, twice the trace of an arbitrary feasible point.
\item Solve the trace-bounded SDP with $\trace \mtx{X} \leq \alpha_i$ to obtain $X_\star(\alpha_i)$. 
\item If the primal objective $\ip{\mtx{C}}{\mtx{X}_\star(\alpha_i)} = \ip{\mtx{C}}{\mtx{X}_\star(\alpha_{i-1})}$, then terminate and return $\mtx{X}_\star(\alpha_i)$.
\item Otherwise, set $\alpha_{i+1} = 2 \alpha_i$ and return to Step~2.
\end{enumerate}
This procedure terminates in $\tilde{\mathcal{O}}(1)$ iterations if a bounded solution exists. It is easy to show by contradiction that there exists no finite $\alpha > \alpha_i$ such that $\ip{\mtx{C}}{\mtx{X}_\star(\alpha)} < \ip{\mtx{C}}{\mtx{X}_\star(\alpha_i)}$. 
Unfortunately, it is nontrivial to extend this claim to the approximate solutions because they are infeasible.
\end{remark}

\subsection{Convex inclusions}

To handle the inclusion in $\mathsf{K}$ that appears in~\cref{eqn:general-problem},
we need an efficient algorithm to perform the Euclidean projection onto $\mathsf{K}$.
That is,
$$
\mathrm{proj}_{\mathsf{K}}(\vct{w})
	:= \argmin\{ \norm{ \vct{w} - \vct{u} } : \vct{u} \in \mathsf{K} \}
	\quad\text{for $\vct{w} \in \R^{d}$.}
$$
Here are some important examples:

\vspace{0.5pc}

\begin{enumerate}
\item	\textbf{Inequality constraints.}  $\mathsf{K} := \{ \vct{u} \in \R^d : \vct{u} \leq \vct{b} \}$.
In this case, the projection takes the form
$\mathrm{proj}_{\mathsf{K}}(\vct{w}) = ( \vct{w} - \vct{b} )_-$,
where $(\cdot)_-$ reports the negative part of a vector.

\item	\textbf{Norm constraints.}  $\mathsf{K} := \{ \vct{u} \in \R^d : \triplenorm{ \vct{u} - \vct{b} } \leq \delta \}$,
where $\triplenorm{\cdot}$ is a norm.  The projector can be computed easily
for many norms, including the $\ell_p$ norm for $p \in \{1,2,\infty\}$.
\end{enumerate}

\subsection{The \CGAL\ iteration for the general template}

To extend the description of the \CGAL\ iteration in \cref{sec:cgal-iteration-supp} for the general template \cref{eqn:general-problem}, we consider the following augmented Lagrangian formulation with the slack variable $\vct{w} \in \mathsf{K}$ instead of \cref{eqn:al-supp}:
\begin{equation*} %
L_{t}(\mtx{X}; \vct{y}) :=
	\ip{\mtx{C}}{\mtx{X}} + \min_{\vct{w} \in \mathsf{K}} \left\{ \ip{\vct{y}}{\mathcal{A}\mtx{X} - \vct{w}}
	+ \frac{1}{2} \beta_t\normsq{\mathcal{A}\mtx{X} - \vct{w}} \right\}.
\end{equation*}
Accordingly, the partial derivative \cref{eqn:cgal-grad-supp} becomes
\begin{equation*}
\mtx{D}_t := \partial_{\mtx{X}} L_t(\mtx{X}_t; \vct{y}_t)
	= \mtx{C} + \mathcal{A}^*\big( \vct{y}_t + \beta_t (\mathcal{A}\mtx{X}_t - \vct{w}_t) \big) 
	\quad \text{where} \quad \vct{w}_t := \mathrm{proj}_{\mathsf{K}}( \mathcal{A}\mtx{X}_t + \beta_t^{-1} \vct{y}_t ).
\end{equation*}
We replace the linear minimization subroutine \cref{eqn:cgal-direction-supp} with \cref{eqn:lmo-general}. We can still use an inexact variant of \cref{eqn:lmo-general} with additive error. 
We also revise the dual update scheme by modifying \cref{eqn:cgal-dual-update-supp} as
\begin{equation*}
\vct{y}_{t+1} = \vct{y}_t + \gamma_t (\mathcal{A}\mtx{X}_{t+1} - \overline{\vct{w}}_t)
	\quad \text{where} \quad \overline{\vct{w}}_t := \mathrm{proj}_{\mathsf{K}}( \mathcal{A}\mtx{X}_{t+1} + \beta_{t+1}^{-1} \vct{y}_t ).
\end{equation*} 
Finally, we replace the dual step size parameter selection rule \cref{eqn:cgal-dual-step-supp} with
\begin{equation} \label{eqn:cgal-gen-dual-step-supp}
\gamma_t \normsq{\mathcal{A}\mtx{X}_{t+1} - \overline{\vct{w}}_t}
	\leq \beta_t \eta_t^2 \alpha^2 \normsq{\mathcal{A}}.
\end{equation}
The bounded travel condition \cref{eqn:cgal-bdd-travel-supp} remains the same.

To obtain the extension of \sCGAL\ to the general template,
we simply pursue the same program outlined in \cref{sec:sketchy-cgal}
to augment \CGAL\ with sketching.

\begin{algorithm}[t]%
  \caption{\sCGAL\ for the general template~\cref{eqn:general-problem}
  \label{alg:full-sketchy-cgal-supp}}
  \begin{algorithmic}[1]
  \vspace{0.5pc}

	\Require{Problem data for~\cref{eqn:general-problem} implemented via the primitives~\cref{eqn:primitives},
	sketch size $R$, number $T$ of iterations}
    \Ensure{Rank-$R$ approximate solution to~\cref{eqn:general-problem}
    in factored form $\widehat{\mtx{X}}_T = \mtx{U\Lambda U}^*$
    where $\mtx{U} \in \R^{n \times R}$ has orthonormal columns and $\mtx{\Lambda} \in \R^{R\times R}$
    is nonnegative diagonal, and the Schatten 1-norm approximation errors
    $\mathrm{err}(\lowrank{\widehat{\mtx{X}}}{r})$ for $1 \leq r \leq R$, as defined in~\cref{eqn:sketch-err-supp}}
    \Recommend{To achieve~\cref{eqn:goal}, set $R$ as large as possible, and set $T \approx \eps^{-1}$}

	\Statex

	\Function{\sCGAL}{$R$; $T$}

	\State	Scale problem data (\cref{sec:numerics-scaling})
		\Comment{\textcolor{dkblue}{\textbf{[opt]}} Recommended!}
	\State	$\beta_0 \gets 1$ and $K \gets +\infty$
		\Comment{Default parameters}
	\State	\textsf{NystromSketch.Init}($n$, $R$)
	\State	$\vct{z} \gets \vct{0}_{d}$ and $\vct{y} \gets \vct{0}_d$
	\For{$t \gets 1, 2, 3, \dots, T$}
  		\State	$\beta \gets \beta_0 \sqrt{t+1}$ and $\eta \gets 2 /(t+1)$
		\State	$\vct{w} \gets \mathrm{proj}_{\mathsf{K}}( \vct{z} + \beta^{-1} \vct{y} )$
		\State	$\mtx{D} \gets \mtx{C} + \mathcal{A}^*(\vct{y} + \beta (\vct{z} - \vct{w}))$
			\Comment Represent via primitives~\cref{eqn:primitives}\primone\primtwo
		\State	$\mtx{H}$ is a (low-rank) matrix that solves~\cref{eqn:lmo-general}
		\State	$\vct{z} \gets (1 - \eta) \, \vct{z} + \eta \, \mathcal{A}( \mtx{H} )$
			\Comment	Use primitive~\cref{eqn:primitives}\primthree
  		\State	$\beta_+ \gets \beta_0 \sqrt{t+2}$
		\State	$\vct{w} \gets \mathrm{proj}_{\mathsf{K}}( \vct{z} + \beta_+^{-1} \vct{y} )$
		\State	$\vct{y} \gets \vct{y} + \gamma (\vct{z} - \vct{w})$
			\Comment	Step size $\gamma$ satisfies~\cref{eqn:cgal-gen-dual-step-supp,eqn:cgal-bdd-travel-supp}
		\State	\textsf{NystromSketch.RankOneUpdate}($\sqrt{\alpha} \vct{v}$, $\eta$)
	\EndFor
	\State	$[\mtx{U}, \mtx{\Lambda}, \mathbf{err}] \gets \textsf{NystromSketch.Reconstruct}()$
	\EndFunction

  \vspace{0.25pc}

\end{algorithmic}
\end{algorithm}

\section{Details of MaxCut experiments}
\label{sec:maxcut-details-supp}

This section presents further details about the termination criteria that we used in the \textsf{MaxCut}
experiments presented in~\cref{sec:numerics}. 

\subsection{Comparison of SDP solvers}
\label{sec:comparison-sdp-solvers-supp}

We begin with a discussion of how we compared the performance
of different SDP solvers for the \textsf{MaxCut} problem~\cref{eqn:maxcut-sdp}.

\subsubsection{Convention for sign of the dual} 

We used the sign convention from the Lagrangian formulation \eqref{eqn:lagrangian-supp} when describing the quality guarantees of each solver, and also in the definition of \textsc{Dimacs} errors and the dual problem in \cref{sec:primal-dual-conv-supp}. This convention can be different in other works, so these definitions might have the sign of the dual variable inverted in some references.

\subsubsection{Details for \sCGAL} 
We implement the stopping criteria based on the discussion in \cref{sec:bound-subopt-supp,sec:scgal-soln-quality-supp}. 
We stop the algorithm when both
\begin{equation}
\label{eqn:scgal-stopping-supp}
\frac{p_t + \ip{ \vct{y}_t }{\vct{b} } + \frac{1}{2}\beta_t \ip{  \vct{z}_t -\vct{b} }{ \vct{z}_t + \vct{b}} - \lambda_{\min}(\mtx{D}_t)}{1 + \abs{p_t}} \leq \varepsilon
\qquad \text{and} \qquad
\frac{\norm{ \vct{z}_t - \vct{b} }}{1+\norm{\vct{b}}} \leq \varepsilon.
\end{equation}

In theory, this bound requires computing $\lambda_{\min}(\mtx{D}_t)$ to high accuracy. 
In practice, we did not observe a significant difference when using a high accuracy approximation or the inexact computation from step \cref{eqn:cgal-direction-supp}. 
We fixed the code to use the latter to avoid additional cost.

If \sCGAL terminates by \cref{eqn:scgal-stopping-supp}, then the implicit variable $\mtx{X}$ provably satisfies
$$
\frac{\ip{\mtx{C}}{\mtx{X}} - \ip{\mtx{C}}{\mtx{X}_\star}}{1 + \abs{\ip{\mtx{C}}{\mtx{X}}}} \leq \varepsilon
\qquad \text{and} \qquad
\frac{\norm{ \mathcal{A}\mtx{X} - \vct{b} }}{1+\norm{\vct{b}}} \leq \varepsilon.
$$
This is the guarantee that we seek.

\subsubsection{Details for \textsf{SDPT3}} 
We used \textsf{SDPT3} version~4.0 \cite{TTT2009:SDPT3version4} in the experiments. 
This software is designed for solving conic optimization problems by using a primal-dual interior-point algorithm. 
The algorithm iterates three variables: $\mtx{X}$, $\vct{y}$, and $\mtx{Z}$. 
For the \textsf{MaxCut} problem, $\mtx{X}$ and $\mtx{Z}$ are $n \times n$ symmetric positive semidefinite matrices, and $\vct{y} \in \R^n$.  

We can control the desired accuracy by changing the parameter \texttt{OPTIONS.gaptol}. 
When we set this parameter to $\varepsilon$, the outputs ensure
$$
\frac{\ip{\mtx{X}}{\mtx{Z}}}{1+\abs{\ip{\mtx{C}}{\mtx{X}}}+\abs{\vct{b}^\top \vct{y}}} \leq \varepsilon, 
\qquad
\frac{\norm{\mathcal{A}\mtx{X} - \vct{b}}}{1 + \norm{\vct{b}}} \leq \varepsilon, 
\qquad \text{and} \qquad
\frac{\norm{\mathcal{A}^*\vct{y} - \mtx{Z} + \mtx{C}}_F}{1 + \norm{\mtx{C}}_F} \leq \varepsilon.
$$
These are the required bounds.

\subsubsection{Details for \textsf{SeDuMi}} 
We used \textsf{SeDuMi}~1.3 \cite{Sedumi13} in the experiments. 
This software implements a self-dual embedding technique \cite{YTM94:self-dual}. 
The algorithm has two outputs. For the \textsf{MaxCut} problem, they are the $n \times n$ symmetric positive semidefinite primal solution $\mtx{X}$ and $n$ dimensional dual solution $\vct{y}$. 
We can control the desired accuracy by changing the parameter \texttt{pars.eps}. 

For some instances, when we set \texttt{pars.eps} very small, even though the algorithm achieves the desired accuracy, we observed that \textsf{SeDuMi}~1.3 overwrites the variables as zeros after solving the problem before returning them, based on a quality control procedure. 
To prevent this issue, we commented out the section between the lines $638$ and $642$ in \texttt{sedumi.m}.

\subsubsection{Details for \textsf{SDPNAL+}} 
We used \textsf{SDPNAL} version 1.0 \cite{STYZ20:SDPNAL+} which implements an augmented Lagrangian based method for solving semidefinite programs. 
When applied to the \textsc{MaxCut} SDP, the algorithm outputs three variables $\mtx{X}$, $\vct{y}$ and $\mtx{Z}$ similar to \textsf{SDPT3}. 
We can control the desired accuracy by changing the parameter \texttt{OPTIONS.tol}. 
When we set this parameter to $\varepsilon$, the outputs ensure
$$
\frac{\norm{\mathcal{A}\mtx{X} - \vct{b}}}{1 + \norm{\vct{b}}} \leq \varepsilon, 
\qquad
\frac{\norm{\mathcal{A}^*\vct{y} - \mtx{Z} + \mtx{C}}_F}{1 + \norm{\mtx{C}}_F} \leq \varepsilon,
\qquad \text{and} \qquad
\frac{\norm{\mtx{X} - \mathrm{proj}_{\mathrm{psd}}(\mtx{X} - \mtx{Z})}_F}{1 + \norm{\mtx{X}}_F + \norm{\mtx{Z}}_F} \leq 5 \varepsilon,
$$ 
where $\mathrm{proj}_{\mathrm{psd}} : \R^{n \times n} \to \Sym_n$ is the projection operator onto positive semidefinite cone.   

\subsubsection{Details for \textsf{Mosek}} 
We used the interior point optimizer for conic optimization from the \textsf{Mosek Optimization Suite} Release 8.1.0.64 \cite{mosek}. 
This optimizer implements of the so-called homogeneous and self-dual algorithm \cite{ART3:implementing-pd-ipm}. 
When applied to the \textsf{MaxCut} SDP, the algorithm returns three variables $\mtx{X}$, $\vct{y}$, and $\mtx{Z}$. 
We can control the target accuracy by changing three parameters: 
\texttt{MSK\_DPAR\_INTPNT\_CO\_TOL\_PFEAS}, 
\texttt{MSK\_DPAR\_INTPNT\_CO\_TOL\_DFEAS}, 
and \texttt{MSK\_DPAR\_INTPNT\_CO\_TOL\_RELGAP}.
For simplicity, and as suggested in the manual, we relax these parameters together and set each one to $\varepsilon$. 
This ensures
\begin{gather*}
\frac{\norm{\mathcal{A}\mtx{X} - \vct{b}}_\infty}{1 + \norm{\vct{b}}_\infty} \leq \varepsilon, 
\qquad
\frac{\norm{\mathcal{A}^*\vct{y} - \mtx{Z} + \mtx{C}}_\infty}{1 + \norm{\mtx{C}}_\infty} \leq \varepsilon, \\
\frac{\ip{\mtx{X}}{\mtx{Z}}}{\max\left\{1,\min\{\abs{\ip{\mtx{C}}{\mtx{X}}},\abs{\vct{b}^\top\vct{y}}\}\right\}} \leq \varepsilon, 
\qquad \text{and} \qquad
\frac{\abs{\ip{\mtx{C}}{\mtx{X}}+\vct{b}^\top\vct{y}}}{\max\left\{1,\min\{\abs{\ip{\mtx{C}}{\mtx{X}}},\abs{\vct{b}^\top\vct{y}}\}\right\}} \leq \varepsilon.
\end{gather*}
These are the quality guarantees that we need.

\subsection{Convergence of \sCGAL\ for the MaxCut SDP}
\label{sec:maxcut-supp}

This section gives further information about our evaluation
of the convergence behavior of the \sCGAL\ algorithm.

\cref{tab:maxcut-supp} presents numerical data from the MaxCut SDP experiment with \textsc{Gset} benchmark. 
We compare the methods in terms of the
$$
\texttt{suboptimality} = \frac{\ip{\mtx{C}}{\mtx{X}} - \ip{\mtx{C}}{\mtx{X}_{\star}}}{1 + \abs{\ip{\mtx{C}}{\mtx{X}_{\star}}}} \quad\text{and}\quad
\texttt{infeasibility} = \frac{\norm{\mathcal{A}\mtx{X} - \vct{b}}}{1 + \norm{\vct{b}}},
$$
as well as the storage cost and computation time. 
We use a high-accuracy solution from \textsf{SDPT3} with default parameters to approximate the optimal point. 
The storage cost is approximated by monitoring the virtual memory size of the process, hence it includes the memory that is swapped out and it can go beyond 16 GB. 

We also compare the weight of cut evaluated after rounding (see \cref{sec:maxcut-rounding} for details of the rounding procedure), relative to the weight obtained by rounding $\mtx{X}_{\star}$:
$$
\texttt{relative\ cut\ weight} = \frac{\texttt{cut\ weight\ from}~\mtx{X} \texttt{ - cut\ weight\ from}~\mtx{X}_{\star} }{\texttt{cut\ weight\ from}~\mtx{X}_{\star}}.
$$
Positive values indicate better cuts with a higher weight. 
The average of this quantity over all datasets in the benchmark indicates a discrepancy of less than $1.5\%$, a small price for huge scalability benefits. 

\subsection{Primal--dual convergence}
\label{sec:primal-dual-conv-supp}

This section presents empirical evidence for the convergence of
the primal variables, the  dual variables, and the surrogate duality gap
generated by \sCGAL.
We use the \textsf{MaxCut} SDP~\cref{eqn:maxcut-sdp} for these tests.

\subsubsection{\textsc{DIMACS} errors}

We evaluate the \textsc{Dimacs} errors \cite{DIMACS7} (with Euclidean scaling) to measure the suboptimality and infeasibilities for primal and dual problems. 
These measures are commonly used for benchmarking SDP solvers \cite{Mittelmann:2003aa}. 
For a given approximate solution triplet $(\mtx{X},\vct{y},\mtx{Z})$, we compute
\begin{gather*}
\texttt{err}_1 = \frac{\norm{\mathcal{A}\mtx{X} - \vct{b}}}{1 + \norm{\vct{b}}}, \qquad
\texttt{err}_2 = \frac{\max\{-\lambda_{\min}(\mtx{X}),0\}}{1 + \norm{\vct{b}}}, \qquad
\texttt{err}_3 = \frac{\norm{\mathcal{A}^*\vct{y} - \mtx{Z} + \mtx{C}}_F}{1 + \norm{\mtx{C}}_F}, \\[1em]
\texttt{err}_4 = \frac{\max\{-\lambda_{\min}(\mtx{Z}),0\}}{1 + \norm{\mtx{C}}_F}, \qquad 
\text{and} \qquad 
\texttt{err}_5 = \frac{\ip{\mtx{C}}{\mtx{X}} + \vct{b}^\top \vct{y}}{1 + \abs{\ip{\mtx{C}}{\mtx{X}}} + \abs{\vct{b}^\top \vct{y}}}. 
\end{gather*}
We also compute the $6$th error defined in \cite{Mittelmann:2003aa} to measure the violation in complementary slackness 
\begin{equation*}
\texttt{err}_6 = \frac{\ip{\mtx{X}}{\mtx{Z}}}{1 + \abs{\ip{\mtx{C}}{\mtx{X}}} + \abs{\vct{b}^\top \vct{y}}}. 
\end{equation*}
\sCGAL and \textsf{Sedumi} do not maintain the third variable $\mtx{Z}$ explicitly. 
We set $\mtx{Z} = \mtx{C} + \mathcal{A}^*\vct{y}$ as suggested in the guideline of the 7th \textsc{Dimacs} Implementation Challenge \cite{DIMACS7,Mittelmann:2003aa}. 

We use a testbed similar to the \textsf{MaxCut} experiments summarized in~\cref{sec:maxcut-numerics}. 
For ease of comparison, we replace $\trace \mtx{X} = n$ constraint with $\trace \mtx{X} \leq 1.05 \cdot n$. 
We explain the reason of this modification in the next subsection. 
We also set a stronger target accuracy by choosing $\varepsilon= 10^{-3}$ for each solver. 
\Cref{tab:maxcut-dual-supp} reports the numerical outcomes.

\subsubsection{The dual of the model problem}
The dual function for the model problem \cref{eqn:model-problem-supp} with the trace constraint is 
\begin{equation} \label{eqn:model-dual-function-supp}
\begin{aligned}
\phi(\vct{y}) := \min_{\mtx{X} \in \alpha \mtx{\Delta}_n} L(\mtx{X}, \vct{y}) & = \min_{\mtx{X} \in \alpha \mtx{\Delta}_n} \ip{\mtx{C} + \mathcal{A}^*\vct{y}}{\mtx{X}} - \ip{\vct{y}}{\vct{b}} \\
& = \alpha \cdot \lambda_{\min}(\mtx{C} + \mathcal{A}^*\vct{y}) - \ip{\vct{y}}{\vct{b}} .
\end{aligned}
\end{equation}
Therefore, the dual problem is an unconstrained nonsmooth concave maximization:
\begin{equation} \label{eqn:model-dual-problem-supp}
\underset{\vct{y}\in \R^n}{\maximize} \quad \phi(\vct{y}) := \alpha \cdot \lambda_{\min}(\mtx{C} + \mathcal{A}^*\vct{y}) - \ip{\vct{y}}{\vct{b}}.
\end{equation}
Under strong duality, $\phi(\vct{y}_\star) = p_\star = \ip{\mtx{C}}{\mtx{X}_\star}$. 

Note that \cref{eqn:model-dual-problem-supp} is slightly different from the dual of the standard-form SDP. 
In the standard-form dual, the maximization in \cref{eqn:model-dual-function-supp} takes place over the psd cone without any further constraints, and the term $\min_{\text{\{$\mtx{X}$ is psd\}}} \ip{\mtx{C} + \mathcal{A}^*\vct{y}}{\mtx{X}}$ becomes an indicator function that constraints $\mtx{C} + \mathcal{A}^*\vct{y}$ to the psd cone. We can formulate the standard-form dual problem as
\begin{equation} \label{eqn:standard-sdp-dual-problem-supp}
\maximize \quad -\ip{\vct{y}}{\vct{b}}
\quad\subjto\quad \text{$\mtx{C} + \mathcal{A}^*\vct{y}$ is psd},
\quad \vct{y} \in \R^n.
\end{equation}
In particular, the problem~\cref{eqn:standard-sdp-dual-problem-supp} is constrained,
while~\cref{eqn:model-dual-problem-supp} is unconstrained.

\textsc{Dimacs} measures are defined to evaluate convergence to a solution of \cref{eqn:standard-sdp-dual-problem-supp}, 
hence they do not fit well for evaluating \cref{eqn:model-dual-problem-supp} in general.
In our experiments, we empirically observed that $\phi(\vct{y})$ converges to $p_\star$ for the dual sequence of \sCGAL, but $\vct{y}$ does not converge to a solution of \cref{eqn:standard-sdp-dual-problem-supp}. 
However, the following lemma from our concurrent work with Ding \cite{DYC+19:ApproximateComplementarity} shows that the solution sets of the duals of the standard SDP and the trace-bounded SDP (\cref{sec:convex-constraint-set}) are the same under some technical conditions.
\begin{lemma}[Lemma~6.1 in \cite{DYC+19:ApproximateComplementarity}]
Suppose that the standard-form SDP has a solution $\mtx{X}_\star$, satisfies strong duality, and $\vct{y}_\star$ is the unique solution of the standard-form dual problem \cref{eqn:standard-sdp-dual-problem-supp}.
Consider the trace bounded SDP with $\mathsf{X} = \{ \mtx{X} \in \Sym_n : \text{$\trace \mtx{X} \leq \alpha$ and $\mtx{X}$ is psd} \}$ for some $\alpha > \trace{\mtx{X}_\star}$. 
The dual problem is 
\begin{equation}\label{eqn:trace-bounded-dual-problem-supp}
\underset{\vct{y}\in \R^n}{\mathrm{maximize}} \quad \alpha \cdot \min\{0, \lambda_{\min}(\mtx{C} + \mathcal{A}^*\vct{y}) \} - \ip{\vct{y}}{\vct{b}}.
\end{equation}
Suppose $\norm{\vct{b}} \neq 0$. Then, $\vct{y}_\star$ is the unique solution of \cref{eqn:trace-bounded-dual-problem-supp}.
\end{lemma}
We refer to \cite{DYC+19:ApproximateComplementarity} for the proof. 

\subsection{Failure of the Burer--Monteiro heuristic}
\label{sec:bm-hard}

This section presents empirical evidence that Burer--Monteiro (BM) factorization methods
cannot support storage costs better than $\Omega(n\sqrt{d})$.  Our approach
is based on the paper of Waldspurger and Waters~\cite{WW18:RankOptimality},
which proves that the BM heuristic~\cref{eqn:factorized-problem}
can produce incorrect results unless $R = \Omega(n\sqrt{d})$. 

Waldspurger provided us code that generates a random symmetric $\mtx{C} \in \Sym_n(\R)$.
The \textsf{MaxCut} SDP~\cref{eqn:maxcut-sdp} with objective $\mtx{C}$
has a unique solution, and the solution has rank $1$.
If the factorization rank $R$ satisfies $R(R+3) \leq 2n$,
then the BM formulation~\cref{eqn:factorized-problem}
has second-order critical points that are not optimal
points of the original SDP~\cref{eqn:maxcut-sdp}.
As a consequence, the Burer--Monteiro approach
is reliable only if the storage budget is $\Theta(n^{3/2})$.
In contrast, for the same problem instances,
our analysis (Theorem~\cref{thm:scgal}) shows
that \sCGAL\ succeeds with factorization rank $R = 2$
and storage budget $\Theta(n)$.

We will demonstrate numerically that Waldspurger and Waters~\cite{WW18:RankOptimality}
have identified a serious obstruction to using the
Burer--Monteiro approach.  Moreover, we will see that
\sCGAL\ resolves the issue.  See the code supplement
for scripts to reproduce these experiments. 

We use the \textsf{Manopt} software~\cite{JMLR:v15:boumal14a}
to solve the Burer--Monteiro formulation of the \textsf{MaxCut} SDP.
For $n = 100$, we drew $10$ random matrices $\mtx{C}_1, \dots, \mtx{C}_{10}$
using Waldspurger's code.
For each instance, we sweep the factorization rank $R = 2, 3, 4, \dots, 13$.
(For $R \geq 13$, we anticipate that each second-order critical point
of the Burer--Monteiro problem is a solution to the original SDP,
owing to the analysis in~\cite{BVB16:BMSmoothSDP}.)
In each experiment, we ran \textsf{Manopt} with $100$ random initializations,
and we counted the number of times the algorithm failed.
We declared failure if \textsf{Manopt} converged to a second-order stationary
point whose objective value is $10^{-3}$ larger than the true optimal value.
See~\cref{table:maxcut-rank-optimality-supp} for the statistics.

In contrast, \sCGAL\ can solve all of these instances, even when the
sketch size $R = 2$.  For these problems, we use the default parameter
choices for \sCGAL, but we do not pre-scale the data or perform tuning.
\Cref{fig:maxcut-rank-optimality-supp} compares the convergence trajectory
of \sCGAL\ and \textsf{Manopt} for one problem instance.
The difference is evident.

\section{Details of phase retrieval experiments}

This section presents further details about the phase retrieval
experiments presented in~\cref{sec:numerics}.

\subsection{Synthetic phase retrieval data}
\label{sec:synthetic-phase-supp}

This section provides additional details
on the construction of synthetic datasets
for the abstract phase retrieval SDP.

For each $n \in \{ 10^2, 10^3, \dots, 10^6 \}$, we generate $20$ independent datasets as follows. %
First, draw $\vct{\chi}_\natural \in \C^n$ from the complex standard normal distribution.
We acquire $d = 12n$ phaseless measurements \cref{eqn:phase-retrieval-measurements}
using the coded diffraction pattern model \cite{CLS15:CodedDiffractionPhaseLift}.

To do so, we randomly draw $12$ independent modulating waveforms $\vct{\psi}_j$ for $j = 1, 2, \ldots,  12$.
Each entry of $\vct{\psi}_j$ is drawn as the product of two independent random variables, one chosen uniformly from $\{ 1, i, -1, -i \}$, and the other from $\{\sqrt{2}/2,\sqrt{3}\}$ with probabilities $0.8$ and $0.2$ respectively.
Then, we modulate $\vct{\chi}_\natural$ with these waveforms and take its Fourier transform.
Each $\vct{a}_i$ corresponds to computing a single entry of this Fourier transform:
$$
\vct{a}_{(j-1)n + \ell} = \mtx{W}_n(\ell,:\, )  \, \diag^*(\vct{\psi}_j), \qquad 1 \leq j \leq 12 \quad \text{and} \quad 1 \leq \ell \leq n,
$$
where $\mtx{W}_n(\ell,:)$ is the $\ell$th row of the $n \times n$ discrete Fourier transform matrix.
We use the fast Fourier transform to implement the measurement operator.

\subsection{Fourier ptychography}
\label{sec:FP-supp}

We study a more realistic measurement setup, Fourier ptychography (FP), for the phase retrieval problem.
In this setup, $\vct{\chi}_\natural \in \C^n$ corresponds to an unknown high resolution image (vectorized) from a microscopic sample in the Fourier domain.
One cannot directly acquire a high resolution image from this sample because of the physical limitations of optical systems.
Any measurement is subject to a filter caused by the lens aperture.
We can represent this filter by a sparse matrix $\mtx{\Phi} \in \C^{m \times n}$ with $m \leq n$, each row of which has only one non-zero coefficient.
Because of this filter, we can acquire only low-resolution images, through $m$-dimensional Fourier transform.

FP enlightens the sample from $L$ different angles using a LED grid.
This lets us to to obtain $L$ different aperture matrices $\Phi_j$, $j=1,2,\ldots,L$.
Then, we acquire phaseless measurements from the sample using the following transmission matrices:
$$
\vct{a}_{(j-1)m + \ell} = \mtx{W}^*_m(\ell ,: )  \, \Phi_j, \qquad 1 \leq j \leq L \quad \text{and} \quad 1 \leq \ell \leq m.
$$
Here, $\mtx{W}^*_m(\ell,:)$ is the $\ell$th row of the conjugate transpose of discrete Fourier transform matrix.

The aim in FP is to reconstruct complex valued $\vct{\chi}_\natural$ from these phaseless measurements.
Once we construct $\vct{\chi}_\natural$, we can generate a high resolution image by taking its inverse Fourier transform.

\section{Details for QAP}
\label{sec:numerics-supp}

This section gives further information about the QAP experiments summarized in~\cref{sec:numerics-qap}.
\cref{tab:QAPLIB-convergence-supp,tab:TSPLIB-convergence-supp} display the performance of \sCGAL, by presenting the upper bound after rounding (\cref{sec:qap-round}), the objective value $\ip{\mtx{B}\otimes\mtx{A}}{\mtx{X}}$, feasibility gap, total number of iterations, memory usage, and computation time. %
Feasibility gap is defined as $\dist_{\mathsf{K}}(\mathcal{A}\mtx{X})$ where $\mathsf{K}$ is the Cartesian product between a singleton (for equality constraints) and the nonnegative orthant (for inequality constraints). It is evaluated with respect to the original (not rescaled) problem data and without normalization. 

Note that the objective value of \sCGAL is not a lower bound for the optimum since the iterates are infeasible. Due to sublinear convergence, \sCGAL might be impractical when high accuracy is required, for example within a branch and bound procedure.

\cref{tab:QAPLIB-gaps-supp,tab:TSPLIB-gaps-supp} compare the relative gap~\cref{eqn:qap-gap} obtained by \sCGAL\ with the values for the \textsf{CSDP} method \cite{BravoFerreira2018} with clique size $k = \{2,3,4\}$ and the \textsf{PATH} method \cite{ZBV09:PathFollowing} reported in \cite[Tab.~6]{BravoFerreira2018}. 
A graphical view of these results appears in \cref{fig:qap}.

\label{sec:QAP-supp}

\clearpage

\begin{figure} %
    \centering
        \includegraphics[width=\textwidth]{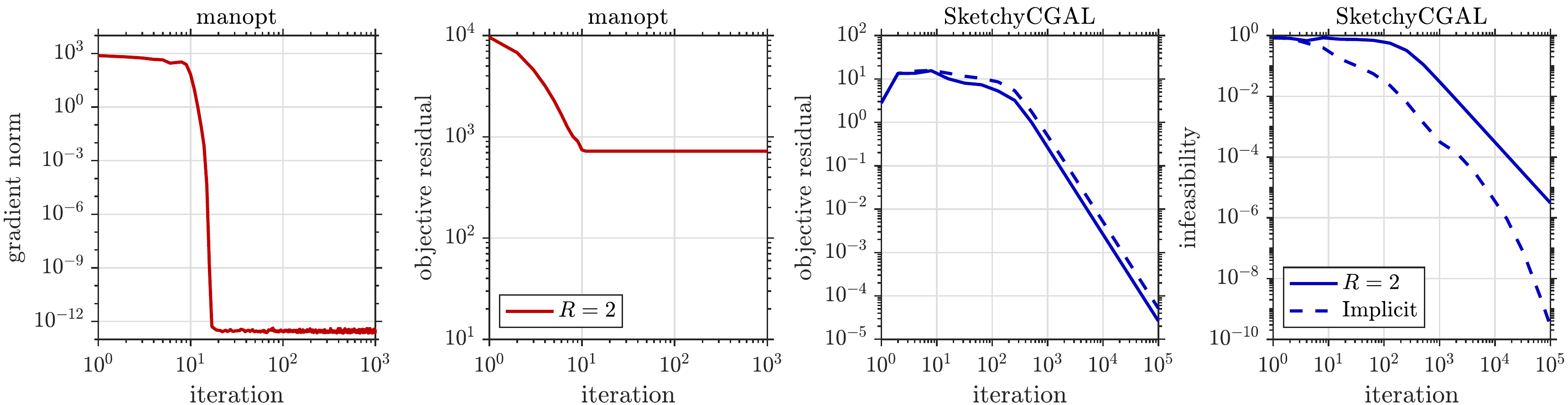}
\caption{\textsf{\textbf{MaxCut SDP: Failure of the Burer--Monteiro heuristic.}}
We apply \textsf{Manopt} and \sCGAL~for solving the \textsf{MaxCut} SDP with the dataset $\mtx{C}_1$. 
The subplots show the gradient norm and suboptimality for \textsf{Manopt} [left] and the suboptimality and infeasibility for \sCGAL [right]. 
\textsf{Manopt} with $R=2$ converges to a spurious solution, whereas \sCGAL~successfully computes a rank-$1$ approximation
of the global optimum.  The dashed line describes the convergence of the \sCGAL\ implicit iterates.   For details, see~\cref{sec:bm-hard}.
}
\label{fig:maxcut-rank-optimality-supp}
\end{figure}

\begin{table}
\label{table:maxcut-rank-optimality-supp}
\centering
\caption{We run \textsf{Manopt} for solving hard instances of the \textsf{MaxCut} SDP.  We consider $10$ datasets $\mtx{C}_1, \dots, \mtx{C}_{10}$. 
For each dataset, we run \textsf{Manopt} with 100 random initializations and report the number of failures. We declare failure when \textsf{Manopt}
converges to second-order critical point that is not a global optimum.  For details, see~\cref{sec:bm-hard}.}
\begin{tabular}{ c | r r r r r r r r r r r r r }
\toprule
  Dataset / R &    $R=2$   &    \ 3 \   &   \ 4 \   &   \ 5 \   &   \ 6 \   &   \ 7 \   &   \ 8 \   &   \ 9 \   &    10   &     11    &     12    &   13   \\
\midrule
  $\mtx{C}_1$  &    82   &     69    &    63    &    53    &    35    &    32    &    24    &    12    &    11    &     1     &     4     &      0   \\
 $\mtx{C}_2$  &      77    &     56     &    56    &     36    &     19    &     17    &  12      &    2     &     0       &   0      &     0      &     0   \\
  $\mtx{C}_3$  &     89   &      65     &    54      &   47    &     44     &    46    &     23     &    11      &    5     &     0       &    3       &    0   \\
  $\mtx{C}_4$  &     84    &     69     &    50     &    40    &     27      &   23    &     18     &    17     &     1    &      0      &     9       &    0   \\
  $\mtx{C}_5$  &     85    &     68     &    52     &    51    &     43      &   30    &     31     &    20     &    14      &    3      &     4      &     0   \\
  $\mtx{C}_6$  &     81    &     68    &     53     &    41    &     23      &   22    &     10     &    10    &      2      &    0      &     1      &     0   \\
  $\mtx{C}_7$  &     83    &     76     &    60     &    39    &     19     &    19     &    19     &     3      &    0     &     0       &    1     &      0   \\
  $\mtx{C}_8$  &     81     &    73    &     44    &     34     &    41     &    25     &     8       &  12      &    5    &      4      &    10      &     0   \\
  $\mtx{C}_9$  &     84    &     64      &   46    &     35    &     25    &     17     &     1     &    10     &     0     &     2       &    4      &     0   \\
  $\mtx{C}_{10}$  &     83    &     71      &  54     &    50      &   31     &    25     &    24    &     16      &   13     &     0      &     8      &     0   \\
\end{tabular}
\end{table}

\clearpage

\begin{figure}[p]
    \centering
    \subfloat[$t = 10\quad$ (6 min)]{\label{fig:Ptychography640-10}\includegraphics[scale=0.25]{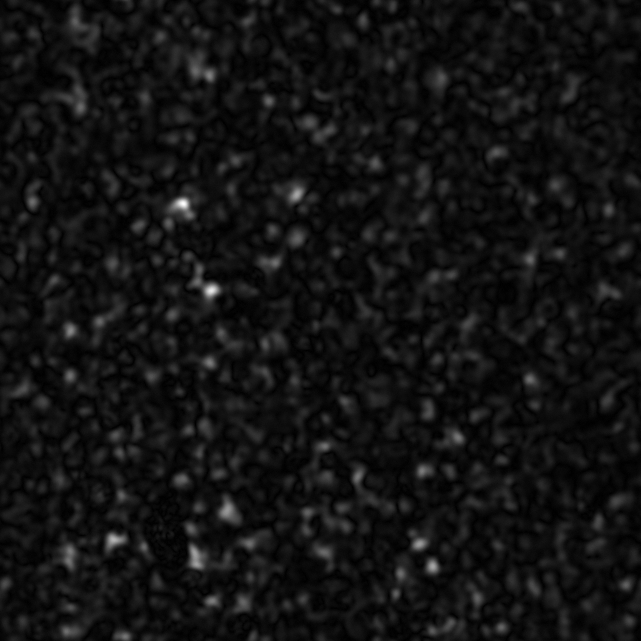} } ~
    \subfloat[$t = 100\quad$ (78 min)]{\label{fig:Ptychography640-100}\includegraphics[scale=0.25]{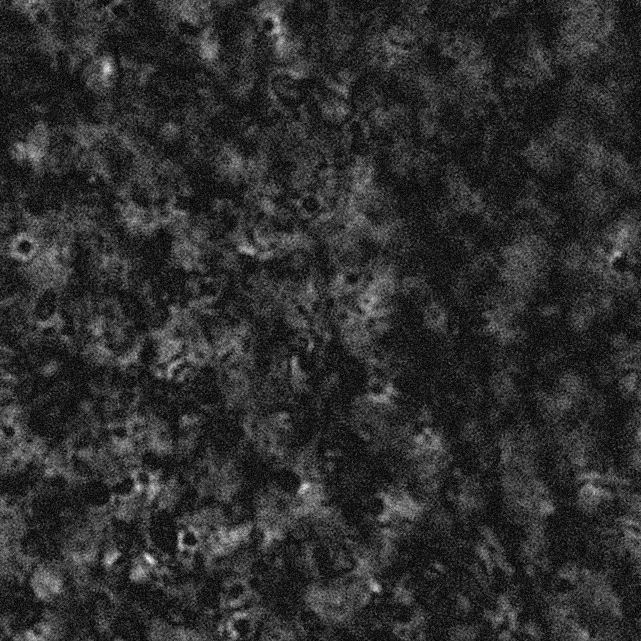} } \\
    \subfloat[$t = 1000\quad$ (1'265 min)]{\label{fig:Ptychography640-1000}\includegraphics[scale=0.25]{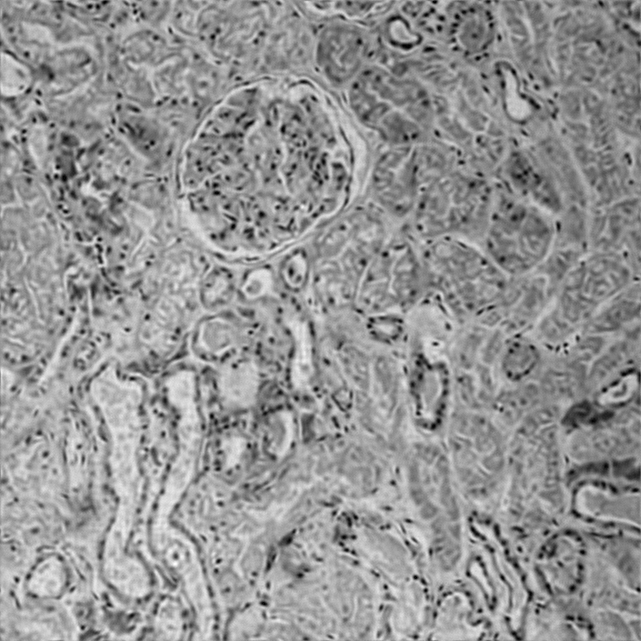} } ~
    \subfloat[$t = 5'000\quad$ (6'248 min)]{\label{fig:Ptychography640-10000}\includegraphics[scale=0.25]{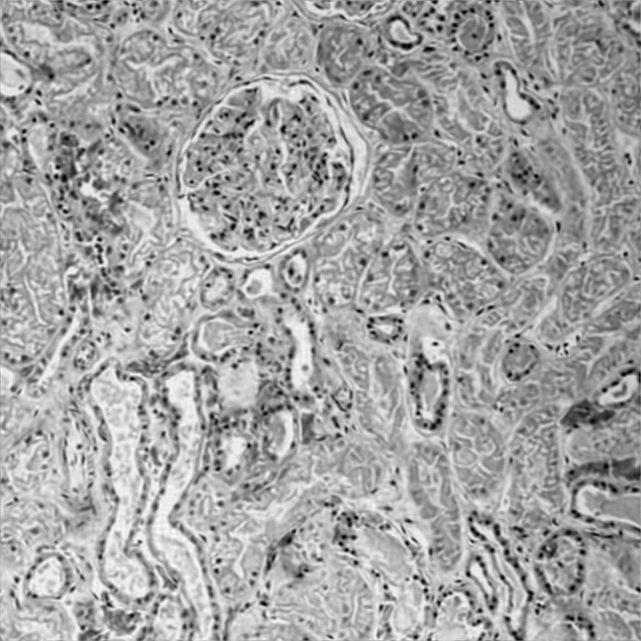} } \\
    \subfloat[original]{\fboxsep=0mm\fboxrule=1pt\label{fig:Ptychography640-original}\fcolorbox{red}{white}{\includegraphics[scale=0.25]{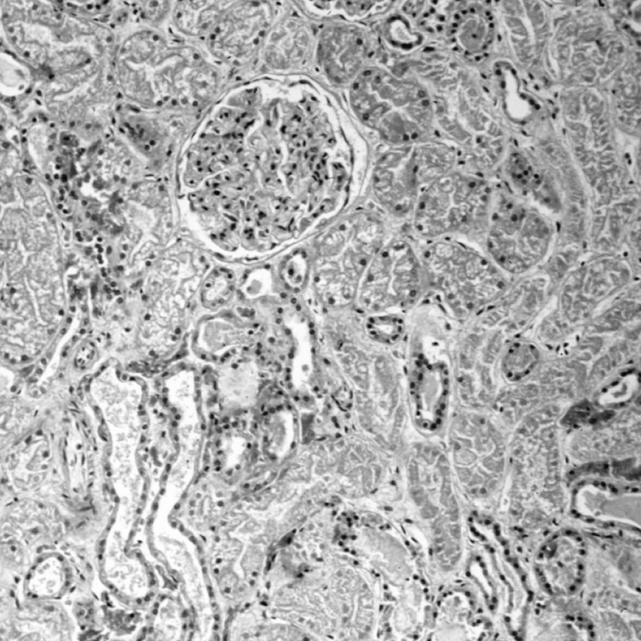}}}
    \caption{\textbf{\textsf{Phase retrieval SDP: Imaging.}}
    Reconstruction of an $n = 640^2$ pixel image from Fourier ptychography data.  We solve an $n \times n$ phase retrieval SDP via \sCGAL\ with rank parameter $R = 5$
    and show the images obtained at iterations $t = 10, 10^2, 10^3, 5 \cdot 10^3$.  The last subfigure is the original.  See~\cref{sec:FP}.}
\label{fig:Ptychography640}
\end{figure}

\clearpage

\begin{footnotesize}
\LTcapwidth=\textwidth
\setlength\LTleft{0pt}
\setlength\LTright{0pt}

\end{footnotesize}

\clearpage
\begin{table}
\centering
\begin{footnotesize}
\caption{We solve SDP relaxations of QAP instances from QAPLIB using \sCGAL. We compute the relative gap and compare it with the values for the \textsf{CSDP} method \cite{BravoFerreira2018} with clique size $k = \{2,3,4\}$ and the \textsf{PATH} method \cite{ZBV09:PathFollowing} reported in \cite[Tab.~4]{BravoFerreira2018}. Smaller is better. See~\cref{sec:numerics-qap}.}
\label{tab:QAPLIB-gaps-supp} 
\begin{tabular}{c | c | c c c c c  }
\toprule
\multicolumn{1}{p{2cm}}{\centering Dataset} & \multicolumn{1}{p{2cm}}{\centering Optimum} & \multicolumn{1}{p{2cm}}{\centering \sCGAL} & \multicolumn{1}{p{1.5cm}}{\centering \textsf{CSDP2}}& \multicolumn{1}{p{1.5cm}}{\centering \textsf{CSDP3}}&\multicolumn{1}{p{1.5cm}}{ \centering \textsf{CSDP4} } & \multicolumn{1}{p{1.5cm}}{ \centering \textsf{PATH}} \\
\midrule 
\midrule
\begin{filecontents*}{tables/QapTableGapComparison.csv}
Row,Optimal,CGAL,CSDP2,CSDP3,CSDP4,PATH
chr12a,9552,0,34.5,6,0,42.7
chr12b,9742,0,38.9,25.4,11.9,38.1
chr12c,11156,0,5.8,2.3,2.3,18.6
chr15a,9896,0.4,2.1,2.1,2.1,52
chr15b,7990,0,26.3,34.5,29.2,158.6
chr15c,9504,0,0,0,0,63.3
chr18a,11098,1.5,69.8,0.2,0.2,76.3
chr18b,1534,8.7,8.9,22.9,29.5,99.3
chr20a,2192,6.6,122.5,76.1,43.8,95.4
chr20b,2298,0,62.9,9.3,9.3,82.2
chr20c,14142,24.7,173,100.1,111.5,88.9
chr22a,6156,0,17.2,7.6,3,38.3
chr22b,6194,7.1,7.3,2.3,1,40.4
chr25a,3796,21.8,107,49.2,25.2,69.9
esc16a,68,2.9,8.8,11.8,11.8,11.8
esc16b,292,0,0,0.7,0,2.7
esc16c,160,5,5,7.5,8.7,6.3
esc16d,16,12.5,12.5,50,25,75
esc16e,28,0,14.3,7.1,14.3,21.4
esc16g,26,7.7,7.7,0,15.4,15.4
esc16h,996,0,1.6,0,1.6,16.9
esc16i,14,0,0,0,0,57.1
esc16j,8,0,0,0,0,75
esc32a,130,52.3,115.4,124.6,113.8,93.8
esc32b,168,57.1,109.5,114.3,111.9,88.1
esc32c,642,2.2,12.8,15.9,13.7,7.8
esc32d,200,14,38,36,39,21
esc32e,2,0,0,0,0,600
esc32g,6,0,0,0,0,366.7
esc32h,438,5,24.7,26.9,22.8,18.3
esc64a,116,6.9,60.3,53.4,60.3,106.9
esc128,64,28.1,250,206.3,175,221.9
ste36a,9526,15.6,70.2,74.7,74.2,76.3
ste36b,15852,15.7,188.8,204.3,211.9,158.6
ste36c,8239110,9,66,62.8,63.7,83.2
\end{filecontents*}
\csvreader[late after line=\\,head to column names]{tables/QapTableGapComparison.csv}{}
{\texttt{\csvcoli}
	& $\csvcolii$ 
		& $\csvcoliii$
 			& $\csvcoliv$
 				&  $\csvcolv$
 					&  $\csvcolvi$
 						&  $\csvcolvii$}
\end{tabular}
\end{footnotesize}
\end{table}

\begin{table}
\centering
\caption{We run \sCGAL\ for (the first of) $10^6$ iterations or $72$ hours of runtime, for solving SDP relaxations of QAP instances from QAPLIB. We report the upper bound after rounding (\cref{sec:qap-round}), the objective value, the feasibility gap, the number of iterations, the memory usage (in MB), and the cpu time (`hh:mm:ss').  See~\cref{sec:numerics-qap}.}
\label{tab:QAPLIB-convergence-supp} 
\begin{footnotesize}
\nprounddigits{2}
\begin{tabular}{c | c | c c c c c c }
\toprule
\multicolumn{1}{p{1.25cm}}{\centering Dataset} & \multicolumn{1}{p{1.25cm}}{\centering Optimum} & \multicolumn{1}{p{1.75cm}}{\centering Upper bnd.} & \multicolumn{1}{p{1.75cm}}{\centering Objective}& \multicolumn{1}{p{1.5cm}}{\centering Feas. gap}&\multicolumn{1}{p{1.5cm}}{ \centering Iteration} & \multicolumn{1}{p{1.25cm}}{ \centering Memory} & \multicolumn{1}{p{1.75cm}}{ \centering Time} \\
\end{tabular}
\npnoround
\end{footnotesize}
\end{table}

\clearpage
\begin{table}
\centering
\begin{footnotesize}
\caption{We solve SDP relaxations of QAP instances from TSPLIB using \sCGAL. We compute the relative gap and compare it with the values for the \textsf{CSDP} method \cite{BravoFerreira2018} with clique size $k = \{2,3,4\}$ and the \textsf{PATH} method \cite{ZBV09:PathFollowing} reported in \cite[Tab.~6]{BravoFerreira2018}. Smaller is better.}
\label{tab:TSPLIB-gaps-supp} 
\begin{tabular}{ c | c | c c c c c  }
\toprule
\multicolumn{1}{p{2cm}}{\centering Dataset} & \multicolumn{1}{p{2cm}}{\centering Optimum} & \multicolumn{1}{p{1.5cm}}{\centering \sCGAL} & \multicolumn{1}{p{1.5cm}}{\centering \textsf{CSDP2}}& \multicolumn{1}{p{1.5cm}}{\centering \textsf{CSDP3}}&\multicolumn{1}{p{1.5cm}}{ \centering \textsf{CSDP4}} & \multicolumn{1}{p{1.5cm}}{ \centering \textsf{PATH}} \\
\end{tabular}
\end{footnotesize}
\end{table}

\begin{table}
\centering
\caption{We run \sCGAL\ for (the first of) $10^6$ iterations or $72$ hours of runtime, for solving SDP relaxations of QAP instances from TSPLIB. We report the upper bound after rounding (\cref{sec:qap-round}), the objective value, the feasibility gap, the number of iterations, the memory usage (in MB), and the cpu time (`hh:mm:ss').  See~\cref{sec:numerics-qap}.}
\label{tab:TSPLIB-convergence-supp} 
\begin{footnotesize}
\nprounddigits{2}
%
\end{footnotesize}

\begin{figure}[p]
    \centering
    \subfloat{\label{fig:G67-dimacs-err1}\includegraphics[scale=0.55]{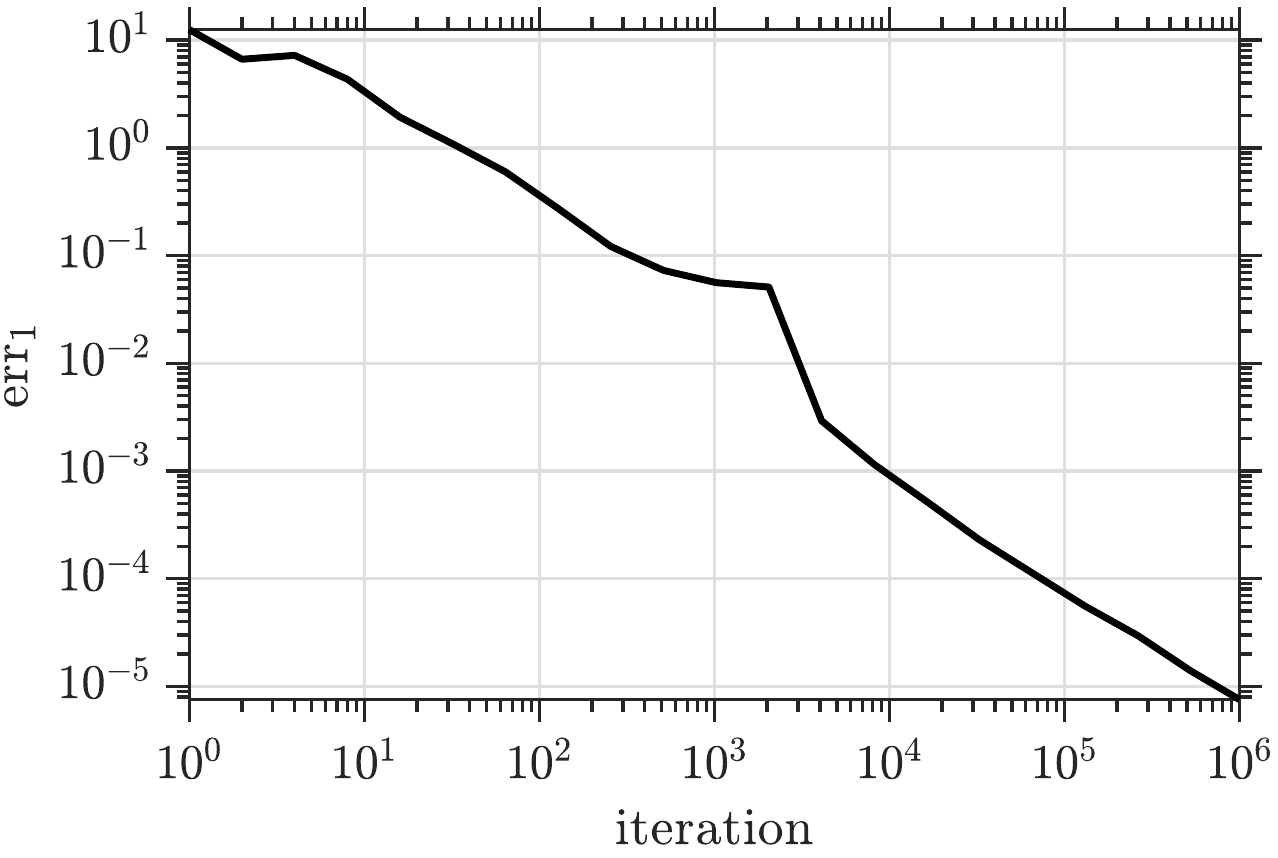} } ~
    \subfloat{\label{fig:G67-dimacs-err4}\includegraphics[scale=0.55]{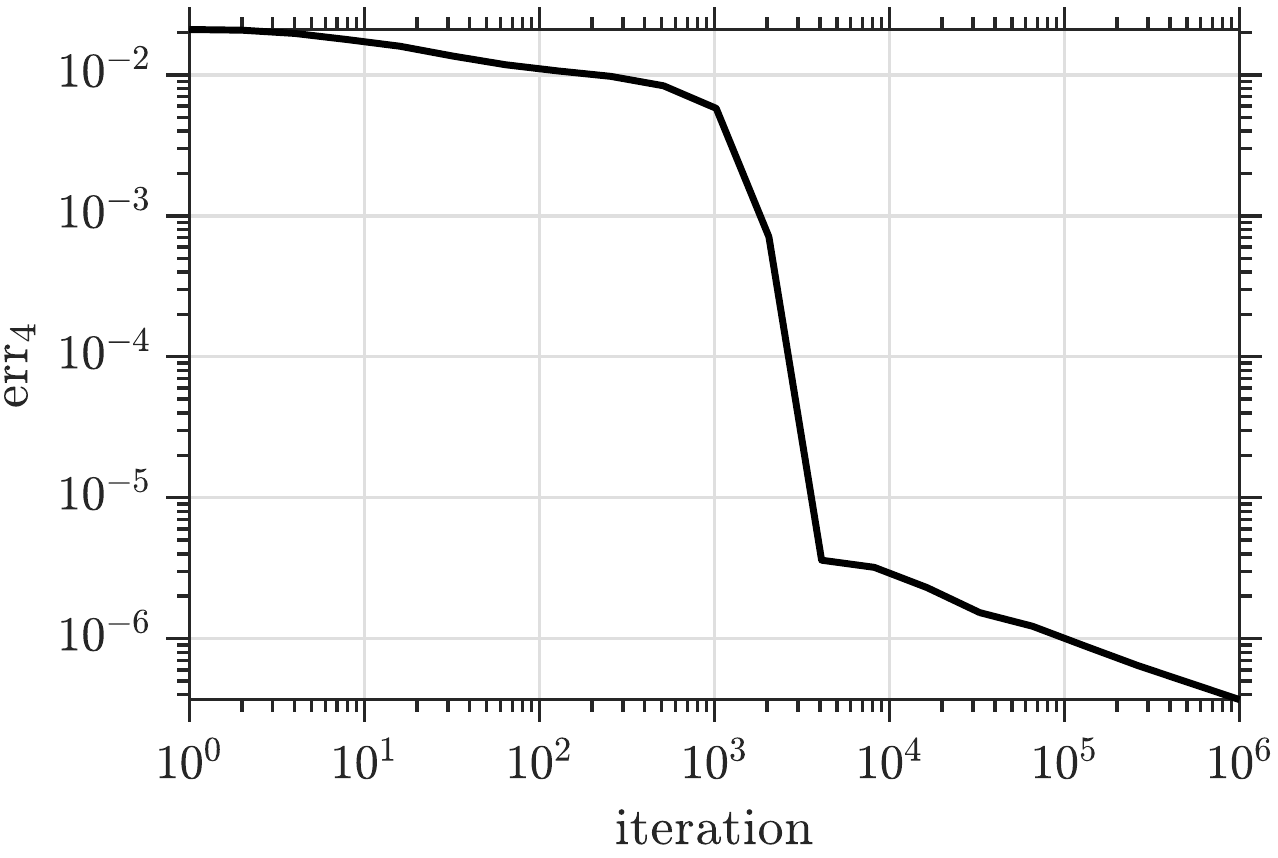} } \\
    \subfloat{\label{fig:G67-dimacs-err5}\includegraphics[scale=0.55]{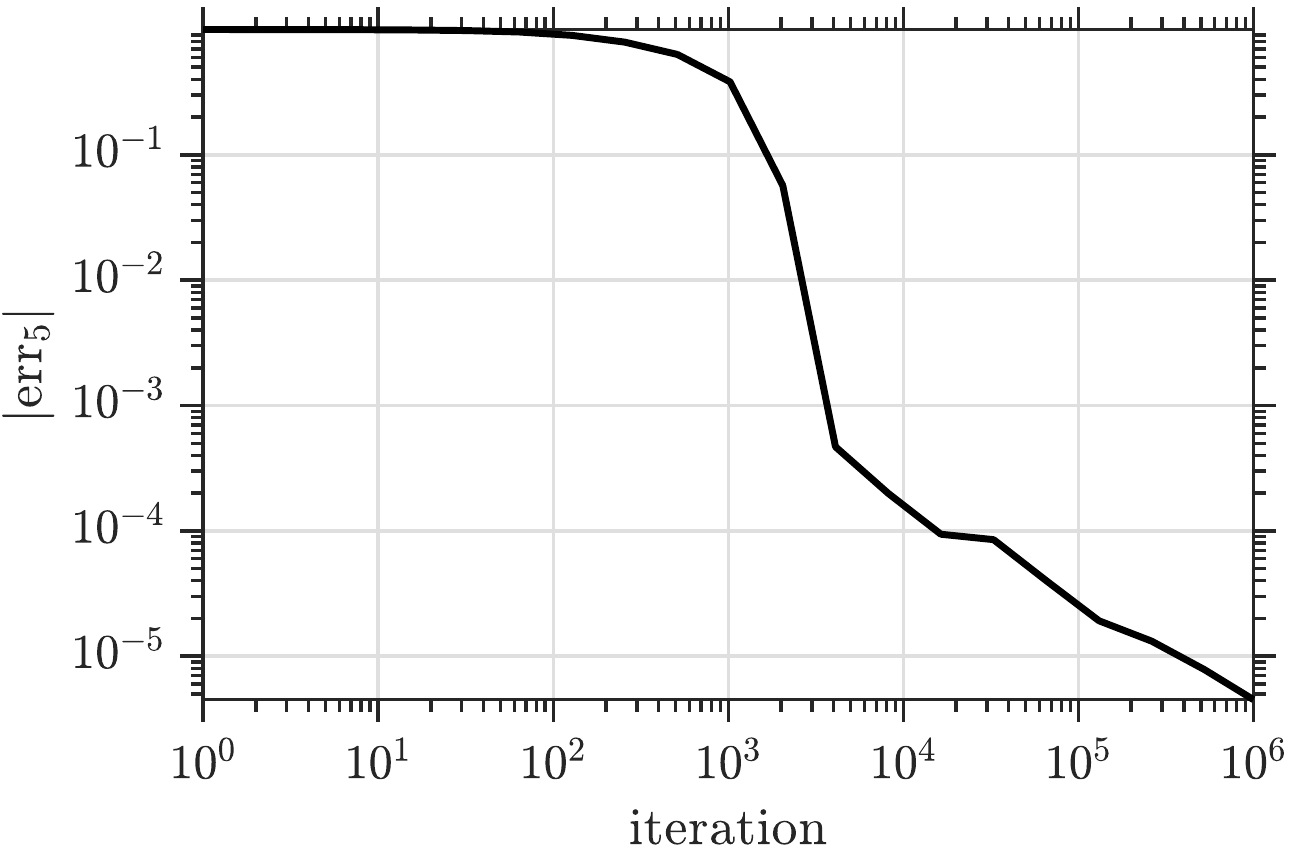} } ~
    \subfloat{\label{fig:G67-dimacs-err6}\includegraphics[scale=0.55]{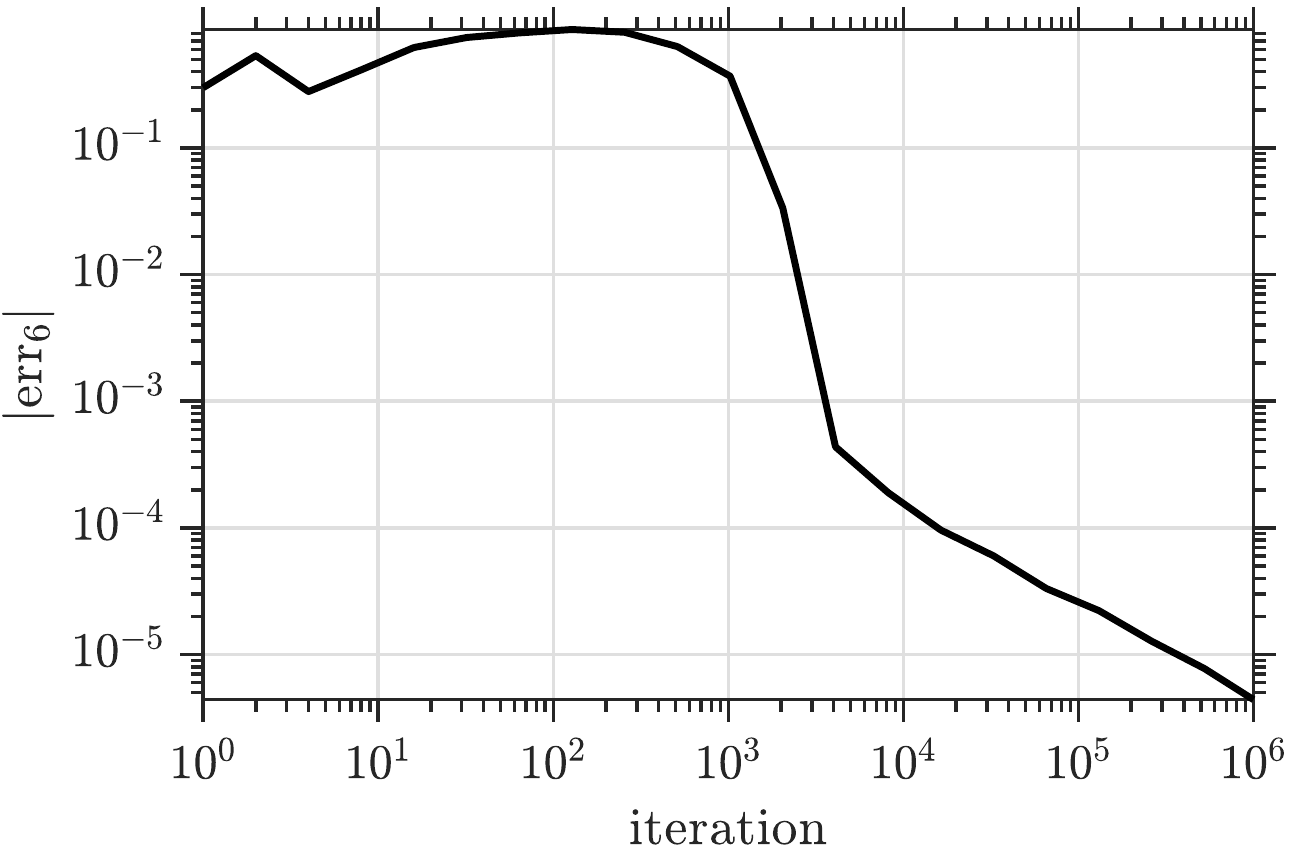} } \\
    \subfloat{\label{fig:G67-surrogate-gap}\includegraphics[scale=0.55]{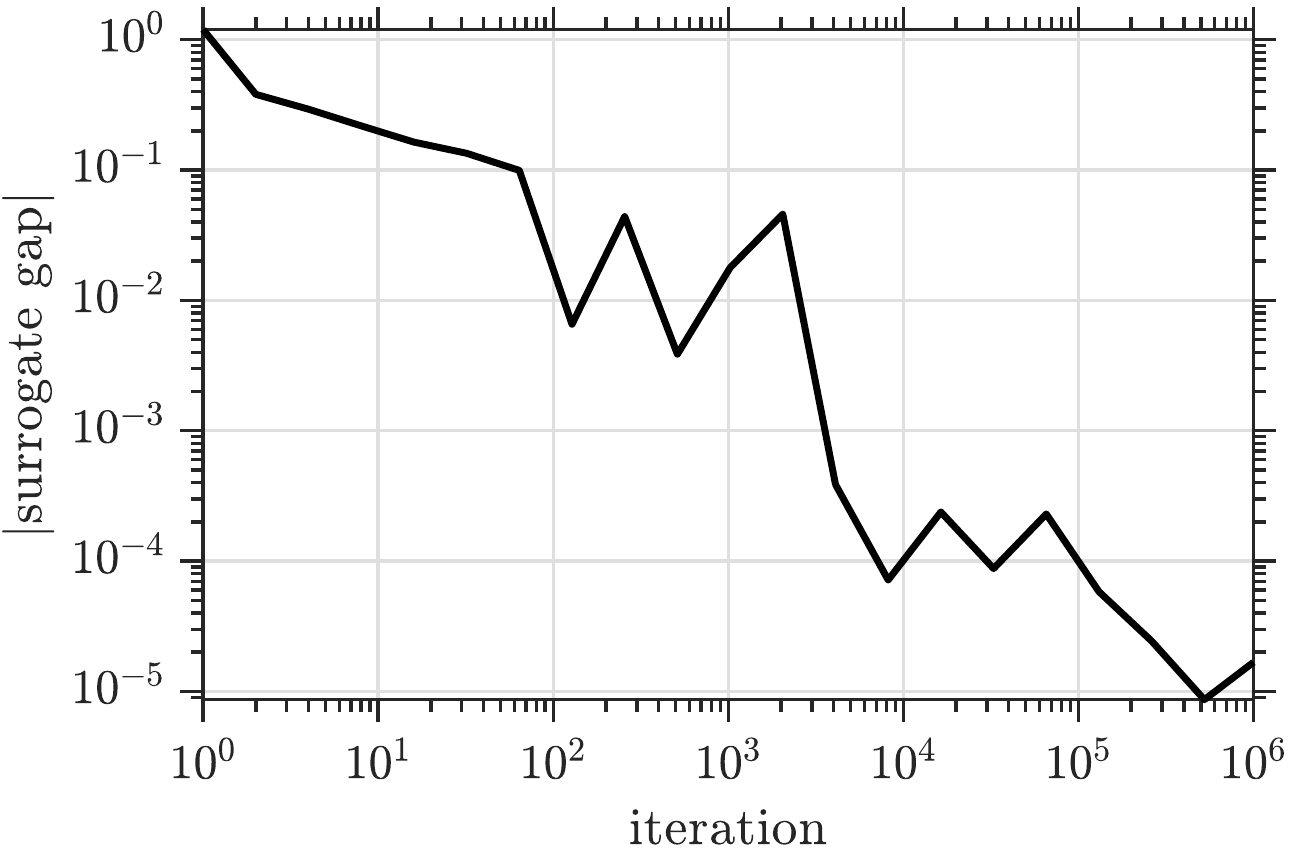}}
\caption{\textsf{\textbf{\textsf{MaxCut} SDP: Primal--dual convergence.}}
We solve the \textsf{MaxCut} SDP for the \textsf{G67} dataset ($n = 10\,000$) with \sCGAL. The subplots show the magnitudes of the \textsc{Dimacs} errors and the surrogate gap of the implicit primal iterate and the dual iterate. We omit $\mathrm{err}_2$ and $\mathrm{err}_3$ since they are zero by construction. See~\cref{sec:primal-dual-conv-supp}.}
\label{fig:maxcut-dimacs-supp}
\end{figure}

\clearpage

\section*{Acknowledgments}

Nicolas Boumal encouraged us to reconsider the storage-optimal
randomized Lanczos method (\cref{alg:rand-lanczos}); this algorithm
simplifies the presentation while improving our theoretical 
and numerical results.
Richard Kueng allowed us to include his results on
trace normalization (Step 13 in \cref{alg:sketchy-cgal}).
Ir{\`e}ne Waldspurger provided code that generates
instances of the \textsf{MaxCut} problem that are difficult for the Burer--Monteiro heuristic.
We would also like to thank the editor and reviewers for their feedback.

\bibliographystyle{siamplain}

\end{document}